\documentclass[oneside, a4paper,11pt,reqno]{amsart}
\usepackage{yfonts}

\RequirePackage[colorlinks,citecolor=blue,urlcolor=blue]{hyperref}

\makeatletter
\def\timenow{\@tempcnta\time
  \@tempcntb\@tempcnta
  \divide\@tempcntb60
  \ifnum10>\@tempcntb0\fi\number\@tempcntb
  \multiply\@tempcntb60
  \advance\@tempcnta-\@tempcntb
  :\ifnum10>\@tempcnta0\fi\number\@tempcnta}
\makeatother

\usepackage[ddmmyyyy,hhmmss]{datetime}
\usepackage{color}

\usepackage{euscript}

\usepackage[applemac]{inputenc}
\usepackage{amssymb,amsmath,amsthm,dsfont}
\usepackage{graphics,graphicx}
\usepackage[T1]{fontenc}
\usepackage{color}
\usepackage{epsfig,psfrag}

\newtheorem{theo}{Theorem}[section]
\newtheorem{prop}[theo]{Proposition}
\newtheorem{lemme}[theo]{Lemma}

\newtheorem{conj}[theo]{Conjecture}

\newtheorem{remarque}[theo]{Remark}
\newtheorem{fact}[theo]{Fact}

\def\tA2{\tilde{A}({\tilde L_2})}
\def \tts2{\tilde \tau^+_2(  h_t/2 ))}

\title{Almost sure behavior for the local time of a diffusion in a spectrally negative L\'evy environment}

\author{Gr\'egoire V\'echambre}
\thanks{NYU-ECNU Institute of Mathematical Sciences at NYU Shanghai, 3663 Zhongshan Road North, Shanghai, 200062, China}
\thanks{E-mail address: ghv2@nyu.edu}
\address{NYU-ECNU Institute of Mathematical Sciences at NYU Shanghai, 3663 Zhongshan Road North, Shanghai, 200062, China}

\email{ghv2@nyu.edu}

\subjclass[2010]{60K37, 60J55, 60G51}
\keywords{Diffusion, random potential, L\'evy process, renewal process, local time, L\'evy processes conditioned to stay positive, exponential functionals.}

\begin{document}

\maketitle

\begin{abstract}
We study the almost sure asymptotic behavior of the supremum of the local time for a transient diffusion in a spectrally negative L\'evy environment. More precisely, we provide the proper renormalizations for the extremely large and the extremely small values of the supremum of the local time. In particular, in the sub-ballistic regime, we link the almost sure behavior of the supremum of the local time with the left asymptotic tail of the exponential functional of the environment conditioned to stay positive, which allows to use the properties of such exponential functionals to characterize the sought behavior. It appears from our results that the renormalization of the extremely large values of the supremum of the local time in the sub-ballistic regime is determined by the asymptotic behavior of the Laplace exponent of the L\'evy environment, and is surprisingly greater than the renormalization that was previously known for the recurrent case. For the renormalization of the extremely small values of the supremum of the local time in the sub-ballistic regime, our results show that there is only one possible renormalization, whatever is the choice of the L\'evy environment. 
\end{abstract}

%\tableofcontents

\pagestyle{myheadings}
\markboth{Right}{Almost sure behavior for the local time of a diffusion in a spectrally negative L\'evy environment}

\section{Introduction}

We study the almost sure asymptotic behavior of the supremum of the local time for a transient diffusion in a spectrally negative L\'evy environment. Let $(V(x), \ x \in \mathbb{R})$ be a spectrally negative L\'evy process on $\mathbb{R}$ which is not the opposite of a subordinator (in particular, $V$ can not be a compound Poisson process), drifts to $-\infty$ at $+\infty$, and such that $V(0)=0$. We denote its Laplace exponent by $\Psi_V$: 
\[ \forall t, \lambda \geq 0, \ \mathbb{E} \left [ e^{\lambda V(t)} \right ] = e^{t \Psi_V(\lambda)}. \]
It is well-known, for such $V$, that $\Psi_V$ admits a non trivial zero that we denote here by $\kappa$, $\kappa := \inf \{ \lambda > 0, \ \Psi_V(\lambda) = 0 \} > 0$. 

We are here interested in a diffusion in this potential $V$. Such a diffusion $(X(t),\ t\geq 0)$ is defined informally by $X(0)=0$ and
\[ \text{d}X(t)= \text{d}\beta(t)-\frac{1}{2}V'(X(t))\text{d}t, \]
where $\beta$ is a Brownian motion independent from $V$. Rigorously, $X$ is defined by its conditional generator given $V$, 
\begin{align*} & \frac{1}{2}e^{V(x)}\frac{\text{d}}{\text{d} x}\left(e^{-V(x)}\frac{\text{d}}{\text{d} x}\right). 
\end{align*}

{Diffusions in random potentials have been introduced, in the case of a brownian potential, by Schumacher \cite{Schumacher} and Brox \cite{Brox}. They are generally considered as the continuous time analogue of \textit{Random Walks in Random Environments} (RWRE), which have many applications in physics and biology (see for example Le Doussal and al. \cite{DoMoFi}). We refer to R\'ev\'esz \cite{Revesz} and Zeitouni \cite{Zeitouni} for general properties of RWRE and to Shi \cite{Shi1} for the embedding of a RWRE into a diffusion in random potential. In the case of a drifted brownian potential, the diffusion $X$ has first been studied by Kawazu, Tanaka \cite{KawazuTanaka} who exhibit three different regimes. These results have then been refined by Kawazu, Tanaka \cite{kawazu1998}, Tanaka \cite{Tanaka19971807}, and Hu, Shi and Yor \cite{HuShYo} who, in the ballistic case, exhibit several possible behaviors and prove theorems of type \textit{central limit}. Large deviations results are obtained in Taleb \cite{Taleb0}, Talet \cite{talet2007}, 
%, and Devulder [20], 
and moderate deviations are given by Hu, Shi \cite{hu2004} in the recurrent case, and by Faraud \cite{Faraud2011} in the transient case. A result on localization and an aging phenomenon are proved in Andreoletti, Devulder \cite{AndDev} for the case of a drifted Brownian potential in the sub-ballistic regime. The diffusion has been studied in the case of a L\'evy potential by Carmona \cite{Carmona}, Cheliotis \cite{Cheliotis2006715} and Singh \cite{Singh}, \cite{Singh2007101}, \cite{Singh2007}. 
%In particular, in \cite{Singh}, the results of Kawazu, Tanaka \cite{KawazuTanaka} are generalized to the case where the potential is a spectrally negative L\'evy process. 
Let us also mention that diffusions in random potentials have been studied in dimension higher than $1$ by Tanaka \cite{tanaka1993} and Mathieu \cite{Mathieu}, \cite{Mathieu2}.}

In our case of a spectrally negative L\'evy potential $V$ drifting to $-\infty$, the fact that $V$ drifts to $- \infty$ puts us in the case where the diffusion $X$ is a.s. transient to the right. The asymptotic behavior of this diffusion has been studied by Singh \cite{Singh}, he distinguishes three main possible behaviors depending on $0 < \kappa < 1$, $\kappa = 1$ or $\kappa > 1$ (the case $\kappa > 1$ being also divided into three sub-cases). We denote by $(\mathcal{L}_X(t, x), t\geq0, x \in \mathbb{R})$ the version of the local time that is continuous in time and c\`ad-l\`ag in space, and we define the supremum of the local time until instant $t$ as
\[ \mathcal{L}_X^*(t) = \sup_{x \in \mathbb{R}} \mathcal{L}_X(t, x). \]
Here, we study the almost sure asymptotic behavior of $\mathcal{L}_X^*(t)$. {The analogous problems for RWRE have attracted much attention, and have been studied, for example, in R\'ev\'esz \cite{Revesz}, Shi \cite{Shi}, Gantert and al. \cite{GanPerShi}, \cite{GanShi}, Hu and al. \cite{HuShi1}, Dembo and al. \cite{DemGanPerShi} and Andreoletti \cite{Pierre3}, \cite{Pierre2}. The local time of processes in random environment plays an important role in the problem of estimating the potential (Comets and al. \cite{Comets_etal}, 
%Adelman and Enriqez \cite{AdeEnr}, NON CAR N'UTILISE PAS LE TL
Andreoletti \cite{Pierre4}), in the study of persistence (see Devulder \cite{devulder2016}) and is useful for the study of related models like random walks in stratified random environment (see Kochler \cite{Kochler}) or processes in random scenery (see Zindy \cite{Zindy2006UpperLO}). When the environment is a Brownian motion with no drift (in this case the diffusion is recurrent), Shi \cite{Shi} has studied the behavior of $\mathcal{L}_X^*(t)$ and shown that}
\begin{eqnarray}
\mathbb{P} \text{-a.s.} \ \limsup_{t \rightarrow +\infty} \frac{\mathcal{L}_X^*(t)}{t \log (\log (\log (t)))} \geq \frac{1}{32}, \label{ltcuriosity}
\end{eqnarray}
where $\mathbb{P}$ is the so-called annealed probability measure whose definition is recalled in Subsection \ref{factnot}. In the same case, Andreoletti and Diel \cite{AndDiel} have proved more recently the convergence in distribution of $\mathcal{L}_X^*(t) /t$ under $\mathbb{P}$. Diel \cite{Diel} has then continued the study by giving a finite upper bound for the $\limsup$ in \eqref{ltcuriosity} and doing the same study for the $\liminf$: 
\begin{eqnarray}
\mathbb{P} \text{-a.s.} \ \limsup_{t \rightarrow +\infty} \frac{\mathcal{L}_X^*(t)}{t \log (\log (\log (t)))} \leq \frac{e^2}{2} \ \ \ \text{and} \ \ \ \frac{j_0^2}{64} \leq \liminf_{t \rightarrow +\infty} \frac{\mathcal{L}_X^*(t)}{t / \log (\log (\log (t)))} \leq \frac{e^2 \pi^2}{4}, \label{resultsdiel}
\end{eqnarray}
where $j_0$ is the smallest positive root of the Bessel function $J_0$. 

For a drifted Brownian environment (in this case the diffusion is transient), {the convergence in distribution of $\mathcal{L}_X^*(t)/t$ under $\mathbb{P}$ has been established by Andreoletti, Devulder and V\'echambre \cite{advech} in the sub-ballistic case $0 < \kappa < 1$ and by Devulder \cite{Devmaxloc} in the case $\kappa \geq 1$. Still in the case of a drifted Brownian environment, the almost sure behavior of $\mathcal{L}_X^*(t)$ has been studied via annealed methods in \cite{Devmaxloc}, where this behavior is totally characterized in the ballistic case $\kappa > 1$.} 
%\begin{align}
%\text{If} \ \kappa=1, \ \mathcal{L}^*_X(t) / t^{1/\kappa} & \overset{\mathcal{L}}{\underset{t \rightarrow + \infty}{\longrightarrow}} \mathcal{F}(1, 1/2), \label{cvloikappa=1} \\
%\text{if} \ \kappa>1, \ \mathcal{L}^*_X(t) / t^{1/\kappa} & \overset{\mathcal{L}}{\underset{t \rightarrow + \infty}{\longrightarrow}} \mathcal{F}(\kappa, 4 (\kappa^2 (\kappa - 1)/8)^{1/\kappa}), \label{cvloikappa>1}
%\end{align}
%where, for $\alpha, s>0$, $\mathcal{F}(\alpha, s)$ is the Fr\'echet distribution with parameters $\alpha$ and $s$, c'est-\`a-dire, la loi de fonction de repartition 
%\begin{eqnarray}
%\mathcal{F}(\alpha, s) \left ( [0, t] \right ) = e^{-(s/t)^{\alpha}}. \label{deffrechet}
%\end{eqnarray}
In this case, for any positive non-decreasing function $a$ we have
\begin{eqnarray}
\sum_{n=1}^{+\infty} \frac{1}{n a(n)} \left\{
\begin{aligned}
& < +\infty \\
& = +\infty \end{aligned} \right. 
\Leftrightarrow  \limsup_{t \rightarrow +\infty} \frac{\mathcal{L}^*_X(t)}{(t a(t))^{1/\kappa}} = \left\{
\begin{aligned}
& 0 \\
& +\infty \end{aligned} \right. 
\mathbb{P} \text{-a.s.}, \label{devlimsupkappa>1}
\end{eqnarray}
and 
\begin{eqnarray}
\mathbb{P} \text{-a.s.} \ \liminf_{t \rightarrow +\infty} \frac{\mathcal{L}^*_X(t)}{(t/ \log (\log(t)))^{1/\kappa}} = 4 (\kappa^2 (\kappa - 1)/8)^{1/\kappa}. \label{devliminfkappa>1} 
\end{eqnarray}

When $\kappa = 1$ he obtains
\[ \mathbb{P} \text{-a.s.} \ \liminf_{t \rightarrow +\infty} \frac{\mathcal{L}^*_X(t)}{t/ \log(t) \log (\log(t))} \leq 1/2. \]

In the sub-ballistic case $0 < \kappa < 1$, his method fails and only provides partial results for the almost sure behavior of the local time. More precisely he proves that the renormalization for the $\limsup$ is greater than $t$: 
\begin{eqnarray}
\mathbb{P} \text{-a.s.} \ \limsup_{t \rightarrow +\infty} \frac{\mathcal{L}_X^*(t)}{t} = +\infty, \label{devpartiellimsup} 
\end{eqnarray}
and that the renormalization for the $\liminf$ is at most $t / \log (\log(t))$ and greater than 

\noindent $t/ (\log(t))^{1/\kappa} (\log (\log(t)))^{(2/\kappa) + \epsilon}$ for any $\epsilon > 0$: 
\begin{align}
& \mathbb{P} \text{-a.s.} \ \liminf_{t \rightarrow +\infty} \frac{\mathcal{L}_X^*(t)}{t / \log (\log(t))} \leq C(\kappa), \label{devpartielliminf1} \\
\forall \epsilon > 0, \ & \mathbb{P} \text{-a.s.} \ \liminf_{t \rightarrow +\infty} \frac{\mathcal{L}^*_X(t)}{t/ (\log(t))^{1/\kappa} (\log (\log(t)))^{(2/\kappa) + \epsilon} } = +\infty, \label{devpartielliminf2}
\end{align}
where $C(\kappa)$ is a positive non explicit constant. The main motivation of the present paper is to provide new methods that allow to determine precisely the almost sure asymptotic behavior of the local time in the case $0 < \kappa < 1$ for general spectrally negative L\'evy environments. 

%Using different methods, we study the general case of a spectrally negative L\'evy environment and we characterizes the almost sure behavior of $\mathcal{L}_X^*(t)$ when $0 < \kappa < 1$ and $\kappa > 1$. 
For the (discrete) transient RWRE, the almost sure behavior of the supremum of the local time has been studied by Gantert and Shi \cite{GanShi}. {They obtain the behavior of the $\limsup$ in the two sub-cases $0 < \kappa \leq 1$ and $\kappa > 1$ but no result has been established so far for the $\liminf$.} 

{For the recurrent diffusion in L\'evy potentials, the convergence in distribution of $\mathcal{L}_X^*(t) /t$ under $\mathbb{P}$ has been established by Diel and Voisin \cite{Dielvois} in the case of a stable L\'evy environment. For the transient diffusion in a general spectrally negative L\'evy potential} $V$, the convergence in distribution of $\mathcal{L}_X^*(t) /t$ under $\mathbb{P}$ has been established by the author in \cite{caslevyvech} using two different methods: when $0 < \kappa < 1$ a path decomposition of the environment that provides an interesting renewal structure to study the diffusion is used, this method is inspired from Andreoletti and al. \cite{advech}, \cite{AndDev} which are themselves inspired from the work of Enriquez and al. \cite{EnSaZi} in the discrete case. When $\kappa > 1$ an equality in law between the local time and a generalized Ornstein-Uhlenbeck process that was introduced in \cite{Singh} is used. The main contribution of the present paper is in the case $0 < \kappa < 1$. {We make an optimized use of the renewal structure of the diffusion, working sometimes in a quenched setting, in order to relate the almost sure asymptotic behavior of $\mathcal{L}_X^*(t)$ with the asymptotic behavior of some functionals of the renewal structure (this requires to overcome the difficulties caused by the "last valley" as explained in Section \ref{sketch} below), and to establish precisely the asymptotic behavior of these functionals}. 
%and related it to the exponential functional of the environment conditioned to stay positive
We characterize the asymptotic almost sure behavior of $\mathcal{L}_X^*(t)$ when $0 < \kappa < 1$ and $\kappa > 1$. In particular, the restriction of our results to the case of a drifted Brownian potential with $0 < \kappa < 1$ considerably improves the results \eqref{devpartiellimsup}, \eqref{devpartielliminf1} and \eqref{devpartielliminf2} of \cite{Devmaxloc} by giving the exact renormalizations for the $\limsup$ and the $\liminf$, and even the exact value of the $\limsup$. We also get an explicit upper bound for the value of the $\liminf$. 

\subsection{Main results} \label{results}

We start with the case $0 < \kappa < 1$. In that case, the limit distribution of $\mathcal{L}_X^*(t)/t$ under $\mathbb{P}$ is given in Theorem 1.3 of \cite{caslevyvech} and it depends on exponential functionals of $V$ and its dual conditioned to stay positive. They are defined as follow: 
\[ I(V^{\uparrow}) := \int_0^{+ \infty} e^{- V^{\uparrow} (t)}dt \ \ \ \text{and} \ \ \ I(\hat V^{\uparrow}) := \int_0^{+ \infty} e^{- \hat V^{\uparrow} (t)}dt, \]
where $\hat V$ denotes the dual of $V$ (it is equal in law to $-V$), and where $V^{\uparrow}$ and $\hat V^{\uparrow}$ denote respectively $V$ and $\hat V$ conditioned to stay positive. In Subsection \ref{toutsurlenv} it is recalled how $V^{\uparrow}$ and $\hat V^{\uparrow}$ are defined rigorously. These functionals are studied by the author in \cite{foncexpovech} where it is proved in Theorems 1.1 and 1.13 that they are indeed finite and well-defined. Let $G_1$ and $G_2$ be two independent random variables with $G_1 \overset{\mathcal{L}}{=}  I(V^{\uparrow})$ and $G_2 \overset{\mathcal{L}}{=}  I(\hat V^{\uparrow})$. We define $\mathcal{R} := G_1 + G_2$. 
%(if $V$ if the $\kappa$-drifted Brownian motion, then note that $V^{\uparrow} \overset{\mathcal{L}}{=}  \hat V^{\uparrow}$, so in this case $\mathcal{R}$ is the sum of two independent copies of $I(V^{\uparrow})$). 

To study the $\limsup$ of $\mathcal{L}_X^*(t)$, we link the almost sur asymptotic behavior of $\mathcal{L}^*_X(t)$ with the left tail of $I(V^{\uparrow})$ (or of $\mathcal{R}$ depending on the case). Before stating our results, let us recall what is known about the left tail of $I(V^{\uparrow})$. 

In \cite{foncexpovech}, the left tail of $I(V^{\uparrow})$ is linked to the asymptotic behavior of $\Psi_{V}$. This asymptotic behavior is usually quantified thanks to two real numbers $\sigma$ and $\beta$: 
\begin{align*} \sigma & := \sup \left \{ \alpha \geq 0, \ \lim_{\lambda \rightarrow + \infty} \lambda^{- \alpha} \Psi_{V}(\lambda) = \infty \right \}, \\
\beta & := \inf \left \{ \alpha \geq 0, \ \lim_{\lambda \rightarrow + \infty} \lambda^{- \alpha} \Psi_{V}(\lambda) = 0 \right \}. \end{align*} 
If $\Psi_{V}$ has $\alpha$-regular variation for $\alpha \in [1, 2]$ (for example if $V$ is a drifted $\alpha$-stable L\'evy process with no positive jumps), we have $\sigma = \beta = \alpha$. Note that when $Q$, the Brownian component of $V$, is positive, then $\Psi_{V}$ has $2$-regular variation, and when $Q=0$, $1 \leq \sigma \leq  \beta \leq 2$. The asymptotic behavior of $\mathbb{P} (I(V^{\uparrow}) \leq x) $ as $x$ goes to $0$ is given by the following results from \cite{foncexpovech}: 

\begin{theo} \label{rappelsurqueues} [V\'echambre, \cite{foncexpovech}]

There is a positive constant $K_0$ (depending on $V$) such that for $x$ small enough 
\begin{eqnarray}
\mathbb{P} \left ( I(V^{\uparrow}) \leq x \right ) \leq e^{-K_0/x}. \label{majouniv}
\end{eqnarray}

More precisely we have
\begin{eqnarray}
\forall l < \frac1{\beta - 1}, \ \mathbb{P} \left ( I(V^{\uparrow}) \leq x \right ) \leq e^{-1/x^l}. \label{encadgene1}
\end{eqnarray}

\begin{eqnarray}
\text{If} \ \sigma > 1, \  \forall l > \frac1{\sigma - 1}, \ \mathbb{P} \left ( I(V^{\uparrow}) \leq x \right ) \geq e^{-1/x^l}. \label{encadgene2}
\end{eqnarray}

If there are two positive constants $c < C$ and $\alpha \in ]1, 2]$ such that $c \lambda^{\alpha} \leq \Psi_V(\lambda) \leq C \lambda^{\alpha}$ for $\lambda$ large enough, then for any $\gamma > 1$ we have, when $x$ is small enough, 
\begin{eqnarray}
\exp \left ( -\frac{\gamma \alpha^{\frac{\alpha}{\alpha-1}}}{(cx)^{\frac1{\alpha-1}}} \right ) \leq \mathbb{P} \left ( I(V^{\uparrow}) \leq x \right ) \leq \exp \left ( -\frac{\alpha - 1}{\gamma (C x)^{\frac1{\alpha-1}}} \right ). \label{encadreg}
\end{eqnarray}
If there is a positive constant $C$ and $\alpha \in ]1, 2]$ such that $\Psi_V(\lambda) \sim_{\lambda \rightarrow +\infty} C \lambda^{\alpha}$, then 
\begin{eqnarray}
- \log \left ( \mathbb{P} \left ( I(V^{\uparrow}) \leq x \right ) \right ) \underset{x \rightarrow 0}{\sim} \frac{\alpha - 1}{(Cx)^{\frac1{\alpha-1}}}. \label{exactereg}
\end{eqnarray}

\end{theo}

%\begin{remarque} \label{reffoctexpo}
\eqref{majouniv} is Remark 1.7 of \cite{foncexpovech}, \eqref{encadgene1} and \eqref{encadgene2} are a reformulation of Theorem 1.4 of \cite{foncexpovech}, \eqref{encadreg} comes from Theorem 1.2 of \cite{foncexpovech}, and \eqref{exactereg} is Corollary 1.6 of \cite{foncexpovech}. 
%\end{remarque}

We can now state our results for the $\limsup$: 

\begin{theo} \label{limsupenfctdelefttail}

\begin{itemize}
\item Assume that $0 < \kappa < 1$, $V$ has unbounded variation, $V(1) \in L^p$ for some $p>1$ and $V$ possesses negative jumps. 
%we have two cases. 
\medbreak
%In the case where $V$ possesses negative jumps: 
If there is $\gamma > 1$ and $C > 0$ such that for $x$ small enough 
\begin{eqnarray}
\mathbb{P} \left ( I(V^{\uparrow}) \leq x \right ) \leq \exp \left ( -\frac{C}{x^{\frac1{\gamma-1}}} \right ), \label{condfinie}
\end{eqnarray}
then we have 
\begin{eqnarray}
\mathbb{P} \text{-a.s.} \ \limsup_{t \rightarrow +\infty} \frac{\mathcal{L}^*_X(t)}{t (\log (\log(t)))^{\gamma - 1}} \leq C^{1-\gamma}. \label{limsupfinie}
\end{eqnarray}
If there is $\gamma > 1$ and $C > 0$ such that for $x$ small enough 
\begin{eqnarray}
\mathbb{P} \left ( I(V^{\uparrow}) \leq x \right ) \geq \exp \left ( -\frac{C}{x^{\frac1{\gamma-1}}} \right ), \label{condfnonnulle}
\end{eqnarray}
then we have 
\begin{eqnarray}
\mathbb{P} \text{-a.s.} \ \limsup_{t \rightarrow +\infty} \frac{\mathcal{L}^*_X(t)}{t (\log (\log(t)))^{\gamma - 1}} \geq C^{1-\gamma}. \label{limsuppos}
\end{eqnarray}

\item Assume now that $V(t) = W_{\kappa}(t) := W(t) - \frac{\kappa}{2} t$ with $0 < \kappa < 1$, (i.e. $V$ is the $\kappa$-drifted Brownian motion), then the above implications (\eqref{condfinie} $\Rightarrow$ \eqref{limsupfinie} and \eqref{condfnonnulle} $\Rightarrow$ \eqref{limsuppos}) are still true with $I(V^{\uparrow})$ replaced by $\mathcal{R}$. 
\end{itemize}

\end{theo}

%\begin{remarque} \label{lackofsymmetry}
In the above theorem we had to distinguish the case where $V$ possesses negative jumps and the case where $V$ is a drifted Brownian motion. First, note that this is a true alternative: since $V$ is spectrally negative, the case where $V$ does not possess negative jumps is the case where $V$ does not possess jumps at all and is therefore a drifted Brownian motion. In this case we assumed for convenience that the gaussian component of $V$ is normalized to $1$. The difference between the two cases in the above theorem comes from the absence or presence of symmetry for the environment. When $V$ possesses negative jumps, $\hat V^{\uparrow}$ possesses positives jumps that may repulse it very fast from $0$, this is why the left tail of $I(V^{\uparrow})$ is thiner than the left tail of $I(\hat V^{\uparrow})$, as we can see from the results of \cite{foncexpovech} (comparing Remark 1.7 of \cite{foncexpovech} with the combination of Theorems 1.15 and 1.16 of \cite{foncexpovech}). As a consequence, only the left tail of $I(V^{\uparrow})$ is relevant in the left tail of $\mathcal{R}$. When $V$ is a drifted Brownian motion, a symmetry appears: $\hat V^{\uparrow}$ and $V^{\uparrow}$ are equal in law, $\mathcal{R}$ is then the sum of two independent random variables having the same law as $I(V^{\uparrow})$ and none of them can be neglected. 
%\end{remarque}

\begin{remarque} \label{limsupetinfcte}
It has to be noted that the $\limsup$ above is $\mathbb{P}$-almost surely equal to a constant belonging to $[0, +\infty]$ and that the inequalities \eqref{limsupfinie} and \eqref{limsuppos} are inequalities relative to this constant that the $\limsup$ equals $\mathbb{P}$-almost surely. The same will be true in all the results below: all the $\limsup$ and $\liminf$ considered are $\mathbb{P}$-almost surely equal to constants. This fact is justified in Subsection \ref{cteness}. 
\end{remarque}

%Thanks to Theorems 1.5 and 1.6 in \cite{foncexpovech}, the left tail of $I(V^{\uparrow})$ is linked to the asymptotic behavior of $\Psi_{V}$. Let us recall how this asymptotic behavior is usually quantified: 
%\begin{align*} \sigma & := \sup \left \{ \alpha \geq 0, \ \lim_{\lambda \rightarrow + \infty} \lambda^{- \alpha} \Psi_{V}(\lambda) = \infty \right \}, \\
%\beta & := \inf \left \{ \alpha \geq 0, \ \lim_{\lambda \rightarrow + \infty} \lambda^{- \alpha} \Psi_{V}(\lambda) = 0 \right \}. \end{align*} 
%If $\Psi_{V}$ has $\alpha$-regular variation for $\alpha \in [1, 2]$ (for example if $V$ is a drifted $\alpha$-stable L\'evy process with no positive jumps), we have $\sigma = \beta = \alpha$. Note that when $Q$, the Brownian component of $V$, is positive, then $\Psi_{V}$ has $2$-regular variation, and when $Q=0$ we have $1 \leq \sigma \leq  \beta \leq 2$. 
Putting together Theorem \ref{limsupenfctdelefttail} and what is known for the left tail of $I(V^{\uparrow})$ (Theorem \ref{rappelsurqueues}), we can state precise results for the $\limsup$:

\begin{theo} \label{limsupkappa<1}
If $0 < \kappa < 1$, $V$ has unbounded variation and $V(1) \in L^p$ for some $p>1$, then we have 
\begin{eqnarray}
\forall \beta' > \beta, \ \mathbb{P} \text{-a.s.} \ \limsup_{t \rightarrow +\infty} \frac{\mathcal{L}^*_X(t)}{t (\log (\log(t)))^{\beta' - 1}} =0, \label{thlimsupnulle1}
\end{eqnarray}
and
\begin{eqnarray}
\text{If} \ \sigma > 1, \ \forall \sigma' \in ]1, \sigma[, \ \mathbb{P} \text{-a.s.} \ \limsup_{t \rightarrow +\infty} \frac{\mathcal{L}^*_X(t)}{t (\log (\log(t)))^{\sigma' - 1}} =+\infty. \label{thlimsupnulle1}
\end{eqnarray}
\end{theo}

If we make further hypothesis on the regularity of the variation of $\Psi_V$ we can give the exact order of $\mathcal{L}^*_X(t)$: 
\begin{theo} \label{limsupkappa<1exact}
\begin{itemize}
\item Assume that $0 < \kappa < 1$, $V$ has unbounded variation, $V(1) \in L^p$ for some $p>1$ and $V$ possesses negative jumps. 
\medbreak
If there are two positive constants $c < C$ and $\alpha \in ]1, 2]$ such that $c \lambda^{\alpha} \leq \Psi_V(\lambda) \leq C \lambda^{\alpha}$ for $\lambda$ large enough, then we have 
\[ \mathbb{P} \text{-a.s.} \ \frac{c}{\alpha^{\alpha}} \leq \limsup_{t \rightarrow +\infty} \frac{\mathcal{L}^*_X(t)}{t (\log (\log(t)))^{\alpha - 1}} \leq \frac{C}{(\alpha - 1)^{\alpha - 1}}. \]
If, more precisely, there is a positive constant $C$ and $\alpha \in ]1, 2]$ such that $\Psi_V(\lambda) \sim C \lambda^{\alpha}$ for large $\lambda$, then we have  
\[ \mathbb{P} \text{-a.s.} \ \limsup_{t \rightarrow +\infty} \frac{\mathcal{L}^*_X(t)}{t (\log (\log(t)))^{\alpha - 1}} = \frac{C}{(\alpha - 1)^{\alpha - 1}}. \]
\item Assume now that $V = W_{\kappa}$, the $\kappa$-drifted Brownian motion, with $0 < \kappa < 1$, then we have 
\[ \mathbb{P} \text{-a.s.} \ \limsup_{t \rightarrow +\infty} \frac{\mathcal{L}^*_X(t)}{t (\log (\log(t)))} = \frac1{8}. \]
\end{itemize}
\end{theo}

%\begin{remarque} SI ON PEUT
%This Theorem contains the case where $V$ is a drifted $\alpha$-stable L\'evy process with no positive jumps. (AND THE DRIFT IS SMALL)
%\end{remarque}

\begin{remarque}
According to the combination of Theorem \ref{limsupenfctdelefttail} and \eqref{majouniv} we see that, if $0 < \kappa < 1$, $V$ has unbounded variation and $V(1) \in L^p$ for some $p>1$, then we always have
\[ \mathbb{P} \text{-a.s.} \ \limsup_{t \rightarrow +\infty} \frac{\mathcal{L}^*_X(t)}{t (\log (\log(t)))} < +\infty. \]
In other words, $t (\log (\log(t)))$ is the maximal possible renormalization for the $\limsup$. 
\end{remarque}

\begin{remarque} \label{diffavecdiscret}
As it was noticed by Shi \cite{Shi} and Diel \cite{Diel} for the recurrent case, we also notice a difference between the renormalization of the local time for the discrete transient RWRE with zero speed (given by Gantert and Shi \cite{GanShi}) and the renormalization of the local time that we give for the transient diffusion with zero speed. This difference can be explained as in the recurrent case: the valleys can potentially be much steeper in the continuous case with a potential having unbounded variation than in the discrete case, so the local maxima of the local time can potentially be higher in the first case. 
\end{remarque}

We see in the above two theorems that the renormalization of $\mathcal{L}^*_X(t)$ for the $\limsup$ depends directly on the asymptotic behavior of $\Psi_{V}$. In particular, Theorem \ref{limsupkappa<1exact} says that for a drifted $\alpha$-stable environment (with $\alpha > 1$, no positive jumps, and the drift is small negative), the renormalization of $\mathcal{L}^*_X(t)$ is $t (\log (\log(t)))^{\alpha - 1}$. We thus see that with general spectrally negative L\'evy environments, we have many possible behaviors for $\mathcal{L}^*_X(t)$, whereas with drifted Brownian environments the only possible renormalization for the $\limsup$ is $t (\log (\log(t)))$. This variety of behaviors is an important motivation to make the study for general spectrally negative L\'evy environments. 

{It is known that the behavior of the diffusion $X$ is related to the right tail of the exponential functional $\int_0^{+\infty} e^{V(s)} ds$, and such results also exist in the discrete case (see for example \cite{Enriquez200900}, \cite{enriquez2013}). A novelty in the present paper is to show, in Theorem \ref{limsupenfctdelefttail}, that the asymptotic almost sure behavior of $\mathcal{L}^*_X(t)$ is crucially related to the left tail of another exponential functional, $I(V^{\uparrow})$, which allows to obtain optimal results for the $\limsup$, in Theorem \ref{limsupkappa<1exact}, when $\Psi_V$ is nice. Even though the distribution of $I(V^{\uparrow})$ appeared in a complicated way in the limit distribution of $\mathcal{L}^*_X(t)/t$ in \cite{caslevyvech}, it is a surprising fact that the distribution of $I(V^{\uparrow})$ (more precisely its left tail) also totally determines the asymptotic behavior of the extremely large values of the local time.} 

Even if, for technical reasons, the above theorems do not apply when the environment $V$ has bounded variation, we can conjecture that the behavior of $\mathcal{L}^*_X(t)$ remains linked in the same way to the left tail of $I(V^{\uparrow})$ which is given by Remark 1.8 of \cite{foncexpovech}. This would imply
\begin{conj}
If $0 < \kappa < 1$ and $V$ has bounded variation, we have 
\[ \mathbb{P} \text{-a.s.} \ 0 < \limsup_{t \rightarrow +\infty} \frac{\mathcal{L}^*_X(t)}{t} < +\infty. \]
\end{conj}
If this conjecture is true, we would have, when the environment has bounded variation, the same renormalization as in the discrete transient case given by Theorem 1.1 of \cite{GanShi}. This would not be surprising since the discrete case gives rise to potentials of bounded variation. Moreover, if $V$ has bounded variation then it is known to be the difference of a deterministic positive drift and a subordinator. The valleys can then not be steeper than the deterministic drift so, according to Remark \ref{diffavecdiscret}, the expected renormalization of $\mathcal{L}^*_X(t)$ has to be the same as in the discrete case. 

For the $\liminf$, there is only one possible renormalization. Our result is as follows: 
\begin{theo} \label{liminfkappa<1}
If $0 < \kappa < 1$, $V$ has unbounded variation and $V(1) \in L^p$ for some $p>1$, then we have 
\begin{eqnarray}
\mathbb{P} \text{-a.s.} \ 0 < \liminf_{t \rightarrow +\infty} \frac{\mathcal{L}^*_X(t)}{t / \log (\log(t))} \leq \frac{1-\kappa}{\kappa (\mathbb{E} [I(V^{\uparrow})] + \mathbb{E} [I(\hat V^{\uparrow})])}. \label{thliminfencad}
\end{eqnarray}
\end{theo}

Note that the expectations $\mathbb{E} [I(V^{\uparrow})]$ and $\mathbb{E} [I(\hat V^{\uparrow})]$ are finite and well defined since $I(V^{\uparrow})$ and $I(\hat V^{\uparrow})$ both admit some finite exponential moments according to Theorems 1.1 and 1.13 of \cite{foncexpovech}. 

\textbf{Example:} We consider $W_{\kappa}$, the $\kappa$-drifted Brownian motion ($W_{\kappa}(t) := W(t) - \frac{\kappa}{2} t$), then, the expression of the Laplace transform of $I(W_{\kappa}^{\uparrow})$ is given by equation $(1.12)$ of \cite{foncexpovech}. This expression allows to compute the moments of $I(W_{\kappa}^{\uparrow})$ and gives in particular $\mathbb{E} [I(W_{\kappa}^{\uparrow})] = 2/(1+\kappa)$. Moreover $W_{\kappa}^{\uparrow}$ and $\hat W_{\kappa}^{\uparrow}$ have the same law so $\mathbb{E} [I(W_{\kappa}^{\uparrow})] + \mathbb{E} [I(\hat W_{\kappa}^{\uparrow})] = 4/(1+\kappa)$. If we choose, as an environment, $V = W_{\kappa}$ (for $0 < \kappa < 1$), then the above upper bound for the $\liminf$ becomes $(1-\kappa^2)/4 \kappa$. Putting this in relation with the results of \cite{Devmaxloc}, we see that the application of Theorem \ref{liminfkappa<1} to the special case of a drifted Brownian environment improves \eqref{devpartielliminf2} and completes \eqref{devpartielliminf1} by proving that this renormalization is exact and by providing an explicit upper bound. 

The fact that we have many possible renormalizations for the $\limsup$, depending on the environment $V$, while only one possible for the $\liminf$, whatever is the environment $V$, might seem surprising, here is an heuristic explanation: In each valley, the time spent equals approximately the contribution to the local time multiplied by an exponential functional of the bottom of the valley (which is close to $\mathcal{R}$). The $\limsup$ involves large values of the local time at a fixed time, it is reached when the contribution to the local time in some valley is large while the time spent in the same valley is around a fixed value, this happens when the exponential functional of the bottom of the valley has a small value. The link between the $\limsup$ and the small values of an exponential functional is stated formally in Theorem \ref{limsupenfctdelefttail}. The $\liminf$ involves small values of the local time at a fixed time, it is reached when the contributions to the local time of the valleys are small while the sum of their contributions to the time spent is around a fixed value, this happens when the exponential functionals of the bottoms of some valleys are large. We see that the difference between $\limsup$ and $\liminf$ comes from the difference between the left and right tails of $\mathcal{R}$. The left tail is mainly the left tail of $I(V^{\uparrow})$ which depends on the asymptotic of $\Psi_V$, according to Theorem \ref{rappelsurqueues}, and $\Psi_V$ has many possible behaviors depending on the choice of $V$. On the other hand, the right tail is always exponential according to Theorems 1.1 and 1.13 of \cite{foncexpovech}. This explains the difference of behaviors between the $\limsup$ and the $\liminf$. 

We now treat the case $\kappa > 1$. In this case the limit distribution of $\mathcal{L}_X^*(t)/t^{1/\kappa}$ under $\mathbb{P}$ is given in Theorem 1.1 of \cite{caslevyvech}. The methods and the results are different from the case $0 < \kappa < 1$. For the almost behavior we have 

\begin{theo} \label{limsupkappa>1}
Let $f$ be a positive non-increasing function. When $\kappa > 1$, we have 
\begin{equation}
\int_1^{+\infty} \frac{(f(t))^{\kappa}}{t} dt \left\{
\begin{aligned}
& < +\infty \nonumber \\
& = +\infty \nonumber \end{aligned} \right. 
\Leftrightarrow  \limsup_{t \rightarrow +\infty} \frac{f(t) \mathcal{L}^*_X(t)}{t^{1/\kappa}} = \left\{
\begin{aligned}
& 0 \nonumber \\
& +\infty \nonumber \end{aligned} \right.
\mathbb{P} \text{-a.s.}
\end{equation}

\end{theo}

The above result is the analogue of Theorem 1.2 of \cite{GanShi} for the continuous case. In the special case where $V = W_{\kappa}$ (for $\kappa > 1$), our result coincides with \eqref{devlimsupkappa>1} (proved by Devulder in \cite{Devmaxloc}). Indeed, since $f$ is decreasing we have the easy equivalences
\begin{align}
\int_1^{+\infty} \frac{(f(t))^{\kappa}}{t} dt < + \infty \Leftrightarrow \sum_{n=1}^{+\infty} \frac{(f(n))^{\kappa}}{n} < + \infty \Leftrightarrow \sum_{n=1}^{+\infty} (f(2^n))^{\kappa} < + \infty. \label{serieint} 
\end{align}
The first equivalence shows that our result agrees with \eqref{devlimsupkappa>1} and the second equivalence allows to reformulate the integrability condition in a form that is convenient to prove Theorem \ref{limsupkappa>1}. 

\begin{remarque}
Comparing Theorems \ref{limsupkappa<1}, \ref{limsupkappa<1exact} and \ref{limsupkappa>1} we see that the renormalization for the $\limsup$ is larger in the transient sub-ballistic regime than in the transient ballistic regime, which is in accordance with intuition. However, it is surprising to see that the renormalization in the transient sub-ballistic regime is also greater than the renormalization in the recurrent case, given by \eqref{ltcuriosity}-\eqref{resultsdiel} which is  Theorem 1.1 of \cite{Diel}. Here is the heuristic explanation: In the recurrent case the diffusion is trapped in the bottom of a large valley while in the transient sub-ballistic regime the diffusion gets successively trapped in the bottom of many valleys, their bottoms being much more narrow. This explains that the large values of the local time have the tendency to be higher in the second case. 
\end{remarque}

For the $\liminf$, we provide an explicit value. Let us define the constants $K$ and $m$ similarly as in \cite{Singh}: 
\[ K := \mathbb{E} \left [ \left ( \int_0^{+\infty} e^{V(t)} dt \right )^{\kappa -1} \right ] \ \ \ \text{and} \ \ \ m := \frac{-2}{\Psi_V(1)} > 0. \]
We see that $m$ is properly defined when $\kappa > 1$. We have: 

\begin{theo} \label{liminfkappa>1}
When $\kappa > 1$, we have 
\[ \mathbb{P} \text{-a.s.} \ \liminf_{t \rightarrow +\infty} \frac{\mathcal{L}^*_X(t)}{(t/ \log (\log(t)))^{1/\kappa}} = 2 (\Gamma(\kappa) \kappa^2 K/m)^{1/\kappa}. \]

\end{theo}

\textbf{Example:} If we choose $V = W_{\kappa}$ (for $\kappa > 1$), then $K = 2^{\kappa - 1}/\Gamma(\kappa)$ (see Example 1.1 in \cite{Singh}) and $m = 4/(\kappa - 1)$. The above limit is then $4 (\kappa^2 (\kappa - 1)/8)^{1/\kappa}$. This coincides with \eqref{devliminfkappa>1} (proved by Devulder in \cite{Devmaxloc}). 

%\begin{remarque} \label{serieint}
%Since $f$ is decreasing we have the easy equivalence
%\[ \int_1^{+\infty} \frac{(f(t))^{\kappa}}{t} dt < + \infty \Leftrightarrow \sum_{n=1}^{+\infty} \frac{(f(n))^{\kappa}}{n} < + \infty \Leftrightarrow \sum_{n=1}^{+\infty} (f(2^n))^{\kappa} < + \infty \]
%\end{remarque}

For the $\liminf$ the behaviors of the different cases are in accordance with intuition: comparing Theorem 1.1 of \cite{Diel} with Theorems \ref{liminfkappa<1} and \ref{liminfkappa>1} we see that the renormalization for the $\liminf$ in the transient ballistic regime is smaller than the renormalization in the transient sub-ballistic regime which is in turn smaller than the renormalization in the recurrent case. 

The assumption that $V$ is spectrally negative is important in this work for several reasons: it allows the Laplace exponent $\Psi_V$ and its non-trivial zero $\kappa$ to be defined and the right tail of the exponential functional $\int_0^{+\infty} e^{V(u)} du$ to be explicit, $V^{\uparrow}$ is well defined, has infinite life-time, and is related to the law of ascending valleys in a simple way, the left tail of $I(V^{\uparrow})$, which is linked to the asymptotic behavior of $\mathcal{L}^*_X(t)$ when $0 < \kappa < 1$, is explicit and particularly interesting in the spectrally negative case, finally some estimates are easier to prove in the spectrally negative case. 

\subsection{Sketch of proofs and organisation of the paper} \label{sketch}

{The methods used in the present paper for the case $0 < \kappa < 1$ bear significant improvements compared with the methods used in \cite{advech} and \cite{caslevyvech}. Since the present paper studies the almost sure behavior, estimates need to be more precise which turns out to be a serious obstacle when one needs to deal with the last valley, that is, the valley where the diffusion is stuck at instant $t$. In general, what happens inside this last valley is not easy to quantify since for this valley we have to look at the contribution of the valley until time $t$ instead of the total contribution of the valley, this truncated peak of local time is difficult to estimate and is not independent from the contributions of the previous valleys. Moreover, a difficulty comes from the fact that this valley and what happens inside it is different from other valleys precisely because it is "a valley that is conditioned to trap the diffusion for a substantial amount of time". This last difficulty can be overcome by discretizing on the possible indices for the last valley but this requires to prove that precise estimates hold for a large number of valleys simultaneously. In order to get fine estimates we here extend the methods from \cite{advech} and \cite{caslevyvech} to a "more quenched" setting: in Section \ref{smallkappa} we define carefully a set $\mathcal{G}_{t}$ of good environments that has high probability and we establish in Fact \ref{fameuxe8} that for all good environments, an inequality satisfied by the truncated peak of local time in the $k^{th}$ valley corresponds with high quenched probability to an inequality for the random variable $R_k^t$ that represent the exponential functional of the bottom of the valley (more precisely, Fact \ref{fameuxe8} says that for good environments the quenched probability that the $k^{th}$ peak of local time satisfies some inequality while $R_k^t$ does not satisfy the corresponding inequality is small). Then, having such estimates available in the quenched setting allows to make a better use of the Markov property, which allows to get more independence and to optimize the use of the renewal structure for the diffusion. A surprising fact that is that, even though the contribution of the last valley appears in the limit distribution of $\mathcal{L}^*_X(t)/t$, the method just described allows us to prove Lemma \ref{neglectlastvalley} which shows that this contribution can be neglected when we deal with the extremely large values of the local time, and this is the reason why we can get optimal results for the $\limsup$. For the extremely small values, the method allows to control the influence of the last valley in the proof of Proposition \ref{approxdupetittl}, but the last valley cannot be totally neglected and stays in the way, this is one of the reasons why the results for the $\liminf$ are a little less precise. Another reason is that the behavior of the extremely small values of the local time is related to the right distribution tail of $\mathcal{R}$ but the later is not precisely known. 

Once the effect of the last valley has been removed from the study, as described above, we can prove Propositions \ref{approxdutl} and \ref{approxdupetittl}, which transforms the problem into the study of the extreme behaviors of functionals (of the renewal structure) that represent the highest peak of local time among all valleys before the last valley (the peak of local time in the last valley is not taken into account for $Y_1^{\natural, t} ( Y_2^{-1, t}(1) -)$, the functional involved in Proposition \ref{approxdutl}, but it is for $Y_1^{\natural, t} ( Y_2^{-1, t}(1) )$, the functional involved in Proposition \ref{approxdupetittl}). The extreme behaviors of these functionals are obtained in Propositions \ref{queueloilimsupdiscret} and \ref{queueloilimdiscret}. It can easily be seen from the results of \cite{caslevyvech} that these functionals converge in distribution to functionals of a $\kappa$-stable bivariate subordinator. However, the extreme behaviors of these functionals had not been studied before and cannot be identified trivially. Our first step to prove Propositions \ref{queueloilimsupdiscret} and \ref{queueloilimdiscret} is to reverse the problem: instead of considering the highest peak of local time until the sum of the contributions to the time spent by the diffusion in the successive valleys becomes greater than $t$, we consider the sum of the contributions to the time until the first peak of local time higher than some extremely large or small value. For example, for Proposition \ref{queueloilimsupdiscret}, we will need to study the event $\{ Y_1^{\natural, t} ( Y_2^{-1, t}(1) -) \geq y_t \}$ (for $y_t$ a large positive number depending on $t$) which represents the fact that, before the sum of the contributions to the time spent by the diffusion in the successive valleys becomes greater than $t$, there is a peak of local time that is greater than $y_t$. This event is equal to the event where, at the first peak of local time higher than $y_t$, the sum of the contributions to the time spent by the diffusion in the successive valleys is less than $t$. We thus need to deal with two quantities: 1) the sum of the contributions to the time spent by the diffusion in the successive valleys before $-$ and not including $-$ the valley of the first peak of local time higher than $y_t$, 2) the contribution to the time spent in the valley of the first peak of local time higher than $y_t$. For Proposition \ref{queueloilimdiscret}, only a quantity similar to 1) appears. In any case, the index of the first peak of local time higher than $y_t$ (or any other fixed value) follows a geometric distribution and the time contributions for 1) and 2) are obtained by taking "normal valleys" conditioned on the fact that their peak of local time is, respectively, smaller and higher than $y_t$, which allows to study their extreme behaviors either directly or via the asymptotic behavior of their Laplace transform. Results on the Laplace transform of these particular time contributions are established in Lemmas \ref{etudelaplace} and \ref{etudelaplacege}. In order to proves these lemmas (and, in the case of the $\limsup$, to understand directly the extreme behavior of the quantity 2)), we fully exploit the precise knowledge of the extreme behavior of the peak of local time in a normal valley, which was established in \cite{caslevyvech} as an intermediary result and is recalled in Fact \ref{queueiid} below.} 

The rest of the paper is organized as follows. 

\textit{In Section} \ref{smallkappa} we study the case $0 < \kappa < 1$. We first recall the decomposition of the environment into valleys from \cite{caslevyvech} and the behavior of the diffusion with respect to these valleys. In particular, we recall how the renewal structure of the diffusion allows to approximate the supremum of the local time and the time spent by the diffusion in the bottom of the valleys by an \textit{iid} sequence of $\mathbb{R}_+^2$-valued random variables. {We also define the set $\mathcal{G}_{t}$ of good environments and state Fact \ref{fameuxe8} mentioned above.} 
%which involves, in particular, random variables close to the exponential functional $I(V^{\uparrow})$. 

For the $\limsup$, we study the asymptotic of the probability $\mathbb{P} (\mathcal{L}^*_X(t)/t \geq x_t)$, where $x_t$ is a suitably chosen quantity that goes to infinity with $t$. {More precisely, we prove Lemma \ref{neglectlastvalley} which allows to neglect the effect of the last valley, and we can then prove Proposition \ref{approxdutl} which relates very precisely the asymptotic of the probability $\mathbb{P} (\mathcal{L}^*_X(t)/t \geq x_t)$} with the one of $\mathbb{P} ( Y_1^{\natural, t} ( Y_2^{-1, t}(1) - ) \geq x_t )$. 
%, where $Y_1^{\natural, t} ( Y_2^{-1, t}(1) - )$ is a functional of the above mentioned \textit{iid} sequence. 
In Proposition \ref{queueloilimsupdiscret} we link the asymptotic behavior of $\mathbb{P} ( Y_1^{\natural, t} ( Y_2^{-1, t}(1) - ) \geq x_t )$ with the left tail of $I(V^{\uparrow})$ (or $\mathcal{R}$ in the case of a drifted Brownian potential). The synthesis of Propositions \ref{approxdutl} and \ref{queueloilimsupdiscret} allows to compare the behavior of $\mathbb{P} (\mathcal{L}^*_X(t)/t \geq x_t)$ with the left tail of $I(V^{\uparrow})$ (or $\mathcal{R}$ in the case of a drifted Brownian potential) in Proposition \ref{synthese}. This proposition entails Theorem \ref{limsupenfctdelefttail} by the mean of the Borel-Cantelli Lemma and of the technical Lemma \ref{lemprindep}, which decomposes the trajectory of the diffusion into large independent parts in order to get the required independence to apply the Borel-Cantelli Lemma. The combination of Theorem \ref{limsupenfctdelefttail} with what is known (and recalled in Theorem \ref{rappelsurqueues}) for the left tail of $I(V^{\uparrow})$ easily yields Theorems \ref{limsupkappa<1} and \ref{limsupkappa<1exact}, which solves the problem for the $\limsup$. 

For the $\liminf$, we study the quantity $\mathbb{P} (\mathcal{L}^*_X(t) \leq t /x_t)$. In Proposition \ref{approxdupetittl} it is compared with $\mathbb{P} ( Y_1^{\natural, t} ( Y_2^{-1, t}(1) ) \leq 1/x_t)$. 
%, where $Y_1^{\natural, t} ( Y_2^{-1, t}(1) )$ is another functional of the $\mathbb{R}_+^2$-valued \textit{iid} sequence. 
Taking care of the last valley is done directly in the proof of Proposition \ref{approxdupetittl}, with the method described in the beginning of this subsection. In Lemmas \ref{etudelaplace} and \ref{etudelaplacege} we study the Laplace transform of a random variable involved in this functional, as explained above. This allows to give a lower and an upper bound for $\mathbb{P} ( Y_1^{\natural, t} ( Y_2^{-1, t}(1) ) \leq 1/x_t)$ in Proposition \ref{queueloilimdiscret}. The synthesis of Propositions \ref{approxdupetittl} and \ref{queueloilimdiscret} gives the asymptotic of $\mathbb{P} (\mathcal{L}^*_X(t) \leq t /x_t)$ in Proposition \ref{pretpourbcindepliminf}. This Proposition entails Theorem \ref{liminfkappa<1} by the mean of the Borel-Cantelli Lemma and Lemma \ref{lemprindep}. This solves the problem for the $\liminf$. 

\textit{In Section} \ref{bigkappa} we study the case $\kappa > 1$. In this case, the local time at $t$ can be approximated by the local time at a hitting time and the latter has the same law as the generalized Ornstein-Uhlenbeck process introduced in \cite{Singh}. Using what is known for the excursion measure of this process we prove Theorems \ref{limsupkappa>1} and \ref{liminfkappa>1}. 

\textit{In Section} \ref{toolbox} we justify some facts about $V$, $V^{\uparrow}$ and the diffusion in $V$ that are used along the paper. 

\subsection{Facts and notations} \label{factnot}

For $Y$ a process and $S$ a borelian set, we denote
\[ \tau(Y, S) := \inf \left \{ t \geq 0, \ Y(t) \in S \right \}, \ \ \ \mathcal{K}(Y, S) := \sup \left \{ t \geq 0, \ Y(t) \in S \right \}. \]
We shall only write $\tau(Y, x)$ instead of $\tau(Y, \{x\})$ and $\tau(Y, x+)$ instead of $\tau(Y, [x, +\infty [)$. Since $V$ has no positive jumps we see that each positive level is reached continuously (or not reached at all): $\forall x > 0, \ \tau(V, x+) = \tau(V, x)$ (which is possibly infinite). Moreover, the law of the supremum of $V$ is known, it is an exponential distribution with parameter $\kappa$ (see Corollary VII.2 in Bertoin \cite{Bertoin}). 

If $Y$ is Markovian and $x \in \mathbb{R}$ we denote $Y_x$ for the process $Y$ starting from $x$. For $Y_0$ we shall only write $Y$. 
%For any (possibly random) time $T > 0$, we write $Y^T$ for the process $Y$ shifted and centered at time $T$: $\forall s \geq 0, \ Y^T(s) := Y(T+s)-Y(T)$. 
When it exists we denote by $(\mathcal{L}_Y(t, x), t\geq0, x \in \mathbb{R})$ the version of the local time of $Y$ that is continuous in time and c\`ad-l\`ag in space and by $(\sigma_Y(t, x), t \geq 0, x \in \mathbb{R})$ the inverse of this local time: $\sigma_Y(t, x) := \inf \{ s \geq 0, \ \mathcal{L}_Y(s, x) > t \}$. 

Let $B$ be a Brownian motion starting at $0$ and independent from $V$. A diffusion in potential $V$ can be defined via the formula 
\begin{equation}
X(t) := A_V^{-1}(B(T_V^{-1}(t))), \label{exprediff1}
\end{equation}
where
\[ A_V(x) := \int_0^x e^{V(u)} du \text{ and for } 0 \leq s \leq \tau \left ( B, \int_0^{+\infty} e^{V(u)} du \right ), \ T_V(s) := \int_0^s e^{-2 V(A_V^{-1}(B(u)))} du. \]
{$A_V$ is the scale function of the diffusion $X$ (at fixed environment $V$). Note that, by the time-reversal property for L\'evy processes, the process $(-V(-(x-)), \ x \geq 0)$ has the same law as $(V(x), \ x \geq 0)$ so $V$ converges almost surely to $+\infty$ at $-\infty$. In particular we have $A_V(-\infty) = -\infty$ while $A_V(+\infty) < +\infty$. This implies that $X$ is almost surely transient toward $+\infty$. }

It is known that the local time of $X$ at $x$ until instant $t$ has the following expression:   
\begin{equation}
\mathcal{L}_X(t,x)= e^{- V(x)}\mathcal{L}_B(T_V^{-1}(t),A_V(x)). \label{expretl1}
\end{equation}
%We sometimes study some events for the diffusion under t
Recall the notation $\mathcal{L}^{*}_X(t)$ for the overall supremum of the local time until time $t$. We use the notation $\mathcal{L}_X^{*, +}$ for the supremum of the local time on the positive half-line: 
\[ \mathcal{L}_X^{*, +}(t) := \sup_{x \in [0, +\infty[} \mathcal{L}_X(t, x). \]
For the hitting times of $r \in \mathbb{R}$ by the diffusion $X$ we shall use the frequent notation $H(r)$ (instead of $\tau(X, r)$). 
%We denote by $H_{+}(r)$ (respectively $H_{-}(r)$) the total time spent by the diffusion in $[0, +\infty [$ (respectively $]-\infty, 0]$) before $H(r)$. 

$D(\mathbb{R}, \mathbb{R})$ is the space of c\`ad-l\`ag functions from $\mathbb{R}$ to $\mathbb{R}$. Let $P$ be the probability measure on $D(\mathbb{R}, \mathbb{R})$ inducing the law of $V$. For $v \in D(\mathbb{R}, \mathbb{R})$, the quenched probability measure $P^v$ is the probability measure (associated with the diffusion $X$) conditionally on $\{ V = v \}$. $\mathbb{P}$ represents the annealed probability measure, it is defined as $\mathbb{P}(.) := \int_{D(\mathbb{R}, \mathbb{R})} P^{v} (.) P(d v)$. $X$ is a Markovian process under $P^v$ but not under $\mathbb{P}$. Note that all the almost sure convergences stated in this Introduction are $\mathbb{P}$-almost sure convergences. For objects not related to the diffusion $X$ we also use the natural notation $\mathbb{P}$ for a probability. 

%$P^v (A)$ is the probability that the event $A$ is realized, conditionally on $\{ V = v \}$. The diffusion in environment $v$ is well defined thanks to formula ... (+ TOUT BIEN DEFINI) so it is easy to see that $P^v (A)$ is well defined. DEFINIR ANNEALED $\mathbb{P}$ (DIRE QU'ON UTILISE AUSSI CETTE NOTATION SUR UN AUTRE ESPACE PROBABILISE) $P$ MESURE SUR $D(\mathbb{R}, \mathbb{R})$
%DEFINIR LES NOTATI ONS, NOTAMMENT DE TEMPS D'ATTEINTE ET D'INVERSE DU TL
%SUP DU TEMPS LOCAL EN LES POSITIFS

%DEFINIR $A$ ET AUSSI $A^j(x)=\int_{\tilde m_j}^x e^{\tilde V^{(j)}(s)} d s$
If $Z$ is a random variable, its law is denoted by $\mathcal{L} ( Z)$ and if $A$ is an event of positive probability, $\mathcal{L} (Z | A)$ denotes the law of $Z$ conditionally to the event $A$. 

For $Z$ an increasing c\`ad-l\`ag process and $s \geq 0$, we put respectively $Z(s-)$, $Z^{\natural}(s)$ and $Z^{-1}(s)$ for respectively the left-limit of $Z$ at $s$, the largest jump of $Z$ before $s$ and the generalized inverse of $Z$ at $s$: 
\[ Z(s-) = \lim_{r \underset{<}{\rightarrow} s} Z(r), \ Z^{\natural}(s) := \sup_{0\leq r \leq s}(Z(r)-Z(r-)),\ Z^{-1}(s) := \inf\{u \geq 0,\ Z(u) > s\}. \]

For two quantities $a$ and $b$ depending on a parameter, $a \approx b$ means that $\log (a) \sim \log(b)$ when the parameter converges (generally to $0$ or infinity). 

%VOIR CE QUI DANS CASLEVY DOIT SE REMETTRE ICI

%POUR COND POS: SI BESOIN METTRE LE FACT MANQUANT (INVOQUE DANS CE QUI SUIT)
%PEUT ETRE METTRE CA DANS LA SOUS-SECTION "PROP DE V ET VUP"

%ON AURAIT PU FAIRE COMME CA MAIS CA PARAIT DIFFICIL
%\section{Technical lemmata}
%
%MONTRER AUSSI LE MEME CHOSE QUE LE LEMME SUIVANT POUR LES VRAIS TEMPS DE SORTIE
%
%\begin{lemme}
%
%There is a positive constant $\mathcal{C}$ such that when $y_t$ goes to $+\infty$ when $t$ goes to $+\infty$ we have
%
%\[ \mathbb{P} \left ( H(\tilde L_{N_t}) \geq y_t \right ) \underset{t \rightarrow +\infty}{\sim} \mathcal{C} y_t^{-\kappa} \]
%
%\end{lemme}
%
%\begin{proof}
%
%For $y_t > 1$ we have
%\begin{align*}
%\mathbb{P} \left ( Y_2^t \left ( Y_2^{-1, t}(a) \right ) \geq y_t \right ) & = mmm
%\end{align*}
%
%\end{proof}
%
%\begin{lemme}
%
%There is a positive constant $\mathcal{C}$ such that when $y_t$ goes to $+\infty$ when $t$ goes to $+\infty$ we have
%
%\[ \mathbb{P} \left ( Y_2^t \left ( Y_2^{-1, t}(a) \right ) \geq y_t \right ) \underset{t \rightarrow +\infty}{\sim} \mathcal{C} y_t^{-\kappa} \]
%
%\end{lemme}
%
%\begin{proof}
%
%For $y_t > 1$ we have
%\begin{align*}
%\mathbb{P} \left ( Y_2^t \left ( Y_2^{-1, t}(a) \right ) \geq y_t \right ) & = mmm
%\end{align*}
%
%\end{proof}

\section{Almost sure behavior when $0 < \kappa < 1$} \label{smallkappa}

In all this section we assume that the hypotheses of Theorems \ref{limsupkappa<1} and \ref{liminfkappa<1} are satisfied: $0 < \kappa < 1$, $V$ has unbounded variation and there exists $p>1$ such that $V (1) \in L^p$. The hypothesis of unbounded variation is necessary to approximate the law of the left part of a valley by the law of $\hat V^{\uparrow}$ and the hypothesis about moments for $V(1)$ allows to neglect the local time outside the bottoms of the valleys. For these reasons, many results of \cite{caslevyvech} (that are recalled in the next subsection) have been proved under these hypotheses. 

\subsection{Traps for the diffusion} \label{trapsdiff}

We now recall some definitions about valleys and describe how the diffusion gets trapped into successive valleys. The facts and lemmas stated in this subsection are more or less classical and there are all proved or justified in Subsection \ref{preuvedesfacts} (except Fact \ref{minimacoincide} which is readily Lemma 3.5 of \cite{caslevyvech}). 

We first recall the notion of $h$-extrema. For $h>0$, we say that $x\in\mathbb{R}$ is an $h$-minimum for $V$ if there exist $u<x<v$ such that $V(y) \wedge V(y-) \geq V(x) \wedge V(x-)$ for all $y\in[u,v]$, $V(u)\geq (V(x) \wedge V(x-))+h$ and $V(v-)\geq (V(x) \wedge V(x-))+h$. Moreover, $x$ is an $h$-maximum for $V$ if $x$ is an $h$-minimum for $-V$, and $x$ is an $h$-extremum for $V$ if it is an $h$-maximum or an $h$-minimum for $V$. 

Since $V$ is not a compound Poisson process, it is known (see Proposition VI.4, in \cite{Bertoin}) that it takes pairwise distinct values in its local extrema. Combining this with the fact that $V$ has almost surely c\`ad-l\`ag paths and drifts to $-\infty$ without being the opposite of a subordinator, we can check that the set of $h$-extrema is discrete, forms a sequence indexed by $\mathbb{Z}$, unbounded from below and above, and that the $h$-minima and $h$-maxima alternate. Let $\mathcal{V} \subset D (\mathbb{R}, \mathbb{R})$ be the set of the environments $v$ that satisfy the above properties and that are such that 
\[ v(x) \underset{x \rightarrow +\infty}{\longrightarrow} -\infty, \ \ \ v(x) \underset{x \rightarrow -\infty}{\longrightarrow} +\infty, \ \ \ \int_0^{+\infty} e^{v(x)} dx < +\infty. \]
Note that the path of $V$ belongs to $\mathcal{V}$ with probability $1$. 

We denote respectively by $(m_i,\ i \in \mathbb{Z})$ and $(M_i,\ i \in \mathbb{Z})$ the increasing sequences of $h$-minima and of $h$-maxima of $V$, such that $m_{0} \leq 0<m_1$ and $m_i<M_i<m_{i+1}$ for every  $i\in\mathbb{Z}$. An $h$-valley is the fragment of the trajectory of $V$ between two $h$-maxima. 
% PAS BESOIN D'EN DIRE PLUS PUISQUE CA NE SERT PLUS APRES, SAUF SI ON VEUT LA LOI DES VALLEES CLASSIQUES
%, translated at the $h$-minima between them: the $i^{th}$ classical $h$-valley is the process \[ \left ( V^{(i)}(x), \ M_{i-1} \leq x \leq M_i \right ) \ \ \ \textit{where} \ \ \ V^{(i)} := V(x)-V(m_i), \ \ \forall x\in\mathbb{R}. \]

The valleys are visited successively by the diffusion. For the size of the valleys to be well adapted with respect to the time scale, we have to make the size of the valleys grow with time $t$. We are thus interested in $h_t$-valleys where
\begin{eqnarray}
h_t := \log(t) - \phi(t), \ \text{with} \ \phi(t) := (\log ( \log (t) ))^{\omega}, \label{paramtaillevallees}
\end{eqnarray}
where $\omega > 1$ will be chosen later in accordance with some other parameters. We also define $N_t$, the index of the largest $h_t$-minima visited by $X$ until time $t$, 
\[ N_t :=\max \left\{ k\in\mathbb{N},\ \sup_{ 0 \leq s \leq t}X(s) \geq m_{k} \right \}. \]
We need deterministic bounds for the number of visited valleys. We define 
\[ n_t := \lfloor e^{\kappa (1+\delta) \phi(t)} \rfloor \ \ \ \text{and} \ \ \ \tilde n_t := \lfloor e^{\rho \phi(t)} \rfloor, \]
where $\delta > 0$ is small enough so that $(1+3\delta)\kappa<1$, $\rho \in ]0, \kappa/(1+\kappa)[$, and $\delta$ and $\rho$ are fixed once and for all in all the paper. The following lemma says that with high probability, $\tilde n_t \leq N_t \leq n_t$: 
\begin{lemme} \label{minmajnbvalleesvisit}

There is a positive constant $c$ such that for all $t$ large enough, 
\begin{align}
\mathbb{P} \left( N_t \geq n_t \right) & \leq e^{- c h_t} \leq e^{- c \phi(t)}, \label{majonbvalleesvisit} \\
\mathbb{P} \left( N_t \leq \tilde n_t \right) & \leq e^{- c \phi(t)}. \label{minonbvalleesvisit}
\end{align}

\end{lemme}

We recall the definition of the standard valleys given in Section 3.3 of \cite{caslevyvech}. Their interest is mainly the fact that they are defined via successive stopping times, which make them convenient to use in the calculations, and also the fact that they take into consideration the descending phases between two $h_t$-minima. 

%VOIR SI ON FAIT AVEC $h$ OU $h_t$. 

Recall that $\delta > 0$ is small enough so that $(1+3\delta)\kappa<1$. Assume $t$ is large enough so that $e^{(1-\delta)\kappa h_t} \geq h_t$. We define  $\tilde \tau_0(h_t) = \tilde L_0 :=0$ and recursively for $i \geq 1$, 
\begin{eqnarray}
     \tilde L_i^{\sharp}
&:= &
    \inf\{x>\tilde L_{i-1},\ V(x)\leq V(\tilde L_{i-1})-e^{(1-\delta)\kappa h_t}   \},
\nonumber\\
    \tilde \tau_i(h_t)
&:= &
    \inf \big\{x \geq   \tilde L_i^{\sharp},\ V(x)-\inf\nolimits_{[\tilde L_i^{\sharp},x]}V = h_t\big\},
\nonumber\\
    \tilde{m}_i
&:= &
    \inf\big\{x \geq \tilde L_i^{\sharp},\ V(x)=\inf\nolimits_{[\tilde L_i^{\sharp},\tilde \tau_i(h_t)]} V \big\},
\nonumber\\
    \tilde L_i
&:= &
    \inf\{x>\tilde \tau_i(h_t),\ V(x) {- V(\tilde{m}_i)} \leq h_t/2 \}
\nonumber\\
&= &
    {\inf\{x>\tilde \tau_i(h_t),\ V(x) - V(\tilde \tau_i(h_t)) \leq -h_t/2 \},}
\nonumber\\
    \tilde \tau_i^-(a)
&:= &
    \sup \{x < \tilde m_i,\  V(x)-V(\tilde{m}_i) \geq a\}, \ \forall a \in [0, h_t],
\nonumber\\
    \tilde \tau_i^+(a)
&:= &
    \inf \{x > \tilde m_i,\  V(x)-V(\tilde{m}_i) = a\}, \ \forall a \in [0, h_t].
\nonumber
\end{eqnarray}

These random variables depend on $h_t$ and therefore on $t$, even if this does not appear in the notations. 
%We also introduce the equivalent of $V^{(i)}$ for the $\tilde m_i, i \in \mathbb{N}^*$ as follows: 
We also define
\[ \tilde V^{(i)}(x) := V(x)-V(\tilde m_i), \ \ \forall x\in\mathbb{R}. \]
We call $i^{th}$ \textit{standard valley} the re-centered truncated potential
$(\tilde V^{(i)} (x),\ \tilde L_{i-1} \leq x \leq \tilde L_{i} )$. The law of the bottom of these valleys is given in Fact \ref{loidesval} of Section \ref{toolbox}. 

Similarly as in \cite{caslevyvech} we define the \textit{deep bottoms} of the $j^{th}$ standard valleys to be the interval
\begin{eqnarray}
\mathcal{D}_j := [\tilde \tau_j^-((\phi(t))^2), \tilde \tau_j^+((\phi(t))^2)]. \label{defdj}
\end{eqnarray}

\begin{remarque} \label{iid}
The random times $\tilde L_i^{\sharp}$, $\tilde \tau_i(h_t)$, and $\tilde L_{i}$ are stopping times. As a consequence, the sequence $(\tilde V^{(i)} (x + \tilde m_i),\ \tilde L_{i-1} - \tilde m_i \leq x \leq \tilde L_{i} - \tilde m_i)_{i \geq 1}$ is \textit{iid}. 
\end{remarque}

%Our definitions take in consideration the absence of positive jumps for $V$, in particular $\tilde \tau_i(h) < + \infty$ and $V(\tilde \tau_i(h)-) = V(\tilde \tau_i(h))=V(\tilde m_i)+h$. We can see that the $\tilde m_i$, $i\in\mathbb{N}^*$, are $h$-minima. The next lemma, which is the analogue of Lemma 2.3 in Andreoletti, Devulder \cite{AndDev}, shows that, with high probability, the sequence $(\tilde m_i)_{i \geq 1}$ coincides with the sequence $(m_i)_{i \geq 1}$ for indices $i \leq n$ when $n$ does not grow too fast with $h$. 

We also have that the sequence $(\tilde m_i)_{i \geq 1}$ of the minima of the standard valleys coincides with the sequence $(m_i)_{i \geq 1}$ with high probability for a large number of indices. Let
\[ \mathcal{V}_t := \left \{ v \in \mathcal{V}, \ \forall i \in \{1, ..., n_t \}, \ m_i = \tilde{m}_i \right \}. \]
Note from the definition of $\mathcal{V}$ that the sequences $(\tilde m_i)_{i \geq 1}$ and $(m_i)_{i \geq 1}$ are always defined for any $v \in \mathcal{V}$ so in particular the event $\mathcal{V}_t$ is well defined. We have: 
%D'ABORD DEFINIR $\mathcal{V}$ QUI SERT EN PARTICULIER A CE QUE LES VALLEES TILDES SOIENT DEFINIES ET L'INTEGRALEDE L'ENV ENTRE 0 ET +L'INFINI DOIT ETRE FINIE

\begin{fact} \label{minimacoincide} (Lemma 3.5 of \cite{caslevyvech})

There is a positive constant $c$ such that for all $t$ large enough, 
\[ P \left ( V \in \mathcal{V}_t \right ) \geq 1 - e^{- c h_t}. \]
\end{fact}

%LOI DES VALLEES STANDARDS ? LEMME 3.7 ?

%VALLEES ET $n_t$

We define $X_{\tilde m_j} := X(. + \tilde m_j)$ which is, according to the Markov property, a diffusion in potential $V$ starting from $\tilde m_j$. We also define, for any $r \in \mathbb{R}$, $H_{X_{\tilde m_j}}(r)$ to be the hitting time of $r$ by $X_{\tilde m_j}$. When we deal with $X_{\tilde m_j}$ we often need the notation $A^j(x)=\int_{\tilde m_j}^x e^{\tilde V^{(j)}(s)} d s$. 

As in \cite{advech} and \cite{caslevyvech}, we approximate the distribution function of the renormalized local time by distribution functions of functionals of the sequence $(e_j S_j^t, e_j S_j^t R_j^t)_{j \geq 1}$ where 
\[ e_j := \mathcal{L}_X ( H(\tilde L_j), \tilde m_j) / A^j(\tilde L_j), \ \ \ S_j^t := \int_{\tilde \tau_j^+(h_t / 2)}^{\tilde L_j} e^{\tilde V^{(j)}(u)} du, \ \ \ R_j^t := \int_{\tilde \tau_j^-(h_t / 2)}^{\tilde \tau_j^+(h_t / 2)} e^{-\tilde V^{(j)}(u)} du. \]
In Section 4 of \cite{caslevyvech}, it is shown that $e_j$ follows an exponential distribution with parameter $1/2$ (since the distribution of $e_j$ does not depend on $t$ we omit the dependence in $t$ for $e_j$ in the notations) and that the random variables $e_j, S_j^t, R_j^t, \ j \geq 1$ are mutually independent. To simplify notations we define, as in \cite{advech} and \cite{caslevyvech}, the process of the renormalized sum of the contributions: 
\[ \forall s \geq 0, \ (Y_1^t, Y_2^t)(s) := \frac1{t} \sum_{j=1}^{\lfloor s e^{\kappa \phi(t)} \rfloor} (e_j S_j^t , e_j S_j^t R_j^t), \]
and the overshoots of $\sum_{i=1}^{.} e_i S_i^t R_i^t$: for any $a \geq 0$, let us define
\[ \mathcal{N}_a := \min \left \{ j \geq 0, \ \sum_{i=1}^{j} e_i S_i^t R_i^t > a \right \}. \]

We have

\begin{fact} \label{cvsub} (Proposition 4.2 of \cite{caslevyvech})

$(Y_1^t, Y_2^t)$ converges in distribution in $(D([0, +\infty[, \mathbb{R}^2), J_1)$ to a non-trivial bidimentional $\kappa$-stable subordinator $(\mathcal{Y}_1,\mathcal{Y}_2)$. 

\end{fact}

Let us introduce some events that occur with high probability. They describe the behavior of the diffusion and provide effective approximations for the time and the local time, in the study of the case $0 < \kappa < 1$. On these events, the diffusion leaves the valleys from the right and never goes back to a previous valley, the local time and the time spent by the diffusion are negligible, compared with $t$, outside the bottoms of the valleys, the supremum of the local time and the time spent by the diffusion in the bottom of the valleys are approximated by the \textit{iid} sequence $(e_j S_j^t, e_j S_j^t R_j^t)_{j \geq 1}$. 
\begin{align*}
\mathcal{E}^1_t & := \bigcap_{j=1}^{n_t} \left \{ H_{X_{\tilde m_j}}(\tilde L_j) < H_{X_{\tilde m_j}}(\tilde L_{j-1}), \ H_{X_{\tilde L_j}}(+\infty) < H_{X_{\tilde L_j}}(\tilde \tau_j(h_t)) \right \}, \\
\mathcal{E}^2_t & := \bigcap_{j=0}^{n_t - 1} \left \{ \sup_{y \in \mathbb{R}} \left ( \mathcal{L}_X (H(\tilde m_{j+1}),y) - \mathcal{L}_X (H(\tilde L_j),y) \right ) \leq t e^{(\kappa (1+3\delta )-1)\phi(t)} \right \}, \\
\mathcal{E}^3_t & := \bigcap_{j=1}^{n_t} \left \{ \sup_{y \in [\tilde{L}_{j-1}, \tilde{L}_j] \cap \overline{\mathcal{D}_j}} \left ( \mathcal{L}_X (H(\tilde L_j),y) - \mathcal{L}_X (H(\tilde m_{j}),y) \right ) \leq t e^{-2 \phi(t)} \right \}, \\
\mathcal{E}^4_t & := \bigcap_{j=1}^{n_t} \left \{ \sup_{y \in \mathcal{D}_j} \left ( \mathcal{L}_X(H(\tilde L_j), y) - \mathcal{L}_X(H(\tilde m_j), y) \right ) \leq (1+e^{-\tilde c h_t}) \mathcal{L}_X(H(\tilde L_j), \tilde m_j) \right \}, \\
%\mathcal{E}^4_t & := \bigcap_{j=1}^{n_t} \left \{ \left | \sup_{y \in \mathcal{D}_j} \mathcal{L}_X(t, y) - \mathcal{L}_X(t, \tilde m_j) \right | \leq ... \right \}, \\
\mathcal{E}^5_t & := \bigcap_{j=1}^{ n_t}\left\{0\leq H(\tilde m_{j})-\sum_{i=1}^{j-1}\left ( H(\tilde L_i)-H(\tilde m_i) \right ) \leq \frac{2t}{\log h_t} \right\}, \\
%\mathcal{E}^5_t(k) & := \left\{0\leq H(\tilde m_{k+1})- H(\tilde L_k) \leq t e^{-\phi(t) /2} \right\}, \ k \geq 1, \\
\mathcal{E}^6_t & := \bigcap_{j=1}^{n_t} \left \{ (1-e^{- \tilde c h_t}) e_j S_j^t \leq \mathcal{L}_X(\tilde m_j, H(\tilde{L}_j)) \leq (1+e^{- \tilde c h_t}) e_j S_j^t \right \}, \\
\mathcal{E}^7_t & := \bigcap_{j=1}^{n_t} \left \{ (1-e^{- \tilde c h_t}) e_j S_j^t R_j^t \leq H(\tilde L_j)-H(\tilde m_j) \leq (1+e^{- \tilde c h_t}) e_j S_j^t R_j^t \right \}. 
%\mathcal{E}^8_t & := \left \{ \frac{\sup_{y \in \mathcal{D}_1} \mathcal{L}_{X'}(t(1-z), y)}{t} \leq \frac{1-z}{R_1^t} (1+ ...) \right \}. 
\end{align*}
Here, $\tilde c$ is a fixed positive constant that has been chosen small enough (the constraints for the choice of $\tilde c$ will be specified in the proofs of Fact \ref{bigevents} and Fact \ref{fameuxe8}). Note that the above events depend both on the environment $V$ and on the Brownian motion driving the diffusion. 

\begin{fact} \label{bigevents}
There is a positive constant $L$ such that for all $t$ large enough, 
%\[ \sum_{i=1}^7 \mathbb{P} \left ( \overline{\mathcal{E}^i_t} \right ) \leq e^{-c \phi(t)} \ \ \ \text{and} \ \ \ \forall k \geq 1, \ \mathbb{P} ( \overline{\mathcal{E}^5_t(k)} ) \leq e^{-c \phi(t)}. \]
\begin{eqnarray}
\mathbb{P} \left ( \overline{\mathcal{E}^7_t} \right ) \leq e^{-L h_t}, \ \ \ \sum_{i=1}^7 \mathbb{P} \left ( \overline{\mathcal{E}^i_t} \right ) \leq e^{-L \phi(t)}. \label{bigeventsleq}
\end{eqnarray}
Fix $\eta \in ]0, 1[$. If $t$ is so large such that $(1-e^{- \tilde c h_t})^{-1} < (1+\eta)$ and $(1+e^{- \tilde c h_t})^{-1} (1 - 2/\log(h_t)) \geq (1-\eta)$, then
\begin{eqnarray}
\left \{ V \in \mathcal{V}_t \right \} \cap \left \{ N_t < n_t \right \} \cap \mathcal{E}^5_t \cap \mathcal{E}^7_t \subset \left \{ \mathcal{N}_{(1-\eta)t} \leq N_t \leq \mathcal{N}_{(1+\eta)t} \right \}. \label{bigeventsnbvalles}
\end{eqnarray}
%\[ \forall i \in \{ 1, ..., 7 \}, \ \mathbb{P} \left ( \overline{\mathcal{E}^i_t} \right ) \leq e^{-c \phi(t)}. \]

%\begin{align*}
%& \mathbb{P} \left ( \overline{\mathcal{E}^1_t} \right ) \leq e^{-c h_t}, \ \mathbb{P} \left ( \overline{\mathcal{E}^2_t} \right ) \leq e^{-c \phi(t)}, \ \mathbb{P} \left ( \overline{\mathcal{E}^3_t} \right )  \leq e^{-c \phi(t)}, \ \mathbb{P} \left ( \overline{\mathcal{E}^4_t} \right ) \leq e^{-c h_t}, \ \mathbb{P} \left ( \overline{\mathcal{E}^5_t} \right ) \leq e^{-c \phi(t)}, \\
%& \mathbb{P} \left ( \overline{\mathcal{E}^6_t} \right ) \leq e^{-c h_t}, \ \mathbb{P} \left ( \overline{\mathcal{E}^7_t} \right ) \leq e^{-c h_t} 
%\end{align*}

%\begin{align}
%\mathbb{P} \left ( \overline{\mathcal{E}^1_t} \right ) & \leq e^{-c h_t} \\
%\mathbb{P} \left ( \overline{\mathcal{E}^2_t} \right ) & \leq e^{-c \phi(t)} \\
%\mathbb{P} \left ( \overline{\mathcal{E}^3_t} \right ) & \leq e^{-c \phi(t)} \\
%\mathbb{P} \left ( \overline{\mathcal{E}^4_t} \right ) & \leq e^{-c h_t} \\
%\mathbb{P} \left ( \overline{\mathcal{E}^5_t} \right ) & \leq e^{-c \phi(t)} \\
%\mathbb{P} \left ( \overline{\mathcal{E}^6_t} \right ) & \leq e^{-c h_t} \\
%\mathbb{P} \left ( \overline{\mathcal{E}^7_t} \right ) & \leq e^{-c h_t} \\
%\end{align}

\end{fact}

%The following fact, which comes from \cite{advech}, prove the negligibility of these events: 
%
%\begin{fact} \label{fameuxe8}
%There is a positive constant $c$ such that for all $t$ large enough, 
%\[ \forall k \geq 1, \ z \in ... \ \ \ \mathbb{P} \left ( \mathcal{E}^8_t(v, k, z) \right ) \leq e^{-c h_t}, \ \ \ \mathbb{P} \left ( \mathcal{E}^9_t(v, k, z) \right ) \leq e^{-c h_t}. \]
%\end{fact}

Our proofs are based on the study of the asymptotics of the quantities $\mathbb{P} (\mathcal{L}^*_X(t) \geq t x_t)$ and $\mathbb{P} (\mathcal{L}^*_X(t) \leq t /x_t)$ where $x_t$ depends on $t$ and goes to infinity with $t$. More precisely we define $x_t$ such that
\begin{eqnarray}
x_t \underset{t \rightarrow +\infty}{\sim} D (\log (\log(t)))^{\mu - 1}, \label{defxtlim}
\end{eqnarray}
where $D > 0$ and $\mu \in ]1, 2]$. Precise choices of $x_t$ will be made later, they will all satisfy \eqref{defxtlim} for some $D > 0$ and $\mu \in ]1, 2]$. 

%IL FAUT FAIRE SUR $\mathcal{E}^9_t(v, k, z)$ LES CHANGEMENTS QU'ON A FAIT SUR $\mathcal{E}^8_t(k)$, NOTAMMENT LA SUPPRESSION DE $H(\tilde L_k) \geq t(1-z)$: ON NE PEUT PAS LA FAIRE DANS $\mathcal{E}^9_t(v, k, z)$ MAIS CE N'EST PAS GRAVE PARCE QUE DANS $\mathcal{E}^8_t(k)$ CA NE POSAIT PROBLEME QUE POUR LE TERME DE BORD
Fix $\epsilon$ small enough so that Fact \ref{3integrals} of Section \ref{toolbox} is satisfied. We define $\mathcal{G}_t$ to be the set of "good environments" in the following way ; $v \in \mathcal{V}_t$ belongs to $\mathcal{G}_t$ if it satisfies the following conditions: 
\begin{align}
\forall j \in \{1, ..., n_t \}, \ e^{- \epsilon h_t /4} \leq R_j^t = \int_{\tilde \tau_j^-(h_t / 2)}^{\tilde \tau_j^+(h_t / 2)} e^{-(v(u) - v(\tilde m_j))} du & \leq e^{h_t /8}, \label{invtl1.1.0} \\
\forall j \in \{1, ..., n_t \}, \ \left | A^j(\tilde \tau_j^-(h_t / 2)) \right | = \left | \int_{\tilde m_j}^{\tilde \tau_j^-(h_t / 2)} e^{v(u) - v(\tilde m_j)} du \right | & \leq e^{5 h_t /8}, \label{invtl1.1.00} \\
\forall j \in \{1, ..., n_t \}, \ A^j(\tilde \tau_j^+(h_t / 2)) = \int_{\tilde m_j}^{\tilde \tau_j^+(h_t / 2)} e^{v(u) - v(\tilde m_j)} du & \leq e^{5 h_t /8}, \label{invtl1.1.001} \\
\forall j \in \{1, ..., n_t \}, \ \int_{\tilde{L}_{j-1}}^{\tilde \tau_j^-(h_t / 2)} e^{-(v(u) - v(\tilde m_j))} du & \leq e^{-\epsilon h_t}, \label{invtl1.1.1} \\
\forall j \in \{1, ..., n_t \}, \ \int_{\tilde \tau_j^+(h_t / 2)}^{\tilde L_j} e^{-(v(u) - v(\tilde m_j))} du & \leq e^{-\epsilon h_t},  \label{invtl1.1.12} \\
%A^j(\tilde L_j) = \int_{\tilde m_j}^{\tilde L_j} e^{-(v(u) - v(\tilde m_j))} du \leq (1 + e^{-h_t / 7}) \int_{\tilde \tau_j^+(h_t / 2)}^{\tilde L_j} e^{-(v(u) - v(\tilde m_j))} du = (1 + e^{-h_t / 7}) & S_j. \label{inegs-a}
P^v ( \cup_{i=1}^7 \overline{\mathcal{E}^i_t} ) & \leq e^{-L \phi(t)/2}, \label{bonenv1}
\end{align}
%\begin{eqnarray}
%\mathcal{G}_t := \left \{ v \in \mathcal{V}_t, \ P^v ( \cup_{i=1}^7 \overline{\mathcal{E}^i_t} ) \leq e^{-c \phi(t)/2} \right \}, \label{bonenv1}
%\end{eqnarray}
where $P^v (.)$ is defined in Subsection \ref{factnot} and $L$ is the constant defined in Fact \ref{bigevents}. Note that, since $\mathcal{G}_t \subset \mathcal{V}_t$, we will often use the fact that $\tilde m_j = m_j$ for $v \in \mathcal{G}_t$ and $j \leq n_t$. 
%DIRE QU'ON A BESOIN D'ETRE SUR $\mathcal{V}_t$ POUR QUE TOUT SOIT DEFINI (ET DU COUP DANS LA PREUVE DU LEMME SUIVANT ON INTERSECTE AVEC $\mathcal{V}_t$. VERIFIER MAIS A PRIORI ON A JUSTE BESOIN D'INTERSECTER AVEC L'EVENEMENT DE PROBA 1 SUR LEQUEL LA SUITE DES VALLEES TILDES EST DEFINIE. IL SEMBLERAIT MEME QUE FACT \ref{minimacoincide} NE SERT A RIEN ICI (VERIFIER): IL DOIT SERVIR DANS L'ARTICLE PRECENDANT POUR PROUVER DES TRUCS QUI SONT DIRECTEMENT ADMIS ICI. EN FAIT CA SERT ICI POUR MINORER LE NOMBRE DE VALLEES

\begin{lemme} \label{measofgoodenv}

There is a positive constant $c$ such that for all $t$ large enough, 
\[ P \left ( V \in \mathcal{G}_{t} \right ) \geq 1 - e^{- c \phi(t)}. \]

\end{lemme}

We need that, as in Lemma 5.3 of \cite{advech}, an inequality for the local time in the bottom of the $k^{th}$ valley is related to an inequality for the random variable $R_k^t$. Since we deal with unlikely inequalities for the local time, we have to prove the negligibility of the event where an inequality is satisfied for the local time, but not the corresponding inequality for $R_k^t$. Recall the notations $X_{\tilde m_j}$ and $H_{X_{\tilde m_j}}(.)$ defined above. For any fixed environment $v \in \mathcal{G}_t$, $k \in \{1, ..., n_t \}$, and $z \in [0, 1]$ we define 
\begin{align*}
\mathcal{E}^8_t(v, k, z) := & \left \{ (1-z) \frac{(1+e^{- \tilde c h_t})}{(1-e^{- \tilde c h_t})} < x_t R_k^t, \ \sup_{\mathcal{D}_k} \mathcal{L}_{X_{\tilde m_k}}(t(1-z), .) \geq t x_t, H_{X_{\tilde m_k}}(\tilde L_k) < H_{X_{\tilde m_k}}(\tilde L_{k-1}) \right \}, \\
\mathcal{E}^9_t(v, k, z) := & \left \{ R_k^t / x_t < (1-e^{- \tilde c h_t})(1-z), \ \mathcal{L}_{X_{\tilde m_k}}(t(1-z), \tilde m_k) \leq t / x_t, \right. \\
& \left. H_{X_{\tilde m_k}}(\tilde L_k) \geq t(1-z), H_{X_{\tilde m_k}}(\tilde L_k) < H_{X_{\tilde m_k}}(\tilde L_{k-1}) \right \}, 
\end{align*}
where $\tilde c$ is the same as in the definitions of $\mathcal{E}^4_t, \mathcal{E}^6_t$ and $\mathcal{E}^7_t$. Note that the inequalitiy $(1-z) (1+e^{- \tilde c h_t})/(1-e^{- \tilde c h_t}) < x_t R_k^t$ (resp. $R_k^t / x_t < (1-e^{- \tilde c h_t})(1-z)$) only depends on the environment $v$. We consider $\mathcal{E}^8_t(v, k, z)$ (resp. $\mathcal{E}^9_t(v, k, z)$) to equal $\emptyset$ when $v$ is such that this equality is not satisfied. In the rest of the paper, we use the same convention when we consider events, at fixed environment, that partially depend on the environment. 
%(OU BIEN ON MET AILLEURS CETTE INEGALITE. ) 
%PRECESIER QUE CE SONT DES EVENEMENTS POUR LA DIFFUSION COMMENCANT EN $\tilde m_k$ (POUR LA DIFFUSION SHIFETEE AU TEMPS D'ATTEINTE)

\begin{fact} \label{fameuxe8}
There is a positive constant $c$ such that for $t$ large enough and any fixed environment $v \in \mathcal{G}_t$ we have
%\begin{align}
%\sup_{z \in [0, 4/\log(h_t)]} \ & P^v \left ( \cup_{k=1}^{n_t} \mathcal{E}^8_t(v, k, z) \right ) \leq e^{-c \phi(t)/2} \label{majoe8} \\
%\sup_{z \in [0, 1]} \ & P^v \left ( \cup_{k=1}^{n_t} \mathcal{E}^9_t(v, k, z) \right ) \leq e^{-c \phi(t)/2}. \label{majoe9}
%\end{align}
\begin{align}
P^v \left ( \cup_{k=1}^{n_t} \{ N_t \geq k, \ H(\tilde m_{k})/t \leq 1 - 4/\log(h_t) \} \cap \mathcal{E}^8_t(v, k, H(\tilde m_{k})/t) \right ) & \leq e^{-c \phi(t)}, \label{majoe8} \\
P^v \left ( \cup_{k=1}^{n_t} \mathcal{E}^8_t(v, k, 1 - 4/\log(h_t)) \right ) & \leq e^{-c \phi(t)}, \label{majoe8bis} \\
P^v \left ( \cup_{k=1}^{n_t} \{ N_t \geq k \} \cap \mathcal{E}^9_t(v, k, H(\tilde m_{k})/t) \right ) & \leq e^{-c \phi(t)}. \label{majoe9}
\end{align}

\end{fact}

Note that in the above fact, $v$ is fixed in $\mathcal{G}_t \subset \mathcal{V}_t$ so $N_t \geq k$ implies $H(\tilde m_{k})/t \leq 1$. The quantities in the above fact are thus well defined. 
%ATTENTION ATTENTION ATTENTION: Comme on a $N_t$ il faut aussi que les vall\'ees coincident. Par exemple, remplacer $\{ N_t < n_t \}$ par $\{ N_t < n_t \} \cap \mathcal{V}$. 

We know from Fact \ref{bigevents} that the contributions to the local time and to the time spent of the successive valleys are approximated by the \textit{iid} sequence $(e_j S_j^t, e_j S_j^t R_j^t)_{j \geq 1}$. Since we deal with extreme values of the local time, we need to know the right tails of the distributions of $e_j S_j^t$ and of $e_j S_j^t R_j^t$. We also need informations about the extreme values of $R_j^t$. Recall the definition of $\mathcal{R}$ in the beginning of Section \ref{results}. We have the following fact from \cite{caslevyvech}: 

\begin{fact} \label{queueiid}
%Assume that the hypotheses of Theorems \ref{limsupkappa<1} and \ref{liminfkappa<1} are satisfied. 
Fix $\eta \in ]0, 1/3[$ and let $\mathcal{C}'$ be the constant in Lemma 4.16 of \cite{caslevyvech}. We have
\begin{align}
\underset{t \rightarrow +\infty}{\lim} & \sup_{x \in \left [ e^{-(1-2 \eta)\phi(t)}, + \infty \right [} \ \left | x^{\kappa} e^{\kappa \phi(t)} \mathbb{P} \left ( e_1 S_1^t /t > x \right ) - \mathcal{C}' \right | & = 0, \label{cvmeasure7.1} \\
\underset{t \rightarrow +\infty}{\lim} & \sup_{y \in \left [ e^{-(1-3\eta)\phi(t)}, + \infty \right [} \ \left | y^{\kappa} e^{\kappa \phi(t)} \mathbb{P} \left ( e_1 S_1^t R_1^t /t > y \right ) - \mathcal{C}' \mathbb{E}\left [ \mathcal{R}^{\kappa} \right ] \right | & = 0. \label{cvmeasure9.1}
\end{align}

$( R_1^t)_{t > 0}$ converges in distribution to $\mathcal{R}$ and there exists a positive $\lambda_0$ such that
\begin{eqnarray}
\forall \lambda < \lambda_0, \ \ \ \mathbb{E} \left [ e^{\lambda R_1^t} \right ] \underset{t \rightarrow +\infty}{\longrightarrow}  \mathbb{E} \left [ e^{\lambda \mathcal{R}} \right ], \label{cvrlapl}
\end{eqnarray}
where the above quantities are all finite. This entails the convergence of the moments of any positive order of $R_1^t$ to those of $\mathcal{R}$ when $t$ goes to infinity. 
\end{fact}

Finally, let us state a general lemma about the diffusion $X$: 

\begin{lemme} \label{infsupdiff}

Let $q$ be the constant in Theorem 1.4 of \cite{caslevyvech}. There are positive constants $c, K$ such that for all $r$ and $t$ large enough, 
\begin{align}
\mathbb{P} \left ( \sup_{[0, t]} X \geq 2 t^{\kappa} e^{\kappa \delta (\log (\log (t)))^{\omega}} /q \right ) & \leq e^{-c h_t}, \label{majodiff} \\
\mathbb{P} \left ( X(t) \leq t^{\kappa} e^{(\rho - \kappa) (\log (\log (t)))^{\omega}} /2q \right ) & \leq e^{-c \phi(t)}, \label{minondiff} \\
\mathbb{P} \left ( \inf_{[0, +\infty[} X \leq -r \right ) & \leq 3 r^{-1}, \label{minoinfndiff} \\
\mathbb{P} \left ( \sup_{]-\infty, 0]} \mathcal{L}_X(+\infty, .) > r \right ) & \leq K r^{-\kappa/(2+\kappa)}. \label{queuetlneg}
\end{align}

\end{lemme}

\subsection{Decomposition of the diffusion into independent parts} \label{decompindeppart}

A necessary point for our proofs is to give a decomposition of the trajectory of $X$ that makes independence appear in order to apply the Borel-Cantelli Lemma. Let us fix $a > 1$ and define the sequences $t_n := e^{n^a}$, $u_n := e^{\kappa (n^a - 2 a n^{a-1}/3)}$, $v_n := e^{\kappa (n^a - a n^{a-1}/3)}$. Let $X^n := X(H(v_n) + .)$, the diffusion shifted by the hitting time of $v_n$ and $T_n := \min \{ t_n, \tau (X^n, u_n), \tau (X^n, u_{n+1}) \}$. Note that from the Markov property for $X$ at time $H(v_n)$ and the stationarity of the increments of $V$, $X^n - v_n$ is equal in law to $X$ under the annealed probability $\mathbb{P}$. Let $n_0$ be large enough so that $u_n \leq v_n \leq u_{n+1}$ for all $n \geq n_0$. Clearly the sequence of processes $(X^n (t), \ 0 \leq t \leq T_n)_{n \geq n_0}$ is independent. We define the events
%\[ \mathcal{C}^1_n := \left \{ \inf_{[H(v_n), T_n]} X > u_n \right \}, \ \ \ \mathcal{C}^2_n := \left \{ \sup_{[H(v_n), T_n]} X < u_{n+1} \right \} \]
\[ \mathcal{C}_n := \left \{ T_n = t_n \right \} \ \ \ \text{and} \ \ \ \mathcal{D}_n := \left \{ H(v_n) < t_n/n \right \}. \]
The idea is that, since the sequence $(X^n (t), \ 0 \leq t \leq T_n)_{n \geq n_0}$ is independent, intersecting a sequence $(\mathcal{B}_n)_{n \geq n_0}$ of interesting events (where each event $\mathcal{B}_n$ only depends on $(X^n (t), \ 0 \leq t \leq t_n)$) with $\mathcal{C}_n$ will result in $(\mathcal{B}_n \cap \mathcal{C}_n)_{n \geq n_0}$, a sequence of independent events. Since $X^n$ is $X$ shifted by $H(v_n)$, the event $\mathcal{D}_n$ is useful to neglect the time shift when dealing with the renormalization of the local time. We will need the following lemma: 
\begin{lemme} \label{lemprindep}
\begin{eqnarray}
\sum_{n \geq 1} \mathbb{P} \left ( \overline{\mathcal{C}_n} \right ) < +\infty, \label{inegindep3}
\end{eqnarray}
and
\begin{eqnarray}
\sum_{n \geq 1} \mathbb{P} \left ( \overline{\mathcal{D}_n} \right ) < +\infty. \label{inegindep4}
\end{eqnarray}
\end{lemme}

\begin{proof}
Here again, let $q$ be the constant in Theorem 1.4 of \cite{caslevyvech}. First, notice from the definitions of $t_n$, $u_n$ and $v_n$ that for all $n$ large enough, 
\begin{eqnarray}
2 t_n^{\kappa} e^{\kappa \delta (\log (\log (t_n)))^{\omega}} /q < u_{n+1} - v_n, \label{inegindep1}
\end{eqnarray}
and
\begin{eqnarray}
(t_n/n)^{\kappa} e^{(\rho - \kappa) (\log (\log (t_n/n)))^{\omega}} /2q > v_n. \label{inegindep2}
\end{eqnarray}
From the definition of $\mathcal{C}_n$ and the Markov property applied to $X$ at time $H(v_n)$ we have
\begin{align*}
\mathbb{P} \left ( \overline{\mathcal{C}_n} \right ) & \leq \mathbb{P} \left ( \tau (X^n, u_{n+1}) < t_n \right ) + \mathbb{P} \left ( \tau (X^n, u_{n}) < +\infty \right ) \\
& \leq \mathbb{P} \left ( \tau (X, u_{n+1} - v_n) < t_n \right ) + \mathbb{P} \left ( \tau (X, u_{n} - v_n) < +\infty \right ) \\
& \leq \mathbb{P} \left ( \sup_{[0, t_n]} X \geq u_{n+1} - v_n \right ) + \mathbb{P} \left ( \inf_{[0, +\infty[} X \leq u_n - v_n) \right ) \\
& \leq e^{- c h_{t_n}} + 3/(v_n - u_n)
\end{align*}
where $c$ is the constant in Lemma \ref{infsupdiff}. The last inequality is true for $n$ large enough and comes, for the first term, from the combination of \eqref{inegindep1} and \eqref{majodiff}, and, for the second term, from \eqref{minoinfndiff}. Recall that $e^{-h_{t_n}} \approx e^{-\log(t_n)} = e^{- n^a}$. We thus deduce \eqref{inegindep3}. 

From the definition of $\mathcal{D}_n$, \eqref{inegindep2} and \eqref{minondiff} we have for all $n$ large enough, 
\[ \mathbb{P} \left ( \overline{\mathcal{D}_n} \right ) \leq \mathbb{P} \left ( X(t_n /n) \leq v_n \right ) \leq \mathbb{P} \left ( X(t_n /n) \leq (t_n/n)^{\kappa} e^{(\rho - \kappa) (\log (\log (t_n/n)))^{\omega}} /2q \right ) \leq e^{-c \phi(t_n/n)}, \]
where $c$ is the constant in Lemma \ref{infsupdiff}. Since $e^{-c \phi(t_n/n)} = e^{- c (\log ( \log (t_n / n) ))^{\omega}} \approx e^{- c a^{\omega} (\log ( n ))^{\omega}}$ we deduce \eqref{inegindep4}. 

\end{proof}

\subsection{The $\limsup$}

We study the asymptotic of the quantity $\mathbb{P} (\mathcal{L}^*_X(t)/t \geq x_t)$. Recall that $x_t$ is defined in \eqref{defxtlim} where $D > 0$ and $\mu \in ]1, 2]$ are fixed constants. In all this subsection the parameter $\omega$ in \eqref{paramtaillevallees} is fixed in $]1, \mu[$. We have: 

\begin{prop} \label{approxdutl}
There is a positive constant $c$ such that for all $a > 1$ and $t$ large enough we have, 
\begin{align*}
\mathbb{P} \left ( Y_1^{\natural, t} \left ( Y_2^{-1, t}(1/a) - \right ) \geq a \ x_t \right ) - e^{-c \phi(t)} & \leq \mathbb{P} \left ( \mathcal{L}^*_X(t) \geq t x_t \right ) \\
& \leq \mathbb{P} \left ( Y_1^{\natural, t} \left ( Y_2^{-1, t}(a) - \right ) \geq x_t / a \right ) + \mathbb{P} \left ( R_1^t \leq \frac{a}{x_t} \right ) + e^{-c \phi(t)}. 
\end{align*}
%\[ \mathbb{P} \left ( Y_1^{\natural, t} \left ( Y_2^{-1, t}(1/a) - \right ) \geq x_t \right ) - ... \leq \mathbb{P} \left ( \mathcal{L}^*_X(t) \geq t x_t \right ), \]
%and 
%\[ \mathbb{P} \left ( \mathcal{L}^*_X(t) \geq t x_t \right ) \leq \mathbb{P} \left ( Y_1^{\natural, t} \left ( Y_2^{-1, t}(a) - \right ) \geq x_t \right ) + ... + ..., \]
%where $\tilde x_t \underset{t \rightarrow +\infty}{\sim} \hat x_t \underset{t \rightarrow +\infty}{\sim} x_t$. GARDER $x_t$ CAR LE $a$ ABSORBE LA CORRECTION
\end{prop}

Note that the functional of $(Y_1^t, Y_2^t)$ involved in this proposition is $Y_1^{\natural, t} ( Y_2^{-1, t}(.)-)$, which represents the supremum of the local time before, and not including, the last valley (see the discussion after Theorem 1.1 of \cite{caslevyvech}). Even though, as we see in Proposition 4.1 of \cite{caslevyvech}, the distribution function of $\mathcal{L}^*_X(t)/t$ involves a complex functional of $(Y_1^t, Y_2^t)$ that represents the last valley (together with the functional $Y_1^{\natural, t} ( Y_2^{-1, t}(.)-)$ that represents the previous valleys), Proposition \ref{approxdutl} says that the right tail of this distribution function does not involves the last valley. Before proving this proposition we prove some lemmas. 

\begin{lemme} \label{cvntlp}

There is a positive constant $C$ such that for $t$ large enough, 
\[ \sum_{k=1}^{n_t} \mathbb{P} \left ( H(\tilde m_k)/t \leq 1 \right ) \leq C e^{\kappa \phi(t)}. \]

\end{lemme}

\begin{proof}
Since for all $k \geq 1$ we have almost surely $H(\tilde m_k) \geq \sum_{j=1}^{k-1} H(\tilde L_j)-H(\tilde m_j)$, we deduce that 
\begin{align*}
\sum_{k=1}^{n_t} \mathbb{P} \left ( H(\tilde m_k)/t \leq 1 \right ) & \leq \sum_{k=1}^{n_t} \mathbb{P} \left ( \sum_{j=1}^{k-1} H(\tilde L_j)-H(\tilde m_j)/t \leq 1 \right ) \\
& \leq \sum_{k=1}^{n_t} \mathbb{P} \left ( \sum_{j=1}^{k-1} (1-e^{- \tilde c h_t}) e_j S_j^t R_j^t/t \leq 1 \right ) + n_t \mathbb{P} \left ( \overline{\mathcal{E}^7_t} \right ) \\
& \leq e^{(1-e^{- \tilde c h_t})^{-1}} \sum_{k=1}^{n_t} \mathbb{E} \left [ e^{-\sum_{j=1}^{k-1} e_j S_j^t R_j^t/t} \right ] + n_t e^{- L h_t} \\
& \leq e^{(1-e^{- \tilde c h_t})^{-1}} \sum_{k=1}^{n_t} \left ( \mathbb{E} \left [ e^{- e_1 S_1^t R_1^t/t} \right ] \right )^{k-1} + e^{(1+\delta) \phi(t) - L h_t}. 
\end{align*}
In the above we have used  the definition of $\mathcal{E}^7_t$, Markov's inequality, \eqref{bigeventsleq}, the fact that the sequence $(e_j S_j^t R_j^t)_{j \geq 1}$ is \textit{iid} and the definition of $n_t$. For $t$ large enough we thus get
\begin{align}
\sum_{k=1}^{n_t} \mathbb{P} \left ( H(\tilde m_k)/t \leq 1 \right ) & \leq e^{2} \sum_{k=1}^{+\infty} \left ( \mathbb{E} \left [ e^{- e_1 S_1^t R_1^t/t} \right ] \right )^{k-1} + e^{- L h_t /2} \nonumber \\
& = \frac{e^{2}}{1 - \mathbb{E} \left [ e^{- e_1 S_1^t R_1^t/t} \right ]} + e^{- L h_t /2}. \label{cvntlp1}
\end{align}
Then, 
\begin{align*}
1 - \mathbb{E} \left [ e^{- e_1 S_1^t R_1^t/t} \right ] & = \int_0^{+\infty} e^{-u} \mathbb{P} \left ( e_1 S_1^t R_1^t/t > u \right ) du \\
& \geq \int_{e^{- \phi(t) /2}}^{+\infty} e^{-u} \mathbb{P} \left ( e_1 S_1^t R_1^t/t > u \right ) du \\
& = e^{-\kappa \phi(t)} \int_{e^{- \phi(t)/2}}^{+\infty} u^{-\kappa} e^{-u} \left ( u^{\kappa} e^{\kappa \phi(t)} \mathbb{P} \left ( e_1 S_1^t R_1^t/t > u \right ) \right ) du. 
\end{align*}
We can now use \eqref{cvmeasure9.1} with $\eta = 1/6$ and we get that for all $t$ large enough 
%\[ 1 - \mathbb{E} \left [ e^{- e_1 S_1^t R_1^t/t} \right ] \geq \frac{\mathcal{C'} \mathbb{E} [\mathcal{R}^{\kappa}]}{2} e^{-\kappa \phi(t)} \int_{e^{- \phi(t) /2}}^{+\infty} u^{-\kappa} e^{-u} du \underset{t \rightarrow +\infty}{\sim} \frac{\mathcal{C'} \mathbb{E} [\mathcal{R}^{\kappa}]}{2} e^{-\kappa \phi(t)} \int_{0}^{+\infty} u^{-\kappa} e^{-u} du. \]
\[ 1 - \mathbb{E} \left [ e^{- e_1 S_1^t R_1^t/t} \right ] \geq \frac{\mathcal{C'} \mathbb{E} [\mathcal{R}^{\kappa}]}{2} e^{-\kappa \phi(t)} \int_{e^{- \phi(t) /2}}^{+\infty} u^{-\kappa} e^{-u} du \underset{t \rightarrow +\infty}{\sim} \frac{\mathcal{C'} \mathbb{E} [\mathcal{R}^{\kappa}] \Gamma(1-\kappa)}{2} e^{-\kappa \phi(t)}. \]
Putting into \eqref{cvntlp1}, we get the result for $t$ large enough. 

\end{proof}

We now link the asymptotic of $\mathbb{P} ( R_1^t \leq a/x_t )$ with the left tail of $I(V^{\uparrow})$. We have to make a distinction between the case where $V$ possesses negative jumps and the case where $V$ possesses no negative jumps, that is, $V$ is the $\kappa$-drifted Brownian motion. $R_1^t$ is an exponential functional of the bottom of the first valley. In the first case, due to the jumps, the left side of the bottom of the valley can be neglected, so only the right side matters. In the second case, both sides have the same law, so both have to be taken in consideration in the left tail of $R_1^t$. 

\begin{lemme} \label{approxder2cas}
Let $z_t$ go to infinity with $t$ and satisfying $(\log(z_t))^2 << h_t$. 

\begin{itemize}
\item Assume that $V$ possesses negative jumps, then there is a positive constant $c$ such that for any $a > 1$ and $t$ large enough, 
\begin{eqnarray}
e^{ -c  \left ( \log ( (1 - 1/ a  ) / z_t )\right )^2} \mathbb{P} \left ( I(V^{\uparrow}) \leq 1/ a   z_t \right )\leq \mathbb{P} \left ( R_1^t \leq 1/z_t \right ) \leq 2 \mathbb{P} \left ( I(V^{\uparrow}) \leq a   /z_t \right ). \label{approxder2caslevy}
\end{eqnarray}
\item Assume now that $V = W_{\kappa}$, the $\kappa$-drifted Brownian motion, then for $t$ large enough, 
\begin{eqnarray}
\mathbb{P} \left ( \mathcal{R} \leq 1/z_t \right ) - 2 e^{- \delta \kappa h_t /3} \leq \mathbb{P} \left ( R_1^t \leq 1/z_t \right ) \leq 2 \mathbb{P} \left ( \mathcal{R} \leq a   /z_t \right ) + 2 e^{- \delta \kappa h_t /3}, \label{approxder2casmdb}
\end{eqnarray}
where $\mathcal{R}$ is defined in the beginning of Section \ref{results}. 
\end{itemize}

\end{lemme}

\begin{proof}

We first assume that $V$ possesses negative jumps. Recall that $R_1^t = \int_{\tilde{\tau}_1^-(h_t/2)}^{ \tilde{\tau}_1^+(h_t/2)} e^{-\tilde V^{(1)}(u)}  \textnormal{d}u$ so, using the equality in law between $(\tilde V^{(i)}(\tilde m_i+x),\  0\leq x \leq \tilde \tau_i(h)- \tilde m_i)$ and $(V^{\uparrow}(x), \ 0 \leq x \leq \tau(V^{\uparrow}, h))$ given by Fact \ref{loidesval}, we get
\[ R_1^t \geq \int_{\tilde m_1}^{ \tilde{\tau}_1^+(h_t/2)} e^{-\tilde V^{(1)}(u)}  \textnormal{d}u \overset{\mathcal{L}}{=} \int_0^{\tau(V^{\uparrow}, h_t/2)} e^{- V^{\uparrow}(u)} du, \]
so
\begin{eqnarray}
\mathbb{P} \left ( R_1^t \leq 1/z_t \right ) \leq \mathbb{P} \left ( \int_0^{\tau(V^{\uparrow}, h_t/2)} e^{- V^{\uparrow}(u)} du \leq 1/z_t \right ). \label{approxr1}
\end{eqnarray}

According to Lemma \ref{fonctronque}, for some positive constant $c$ and $t$ large enough, $(1-e^{-ch_t}) \times \mathbb{P} \left (  \int_0^{\tau(V^{\uparrow}, h_t/2)} e^{- V^{\uparrow}(u)} du \leq 1/z_t \right )$ is less than
\begin{align}
\mathbb{P} \left ( \int_0^{\tau(V^{\uparrow}, h_t/2)} e^{- V^{\uparrow}(u)} du \leq 1/z_t \right ) & \times \mathbb{P} \left ( \int_{\tau(V^{\uparrow}, h_t/2)}^{+\infty} e^{- V^{\uparrow}(u)} du \leq e^{-h_t /4} \right ) \nonumber \\
& \leq \mathbb{P} \left ( I(V^{\uparrow}) \leq 1/z_t + e^{-h_t /4} \right ). \label{approxr1.1}
\end{align}

Combining with \eqref{approxr1} we get that for $t$ large enough, 
\[ \mathbb{P} \left ( R_1^t \leq 1/z_t \right ) \leq 2 \mathbb{P} \left ( I(V^{\uparrow}) \leq 1/z_t + e^{-h_t /4} \right ) \leq 2 \mathbb{P} \left ( I(V^{\uparrow}) \leq a   /z_t \right ), \]
%\begin{eqnarray}
%\mathbb{P} \left ( R_1^t \leq 1/z_t \right ) \leq 2 \mathbb{P} \left ( I(V^{\uparrow}) \leq 1/z_t + e^{-h_t /4} \right ) \leq 2 \mathbb{P} \left ( I(V^{\uparrow}) \leq a   /z_t \right ), \label{approxr2}
%\end{eqnarray}
because $e^{-h_t /4} \leq (a   -1)/z_t$ for large $t$ (thanks to the hypothesis $(\log(z_t))^2 << h_t$). This is the asserted upper bound in \eqref{approxder2caslevy}. 
%\leq \int_0^{+\infty} e^{- V^{\uparrow}(x)} dx = I(V^{\uparrow})

On the other hand, using the independence between the left and right parts of the valleys (given by Fact \ref{loidesval}), we get
\begin{align}
\mathbb{P} \left ( R_1^t \leq 1/z_t \right ) \geq \mathbb{P} \left ( \int_{\tilde m_1}^{ \tilde{\tau}_1^+(h_t/2)} e^{-\tilde V^{(1)}(u)}  \textnormal{d}u \leq 1 / a   z_t \right ) \times \mathbb{P} \left ( \int_{\tilde{\tau}_1^-(h_t/2)}^{\tilde m_1} e^{-\tilde V^{(1)}(u)}  \textnormal{d}u \leq (1 - 1/ a  ) / z_t \right ). \label{approxr3}
\end{align}

From Fact \ref{loidesval}, we get that the first factor equals 
\[ \mathbb{P} \left ( \int_0^{\tau(V^{\uparrow}, h_t/2)} e^{- V^{\uparrow}(u)} du \leq 1 / a   z_t \right ) \geq  \mathbb{P} \left ( I(V^{\uparrow}) \leq 1 / a   z_t \right ), \]
while the second factor is more than
\begin{align*}
& \mathbb{E} \left [ \frac{c_{h_t}}{1-e^{-\kappa \hat{V}^{\uparrow}(\tau(\hat{V}^{\uparrow}, h_t+))}} ; \int_{0}^{\tau (\hat{V}^{\uparrow}, h_t/2+)} e^{-\hat{V}^{\uparrow}(u)}  \textnormal{d}u \leq (1 - 1/ a  ) / z_t \right ] - 2 e^{- \delta \kappa h_t /3} \\
\geq & c_{h_t} \mathbb{P} \left ( \int_{0}^{+\infty} e^{-\hat{V}^{\uparrow}(u)}  \textnormal{d}u \leq (1 - 1/ a  ) / z_t \right ) - 2 e^{- \delta \kappa h_t /3}. 
\end{align*}

Then, $c_{h_t} \geq c_1$ when $h_t \geq 1$ so, putting in \eqref{approxr3}, we get that for $t$ large enough, 
\begin{eqnarray}
\mathbb{P} \left ( R_1^t \leq 1 / z_t \right ) \geq \mathbb{P} \left ( I(V^{\uparrow}) \leq 1 / a   z_t \right ) \times \left( c_1 \mathbb{P} \left ( I(\hat{V}^{\uparrow}) \leq (1 - 1/ a  ) / z_t \right ) - 2 e^{- \delta \kappa h_t /3} \right ). \label{approxr4}
\end{eqnarray}

According to the combination of the first point of Theorem 1.15 of \cite{foncexpovech} with the second point of Theorem 1.16 of \cite{foncexpovech}, there is a positive constant $c$ such that for $t$ large enough, 
\[ \mathbb{P} \left ( I(\hat{V}^{\uparrow}) \leq (1 - 1/ a  ) / z_t \right ) \geq e^{ -c  \left ( \log ( (1 - 1/ a  ) / z_t )\right )^2}. \]
Thanks to the hypothesis $(\log(z_t))^2 << h_t$ we deduce that, for $c$ decreased a little, the second factor in the right hand side of \eqref{approxr4} is more than $e^{ -c  \left ( \log ( (1 - 1/ a  ) / z_t )\right )^2}$. This yields the lower bound in \eqref{approxder2caslevy}. 
%so in particular, when $t$ is large enough, the second factor in the right hand side of \eqref{approxr4} is more than $1/z_t$, so
%\[ \mathbb{P} \left ( R_1^t \leq 1 / z_t \right ) \geq \frac{1}{z_t} \mathbb{P} \left ( I(V^{\uparrow}) \leq 1 / a   z_t \right ). \]
%Combining with \eqref{approxr2} we get \eqref{approxder2caslevy}. 

We now consider the case where $V$ is the $\kappa$-drifted Brownian motion $W_{\kappa}$. Let $Z_1$ and $Z_2$ be two independent versions of the process $W_{\kappa}^{\uparrow}$. Since $W_{\kappa}$ has no jumps, the density of the process $P^{(2)}$ in Fact \ref{loidesval} is almost surely constant so $P^{(2)}$ is equal in law to $(\hat{W_{\kappa}}^{\uparrow}(x), \ 0 \leq x \leq \tau(\hat{W_{\kappa}}^{\uparrow}, h_t+)) = (W_{\kappa}^{\uparrow}(x), \ 0 \leq x \leq \tau(W_{\kappa}^{\uparrow}, h_t))$ (the last equality comes from the fact that $\hat{W_{\kappa}}^{\uparrow} = W_{\kappa}^{\uparrow}$ and $W_{\kappa}$ is continuous). Combining this with \eqref{vartot} and the equality in law between $(\tilde V^{(i)}(\tilde m_i+x),\  0\leq x \leq \tilde \tau_i(h)- \tilde m_i)$ and $(V^{\uparrow}(x), \ 0 \leq x \leq \tau(V^{\uparrow}, h))$ (both are from Fact \ref{loidesval}), we get
\begin{align}
\mathbb{P} \left ( R_1^t \leq 1/z_t \right ) & \leq \mathbb{P} \left ( \int_0^{\tau(Z_1, h_t/2)} e^{- Z_1(u)} du + \int_0^{\tau(Z_2, h_t/2)} e^{- Z_2(u)} du \leq 1/z_t \right ) + 2 e^{- \delta \kappa h_t /3}, \label{approxr5} \\
\mathbb{P} \left ( R_1^t \leq 1/z_t \right ) & \geq \mathbb{P} \left ( \int_0^{\tau(Z_1, h_t/2)} e^{- Z_1(u)} du + \int_0^{\tau(Z_2, h_t/2)} e^{- Z_2(u)} du \leq 1/z_t \right ) - 2 e^{- \delta \kappa h_t /3}. \label{approxr6}
\end{align}
Reasoning as in \eqref{approxr1.1} we get for some positive constant $c$: 
\begin{align*}
(1-e^{-ch_t})^2 & \mathbb{P} \left ( \int_0^{\tau(Z_1, h_t/2)} e^{- Z_1(u)} du + \int_0^{\tau(Z_2, h_t/2)} e^{- Z_2(u)} du \leq 1/z_t \right ) \\
\leq & \mathbb{P} \left ( I(Z_1) + I(Z_2) \leq 1/z_t + 2 e^{-h_t /4} \right ). 
\end{align*}
Combining with \eqref{approxr5} we get that for $t$ large enough, 
\[ \mathbb{P} \left ( R_1^t \leq 1/z_t \right ) \leq 2 \mathbb{P} \left ( I(Z_1) + I(Z_2) \leq 1/z_t + 2 e^{-h_t /4} \right ) + 2 e^{- \delta \kappa h_t /3} \leq 2 \mathbb{P} \left ( \mathcal{R} \leq a   /z_t \right ) + 2 e^{- \delta \kappa h_t /3}, \]
because $2 e^{-h_t /4} \leq (a   -1)/z_t$ for large $t$ and because, form the definitions of $\mathcal{R}$, $Z_1$ and $Z_2$ we have $\mathcal{R} \overset{\mathcal{L}}{=}  I(Z_1) + I(Z_2)$. The above inequality is the asserted upper bound in \eqref{approxder2casmdb}. 

On the other hand, we have trivially 
\[ \mathbb{P} \left ( \int_0^{\tau(Z_1, h_t/2)} e^{- Z_1(u)} du + \int_0^{\tau(Z_2, h_t/2)} e^{- Z_2(u)} du \leq 1/z_t \right ) \geq \mathbb{P} \left ( I(Z_1) + I(Z_2) \leq 1/z_t \right ), \]
and combining with \eqref{approxr6} we get that for $t$ large enough, 
\[ \mathbb{P} \left ( R_1^t \leq 1/z_t \right ) \geq \mathbb{P} \left ( I(Z_1) + I(Z_2) \leq 1/z_t \right ) - 2 e^{- \delta \kappa h_t /3} = \mathbb{P} \left ( \mathcal{R} \leq 1/z_t \right ) - 2 e^{- \delta \kappa h_t /3}. \]
This yields the lower bound in \eqref{approxder2casmdb}. 

\end{proof}

The next lemma {is crucial for our analysis and shows that, surprisingly, the contribution of the last valley can be neglected when we deal with the extremely large values of the local time, even though the contribution of the last valley appears in the limit distribution of $\mathcal{L}^*_X(t)/t$. The removing of the last valley allowed by this lemma is a key point in proving (in Proposition \ref{approxdutl}) that $\mathbb{P} (\mathcal{L}^*_X(t)/t \geq x_t)$ is very closely related to simpler quantities of the type of $\mathbb{P} ( Y_1^{\natural, t} ( Y_2^{-1, t}(1) - ) \geq x_t )$, for which we will be able to describe the precise asymptotic (in Proposition \ref{queueloilimsupdiscret}). Therefore, having this lemma is the reason why we get precise results for the $\limsup$, and unfortunately a counterpart does not hold for the $\liminf$, as we will see in Section \ref{preuveliminf}, which is why we get slightly less precise results for the $\liminf$.} Recall the definition of $\mathcal{D}_j$ in \eqref{defdj}. We have: 

\begin{lemme} \label{neglectlastvalley}

There are positive constants $C$ and $c$ such that for all $u > 1$ and $t$ large enough, 
\[ \mathbb{P} \left ( \sup_{y \in \mathcal{D}_{N_t}} \left ( \mathcal{L}_X (t,y) - \mathcal{L}_X (H(\tilde m_{N_t}),y) \right ) \geq t x_t \right ) \leq \frac{C}{x_t^{\kappa}} \mathbb{P} \left ( R_1^t \leq \frac{u}{x_t} \right ) + e^{-c \phi(t)}. \]

\end{lemme}

\begin{proof}

We fix $v \in \mathcal{V}$, a realization of the environment. Let us define
\begin{align*}
\mathcal{E}_t(v, k, z) := & \left \{ \sup_{y \in \mathcal{D}_k} \mathcal{L}_{X_{\tilde m_k}}(t(1-z), y) \geq t x_t, \right. \\
& \left. H_{X_{\tilde m_k}}(\tilde m_{k+1}) \geq t(1-z), H_{X_{\tilde m_k}}(\tilde L_k) < H_{X_{\tilde m_k}}(\tilde L_{k-1}) \right \}. 
\end{align*}
%and $\nu_t(v, k, z) := P^v(\mathcal{E}_t(v, k, z))$. 
We have 
%\begin{align}
%& P^v \left ( \sup_{\mathcal{D}_{N_t}} \left ( \mathcal{L}_X (t,.) - \mathcal{L}_X (H(\tilde m_{N_t}),.) \right ) \geq t x_t, \ N_t < n_t, \mathcal{E}^1_t \right ) \nonumber \\
%\leq & \sum_{k=1}^{n_t} \int_0^1 \nu_t(v, k, z) \times P^v \left ( H(\tilde m_k)/t \in dz \right ). \label{yeswecan2.1}
%\end{align}
\begin{align}
& P^v \left ( \sup_{y \in \mathcal{D}_{N_t}} \left ( \mathcal{L}_X (t,y) - \mathcal{L}_X (H(\tilde m_{N_t}),y) \right ) \geq t x_t, \ N_t < n_t, \mathcal{E}^1_t \right ) \nonumber \\
\leq & \mathds{1}_{v \in \mathcal{G}_t} \sum_{k=1}^{n_t} \int_0^1 P^v \left ( \mathcal{E}_t(v, k, z), \ H(\tilde m_k)/t \in dz \right ) + \mathds{1}_{v \in \overline{\mathcal{G}_t}} \nonumber \\
\leq & \mathds{1}_{v \in \mathcal{G}_t} \sum_{k=1}^{n_t} \int_0^{1-4/\log(h_t)} P^v \left ( \mathcal{E}_t(v, k, z), \ H(\tilde m_k)/t \in dz \right ) \nonumber \\ 
+ & \mathds{1}_{v \in \mathcal{G}_t} \sum_{k=1}^{n_t} \int_{1-4/\log(h_t)}^1 P^v \left ( \mathcal{E}_t(v, k, z), \ H(\tilde m_k)/t \in dz \right ) + \mathds{1}_{v \in \overline{\mathcal{G}_t}}. \label{yeswecan2.1}
\end{align}

The fact that the sum stops at $n_t$ comes from $N_t < n_t$ together with the fact that $v \in \mathcal{G}_t \subset \mathcal{V}_t$. From the definitions of $\mathcal{E}_t(v, k, z)$, $\mathcal{E}^8_t(v, k, z)$, $\mathcal{E}^5_t$ and $\mathcal{E}^7_t$ we have, for $v \in \mathcal{G}_t$, 
%From the definition of $\mathcal{E}^8_t$, the fact that $H(\tilde m_{k+1}) - H(\tilde m_k) \geq t(1-z), H(\tilde m_{k+1}) - H(\tilde L_k) \leq t e^{-\phi(t) /2}$ and the definition of $\mathcal{E}^7_t$, we get
\begin{align*}
\mathcal{E}_t(v, k, z) \subset & \left \{ \frac{1-z}{R_k^t} \frac{(1+e^{- \tilde c h_t})}{(1-e^{- \tilde c h_t})} \geq x_t, \ (1+e^{- \tilde c h_t}) e_k S_k^t R_k^t \geq t \left (1-z - 2/\log(h_t) \right ) \right \} \\
& \cup \mathcal{E}^8_t(v, k, z) \cup \overline{\mathcal{E}^5_t} \cup \overline{\mathcal{E}^7_t}. 
\end{align*}
When $z \leq 1-4/\log(h_t)$ we have, on the big event in the right hand side, 
\begin{align*}
t(1-z)/2 \leq t(1-z - 2/\log(h_t)) & \leq (1+e^{- \tilde c h_t}) e_k S_k^t R_k^t \\
& \leq (1+e^{- \tilde c h_t})^2 (1- e^{- \tilde c h_t})^{-1} (1-z) e_k S_k^t / x_t \\
& \leq 2 (1-z) e_k S_k^t / x_t, 
\end{align*}
for $t$ large enough. Moreover, $(1-z)(1+e^{- \tilde c h_t})(1-e^{- \tilde c h_t})^{-1} \leq 1+e^{- \tilde c h_t / 2}$ for $t$ large enough. As a consequence, for $z \in [0, 1-4/\log(h_t)]$ and $t$ large enough, 
\[ \mathcal{E}_t(v, k, z) \subset \left \{ R_k^t \leq (1+e^{- \tilde c h_t/2}) / x_t, \ e_k S_k^t /t \geq x_t / 4 \right \} \cup \mathcal{E}^8_t(v, k, z) \cup \overline{\mathcal{E}^5_t} \cup \overline{\mathcal{E}^7_t}. \]
Note also that the sum over $k$ in \eqref{yeswecan2.1} corresponds to disjoint events (so it is actually the probability of a union of events). We thus get that $\mathds{1}_{v \in \mathcal{G}_t} \sum_{k=1}^{n_t} \int_0^{1-4/\log(h_t)} P^v ( \mathcal{E}_t(v, k, z), \ H(\tilde m_k)/t \in dz )$ is less than
\begin{align*}
& \sum_{k=1}^{n_t} \int_0^{1-4/\log(h_t)} P^v \left ( R_k^t \leq (1+e^{- \tilde c h_t/2}) / x_t, \ e_k S_k^t /t \geq x_t / 4, \ H(\tilde m_k)/t \in dz \right ) \\
+ & \mathds{1}_{v \in \mathcal{G}_t} P^v \left ( \overline{\mathcal{E}^5_t} \cup \overline{\mathcal{E}^7_t} \cup \cup_{k=1}^{n_t} \{ N_t \geq k, \ H(\tilde m_{k})/t \leq 1 - 4/\log(h_t) \} \cap \mathcal{E}^8_t(v, k, H(\tilde m_{k})/t) \right ). 
\end{align*}
Now, recall that $(S_k^t, R_k^t)$ only depends on $v$ and that, $v$ being fixed, $e_k$ belongs to the $\sigma$-field $\sigma \left ( X(t), t \geq H(\tilde m_k) \right )$. In other words, it only depends on the diffusion after time $H(\tilde m_k)$. On the other hand, $H(\tilde m_{k})$ is measurable with respect to the $\sigma$-field $\sigma \left ( X(t), 0 \leq t \leq H(\tilde m_{k}) \right )$. From the Markov property applied to 
%$V$ at $\tilde L_{k-1}$ and applied to 
$X$ at $H(\tilde m_k)$, we get that $H(\tilde m_k)$ is independent from $(e_k, S_k^t, R_k^t)$ so the above is less than
%Note that the first event in the right hand side does not depend on $z$. We thus get that $\sum_{k=1}^{n_t} \int_0^{1-4/\log(h_t)} P^v ( \mathcal{E}_t(v, k, z), \ H(\tilde m_k)/t \in dz )$ is less than
\begin{align*}
& \sum_{k=1}^{n_t} P^v \left ( R_k^t \leq (1+e^{- \tilde c h_t/2}) / x_t, \ e_k S_k^t /t \geq x_t / 4 \right ) \times P^v \left ( H(\tilde m_k)/t \leq 1 \right ) \\
+ & \mathds{1}_{v \in \mathcal{G}_t} P^v \left ( \overline{\mathcal{E}^5_t} \cup \overline{\mathcal{E}^7_t} \cup \cup_{k=1}^{n_t} \{ N_t \geq k, \ H(\tilde m_{k})/t \leq 1 - 4/\log(h_t) \} \cap \mathcal{E}^8_t(v, k, H(\tilde m_{k})/t) \right ) \\
\leq & \sum_{k=1}^{n_t} P^v \left ( R_k^t \leq (1+e^{- \tilde c h_t/2}) / x_t, \ e_k S_k^t /t \geq x_t / 4 \right ) \times P^v \left ( H(\tilde L_{k-1})/t \leq 1 \right ) + e^{-c \phi(t)}. 
\end{align*}
%\begin{align*}
%& \sum_{k=1}^{n_t} P^v \left ( \frac{1-z}{R_k^t} (1+ e^{- \tilde c h_t /2}) \geq x_t, \ e_k S_k^t /t \geq \frac{x_t}{2(1+e^{- \tilde c h_t /2})} \right ) \times P^v \left ( H(\tilde m_k)/t \leq 1 \right ) \\
%& + P^v ( \overline{\mathcal{E}^5_t} \cup \overline{\mathcal{E}^7_t} \cup \cup_{k=1}^{n_t} \mathcal{E}^8_t(v, k, z)) \\
%\leq & \sum_{k=1}^{n_t} P^v \left ( \frac{1-z}{R_k^t} (1+ e^{- \tilde c h_t /2}) \geq x_t, \ e_k S_k^t /t \geq \frac{x_t}{2(1+e^{- \tilde c h_t /2})} \right ) \times P^v \left ( H(\tilde L_{k-1})/t \leq 1 \right ) \\
%& + e^{-c \phi(t)/2}, 
%\end{align*}
In the above, $c$ is a positive constant and we have used the fact that $\tilde L_{k-1} \leq \tilde m_k$ for the first term and the fact that $v \in \mathcal{G}_t$ together with \eqref{bonenv1} and \eqref{majoe8} for the second term. In conclusion we get 
\begin{align}
& \mathds{1}_{v \in \mathcal{G}_t} \sum_{k=1}^{n_t} \int_0^{1-4/\log(h_t)} P^v ( \mathcal{E}_t(v, k, z), \ H(\tilde m_k)/t \in dz ) \nonumber \\
\leq & \sum_{k=1}^{n_t} P^v \left ( R_k^t \leq (1+e^{- \tilde c h_t/2}) / x_t, \ e_k S_k^t /t \geq x_t / 4 \right ) \times P^v \left ( H(\tilde L_{k-1})/t \leq 1 \right ) + e^{-c \phi(t)}. \label{inegquenched}
\end{align}
Now, note that the first factor in the above product only depends on $(\tilde v^{(k)} (x),\ \tilde L_{k-1} \leq x \leq \tilde L_{k} )$ (that is, only on $v$ shifted at $\tilde L_{k-1}$) while the second factor depends on $(v (x),\ x \leq \tilde L_{k-1} )$ (that is, on $v$ before $\tilde L_{k-1}$). $\tilde L_{k-1}$ is a stopping time for the L\'evy process $V$ of which $v$ is a fixed possible path. As a consequence, when we integrate \eqref{inegquenched} with respect to $v$ over $D(\mathbb{R}, \mathbb{R})$ equipped with the probability mesure $P$, we get that the two factors are independent so $E [\mathds{1}_{V \in \mathcal{G}_t} \sum_{k=1}^{n_t} \int_0^{1-4/\log(h_t)} P^V ( \mathcal{E}_t(V, k, z), \ H(\tilde m_k)/t \in dz ) ]$, is less than
\begin{align*}
& \sum_{k=1}^{n_t} \mathbb{P} \left ( R_k^t \leq (1+e^{- \tilde c h_t/2}) / x_t, \ e_k S_k^t /t \geq x_t / 4 \right ) \mathbb{P} \left ( H(\tilde L_{k-1})/t \leq 1 \right ) + e^{-c \phi(t)} \\ 
%+ P(V \notin \mathcal{G}_t) \\
\leq & \sum_{k=1}^{n_t} \mathbb{P} \left ( R_k^t \leq (1+e^{- \tilde c h_t/2}) / x_t \right ) \mathbb{P} \left ( e_k S_k^t /t \geq x_t / 4 \right ) \mathbb{P} \left ( H(\tilde L_{k-1})/t \leq 1 \right ) + e^{-c \phi(t)}. 
\end{align*}
%\begin{align*}
%\nu_t(v, k, z) & \leq P^v \left ( \frac{1-z}{R_k^t} (1+ ...) \geq x_t, \ e_k S_k^t /t \geq \frac{x_t}{2(1+...)} \right ) +\mathbb{P} \left ( \overline{\mathcal{E}^7_t} \right ) + \mathbb{P} \left ( \overline{\mathcal{E}^8_t} \right ) \\
%& \leq \mathbb{P} \left ( R_k^t \leq \frac{1+ ...}{x_t} \right ) \mathbb{P} \left ( e_k S_k^t /t \geq \frac{x_t}{2(1+...)} \right ) + e^{-c h_t}, 
%\end{align*}
For the equality we have used the independence between $R_k^t$ and $e_k S_k^t$. 
%and Lemma \ref{measofgoodenv}, and the constant $c$ has been suitably decreased. 
Since the sequence $(e_k, S_k^t, R_k^t)_{k \geq 1}$ is \textit{iid}, we deduce that $E [\mathds{1}_{V \in \mathcal{G}_t}\sum_{k=1}^{n_t} \int_0^{1-4/\log(h_t)} P^V ( \mathcal{E}_t(V, k, z), \ H(\tilde m_k)/t \in dz ) ]$ is less than
%\begin{align}
%& \mathbb{P} \left ( R_1^t \leq \frac{1+ \epsilon_t}{x_t} \right ) \left [ \mathbb{P} \left ( e_1 S_1^t /t \geq \frac{x_t}{2(1+\epsilon_t)(1+\epsilon_t)} \right ) \sum_{k=1}^{n_t} \mathbb{P} \left ( H(\tilde m_k)/t \leq 1 \right ) \right ] \nonumber \\
%+ & \mathbb{P} \left ( R_1^t \leq \frac{1+ \epsilon_t}{x_t} \right ) \left [ \sum_{k=1}^{n_t} \int_{1-2 \tilde v_t /t}^1 \mathbb{P} \left ( H(\tilde m_k)/t \in dz \right ) \right ] + n_t e^{-c h_t}. \label{yeswecan2.2}
%\end{align}
\[ \mathbb{P} \left ( R_1^t \leq \frac{1+ e^{- \tilde c h_t/2}}{x_t} \right ) \times \mathbb{P} \left ( e_1 S_1^t /t \geq x_t / 4 \right ) \sum_{k=1}^{n_t} \mathbb{P} \left ( H(\tilde L_{k-1})/t \leq 1 \right ) + e^{-c \phi(t)}. \]
Using \eqref{cvmeasure7.1} and Lemma \ref{cvntlp} to bound respectively the second and third factor, we get the existence of a positive constant $C$ such that for $t$ large enough, 
\begin{align}
E \left [ \mathds{1}_{V \in \mathcal{G}_t} \sum_{k=1}^{n_t} \int_0^{1-4/\log(h_t)} P^V ( \mathcal{E}_t(V, k, z), \ H(\tilde m_k)/t \in dz ) \right ] \leq \frac{C}{x_t^{\kappa}} \mathbb{P} \left ( R_1^t \leq \frac{1+ e^{- \tilde c h_t/2}}{x_t} \right ) + e^{-c \phi(t)}. \label{yeswecan2.3}
\end{align}

%From Fubini, the term $\mathbb{P} ( e_1 S_1^t /t \geq \frac{x_t}{2(1+\epsilon_t)(1+\epsilon_t)} ) \sum_{k=1}^{n_t} \mathbb{P} ( H(\tilde m_k)/t \leq 1 )$ is less than
%\[ e^{\kappa \phi(t)} \mathbb{P} \left ( e_1 S_1^t /t \geq \frac{x_t}{2(1+\epsilon_t)(1+\epsilon_t)} \right ) \times e^{-\kappa \phi(t)} \mathbb{E} \left [ N_t \right ] \ \underset{t \rightarrow +\infty}{\sim} \ \frac{\mathcal{C}}{x_t^{\kappa}}, \]
%where $\mathcal{C}$ is some positive constant, and the equivalence comes using Lemma 4.2 of A-D-V for the first factor and and Lemma \ref{cvnt}. We have thus obtained that 
%\begin{eqnarray}
%\sum_{k=1}^{n_t} \int_0^{1-2\tilde v_t /t} \nu_t(v, k, z) \times \mathbb{P} ( H(\tilde m_k)/t \in dz ) = \mathcal{O} \left ( \mathbb{P} \left ( R_1^t \leq \frac1{x_t} \right ) \right ). \label{yeswecan2.3}
%\end{eqnarray}

%For any $\eta > 0$, we have
%\begin{align*}
%\limsup_{t \rightarrow +\infty} \sum_{k=1}^{n_t} \int_{1-2 \tilde v_t /t}^1 \mathbb{P} \left ( H(\tilde m_k)/t \in dz \right ) & \leq \limsup_{t \rightarrow +\infty} \sum_{k=1}^{n_t} \int_{1-\eta}^1 \mathbb{P} \left ( H(\tilde m_k)/t \in dz \right ), \\
%& = \limsup_{t \rightarrow +\infty} \mathbb{E} \left [ \sum_{k=1}^{n_t} \mathds{1}_{H(\tilde m_k)/t \in [1-\eta, 1]} \right ], \\
%\end{align*}

It remains to study $E [\mathds{1}_{V \in \mathcal{G}_t} \sum_{k=1}^{n_t} \int_{1-4/\log(h_t)}^1 P^V ( \mathcal{E}_t(V, k, z), \ H(\tilde m_k)/t \in dz ) ]$. For $v \in \mathcal{G}_t$ and $z \in [1-4/\log(h_t), 1]$ we have, using the definitions of $\mathcal{E}_t(v, k, z)$ and $\mathcal{E}^8_t(v, k, z)$: 
\begin{align*}
\mathcal{E}_t(v, k, z) & \subset \left \{ \sup_{y \in \mathcal{D}_k} \mathcal{L}_{X_{\tilde m_k}}(4/\log(h_t), y) \geq t x_t, \ H_{X_{\tilde m_k}}(\tilde L_k) < H_{X_{\tilde m_k}}(\tilde L_{k-1}) \right \} \\
& \subset \left \{ \frac{4}{\log(h_t)} \frac{(1+e^{- \tilde c h_t})}{(1-e^{- \tilde c h_t})} \geq x_t R_k^t \right \} \cup \mathcal{E}^8_t(v, k, 1-4/\log(h_t)) \\
& \subset \left \{ 8/x_t \log(h_t) \geq R_k^t \right \} \cup \mathcal{E}^8_t(v, k, 1-4/\log(h_t)), 
\end{align*}
where, in the last inclusion, we used the fact that $(1+e^{- \tilde c h_t})/(1-e^{- \tilde c h_t}) \leq 2$ for large $t$. 
%We can prove, reasoning as in the proof of $(5.79)$ of A-D-V, that 
%\[ \mathbb{P} \left ( \frac{\sup_{y \in \mathcal{D}_k} \mathcal{L}_{X'}(4 t /\log(h_t), y)}{t} \leq \frac{4/\log(h_t)}{R_k^t} (1+ \epsilon_t) \right ) \geq 1 - e^{-c' h_t}, \]
%where $\epsilon_t := ...$ and $c'$ is some positive constant. 
%Since the event $\{ 8/x_t \log(h_t) \geq R_k^t \}$ does not depend on $z$ 
Recall that the sum over $k$ corresponds to disjoint events, we thus get that \\ $\mathds{1}_{v \in \mathcal{G}_t} \sum_{k=1}^{n_t} \int_{1-4/\log(h_t)}^1 P^v ( \mathcal{E}_t(v, k, z), \ H(\tilde m_k)/t \in dz )$ is less than
\begin{align*}
& \sum_{k=1}^{n_t} P^v \left ( R_k^t \leq 8/x_t \log(h_t), \ H(\tilde m_k)/t \in dz \right ) + \mathds{1}_{v \in \mathcal{G}_t} P^v ( \cup_{k=1}^{n_t} \mathcal{E}^8_t(v, k, 1-4/\log(h_t))) \\
\leq & \sum_{k=1}^{n_t} \mathds{1}_{R_k^t \leq 8/x_t \log(h_t)} \times P^v \left ( H(\tilde m_k)/t \leq 1 \right ) + \mathds{1}_{v \in \mathcal{G}_t} P^v ( \cup_{k=1}^{n_t} \mathcal{E}^8_t(v, k, 1-4/\log(h_t))) \\
\leq & \sum_{k=1}^{n_t} \mathds{1}_{R_k^t \leq 8/x_t \log(h_t)} \times P^v \left ( H(\tilde L_{k-1})/t \leq 1 \right ) + e^{-c \phi(t)}. 
\end{align*}
In the above, for the second inequality we have used the fact that $R_k^t$ only depends on $v$, and for the third we have used the fact that $\tilde L_{k-1} \leq \tilde m_k$ for the first term and \eqref{majoe8bis} for the second term. 
Here again, the first factor in the above product only depends on $(\tilde v^{(k)} (x),\ \tilde L_{k-1} \leq x \leq \tilde L_{k} )$ while the second factor depends on $(v (x),\ x \leq \tilde L_{k-1} )$. As a consequence, when we integrate the above inequality with respect to $v$ over $D(\mathbb{R}, \mathbb{R})$ equipped with the probability mesure $P$, we get that the two factors are independent so $E [\mathds{1}_{V \in \mathcal{G}_t} \sum_{k=1}^{n_t} \int_{1-4/\log(h_t)}^1 P^V ( \mathcal{E}_t(V, k, z), \ H(\tilde m_k)/t \in dz ) ]$ is less than
\begin{align*}
& \sum_{k=1}^{n_t} \mathbb{P} \left ( R_k^t \leq 8/x_t \log(h_t) \right ) \mathbb{P} \left ( H(\tilde L_{k-1})/t \leq 1 \right ) + e^{-c \phi(t)} \\
%+ P(V \notin \mathcal{G}_t) \\
= & \mathbb{P} \left ( R_1^t \leq 8/x_t \log(h_t) \right ) \sum_{k=1}^{n_t} \mathbb{P} \left ( H(\tilde L_{k-1})/t \leq 1 \right ) + e^{-c \phi(t)}, 
\end{align*}
where we have used the fact that the sequence $(R_k^t)_{k \geq 1}$ is \textit{iid}. 
%and Lemma \ref{measofgoodenv}, and where the constant $c$ has been suitably decreased. 
Note that, from the definitions of $x_t$ and $h_t$ in respectively \eqref{defxtlim} and \eqref{paramtaillevallees}, we have  $x_t \log(h_t) \sim K (\log (\log (t)))^{\mu}$. Therefore, using Lemmas \ref{approxder2cas} (with $z_t = x_t \log(h_t)/8$, $a = 2$) and \ref{cvntlp} to bound respectively the first and second factor we get that for $t$ large enough the above is less than
\[ C e^{\kappa \phi(t)} \mathbb{P} \left ( I(V^{\uparrow}) \leq 1 / K' (\log (\log (t)))^{\mu} \right ) + C e^{\kappa \phi(t) - \delta \kappa h_t /3} + e^{-c \phi(t)}, \]
for some positive constants $C$ and $K'$. The term $C e^{\kappa \phi(t) - \delta \kappa h_t /3}$ appears when $V = W_{\kappa}$ (because of \eqref{approxder2casmdb}) and it is not necessary otherwise, note that this term is ultimately less than $e^{-c \phi(t)}$. Combining with \eqref{majouniv} we get
\[ E \left [ \mathds{1}_{V \in \mathcal{G}_t} \sum_{k=1}^{n_t} \int_{1-4/\log(h_t)}^1 P^V ( \mathcal{E}_t(V, k, z), \ H(\tilde m_k)/t \in dz ) \right ] \leq C e^{\kappa \phi(t) - K'' (\log (\log (t)))^{\mu}} + 2 e^{-c \phi(t)}, \]
where $K''$ is a positive constant. Since we have chosen $\omega \in ]1, \mu[$ in \eqref{paramtaillevallees} we have $(\log (\log (t)))^{\mu} >> \phi(t)$ so for $t$ large enough the above inequality yields
\begin{eqnarray}
E \left [ \mathds{1}_{V \in \mathcal{G}_t} \sum_{k=1}^{n_t} \int_{1-4/\log(h_t)}^1 P^V ( \mathcal{E}_t(V, k, z), \ H(\tilde m_k)/t \in dz ) \right ] \leq e^{-c \phi(t)}, \label{yeswecan2.5}
\end{eqnarray}
where the constant $c$ has been suitably decreased. 
%so for $t$ large enough. 
%\begin{eqnarray}
%\sum_{k=1}^{n_t} \int_{1-2e^{-\phi(t) /2}}^1 \nu_t(v, k, z) \times \mathbb{P} ( H(\tilde m_k)/t \in dz ) \leq 2 n_t e^{-c' h_t} \leq e^{-c' h_t/2}. \label{yeswecan2.4}
%\end{eqnarray}
Now putting \eqref{yeswecan2.3} and \eqref{yeswecan2.5} into \eqref{yeswecan2.1} we get for some constant $c$ and $t$ large enough: 
\begin{align*}
\mathbb{P} \left ( \sup_{y \in \mathcal{D}_{N_t}} \left ( \mathcal{L}_X (t,y) - \mathcal{L}_X (H(\tilde m_{N_t}),y) \right ) \geq t x_t \right ) & \leq \frac{C}{x_t^{\kappa}} \mathbb{P} \left ( R_1^t \leq \frac{u}{x_t} \right ) + e^{-c \phi(t)} \\
& + \mathbb{P} \left ( N_t \geq n_t \right ) + \mathbb{P} \left ( \overline{\mathcal{E}^1_t} \right ) + P(V \notin \mathcal{G}_t). 
\end{align*}
%\[ \mathbb{P} \left ( \mathcal{L}_X(t, \tilde m_{N_t}) \geq t x_t, \ \mathcal{A}_t \right ) = \mathcal{O} \left ( \mathbb{P} \left ( R_1^t \leq \frac1{x_t} \right ) \right ) \]
Using \eqref{majonbvalleesvisit}, \eqref{bigeventsleq}, and Lemma \ref{measofgoodenv} we get the result for a suitably chosen constant $c$ and $t$ large enough. 

\end{proof}

\begin{proof} of Proposition \ref{approxdutl}

\textit{Upper bound}
\begin{align}
\mathbb{P} \left ( \mathcal{L}^*_X(t) \geq t x_t \right ) & \leq \mathbb{P} \left ( \mathcal{L}^*_X(t) \geq t x_t, \ V \in \mathcal{V}_t, \ N_t < n_t, \ \mathcal{E}^1_t, \mathcal{E}^2_t \right ) \nonumber \\
& + \mathbb{P} \left ( V \notin \mathcal{V}_t \right ) + \mathbb{P} \left ( N_t \geq n_t \right ) + \mathbb{P} \left ( \overline{\mathcal{E}^1_t} \right ) + \mathbb{P} \left ( \overline{\mathcal{E}^2_t} \right ). \label{lsub0}
\end{align}
The event $\{ N_t < n_t \} \cap \mathcal{E}^1_t$ ensures that for $j \leq N_t-1 < n_t$, $\tilde m_{j}$ is no longer reached after $H(\tilde L_{j})$ and $\tilde{L}_{j}$ is no longer reached after $H(\tilde m_{j+1})$, and the event $\{ V \in \mathcal{V}_t, \ N_t < n_t \}$ ensures that $t$ is larger than $H(\tilde m_{N_t}) \geq H(\tilde L_{N_t-1})$ and less than $H(\tilde m_{N_t+1})$. Therefore, for each $j \in \{1, ..., N_t - 1\}$ we have 
\begin{align*}
\forall x \in [\tilde{L}_{j-1}, \tilde{m}_j], \ \mathcal{L}_X (t,x) = \mathcal{L}_X (H(\tilde L_j),x) & = \left ( \mathcal{L}_X (H(\tilde L_j),x) - \mathcal{L}_X (H(\tilde m_{j}),x) \right ) \\
& + \left ( \mathcal{L}_X (H(\tilde m_{j}),x) - \mathcal{L}_X (H(\tilde L_{j-1}),x) \right ), \\
\forall x \in [\tilde{m}_j, \tilde{L}_{j}], \ \mathcal{L}_X (t,x) = \mathcal{L}_X (H(\tilde m_{j+1}),x) & = \left ( \mathcal{L}_X (H(\tilde m_{j+1}),x) - \mathcal{L}_X (H(\tilde L_{j}),x) \right ) \\
& + \left ( \mathcal{L}_X (H(\tilde L_{j}),x) - \mathcal{L}_X (H(\tilde m_{j}),x) \right ), \\
\forall x \leq \tilde L_0 =0, \ \mathcal{L}_X (t,x) = \mathcal{L}_X (H(\tilde m_{1}),x) & = \left ( \mathcal{L}_X (H(\tilde m_{1}),x) - \mathcal{L}_X (H(\tilde L_{0}),x) \right ). 
%\forall x \geq \tilde{L}_{N_t -1}, \ \mathcal{L}_X (t,x) \leq \left ( \mathcal{L}_X (t,x) - \mathcal{L}_X (H(\tilde m_{N_t}),x) \right ) %& + \left ( \mathcal{L}_X (H(\tilde m_{N_t}),x) - \mathcal{L}_X (H(\tilde L_{N_t -1}),x) \right ). 
\end{align*}
Moreover, 
\begin{align*}
\forall x \in [\tilde{L}_{N_t-1}, \tilde{L}_{N_t}], \ \mathcal{L}_X (t,x) & = \left ( \mathcal{L}_X (t,x) - \mathcal{L}_X (H(\tilde m_{N_t}),x) \right ) \\
& + \left ( \mathcal{L}_X (H(\tilde m_{N_t}),x) - \mathcal{L}_X (H(\tilde L_{N_t-1}),x) \right ), \\
\forall x \geq \tilde{L}_{N_t}, \ \mathcal{L}_X (t,x) & \leq \left ( \mathcal{L}_X (H(\tilde m_{N_t+1}),x) - \mathcal{L}_X (H(\tilde L_{N_t}),x) \right ). 
\end{align*}
The event $\{ N_t < n_t \} \cap \mathcal{E}^2_t$ ensures that the local time does not grow too much between $H(\tilde L_{j-1})$ and $H(\tilde m_{j})$ for $j \in \{1, ..., N_t + 1 \}$: $\sup_{y \in \mathbb{R}} ( \mathcal{L}_X (H(\tilde m_{j}),y) - \mathcal{L}_X (H(\tilde L_{j-1}),y) ) \leq t e^{(\kappa (1+3\delta )-1)\phi(t)}$. 
%The event $\{ V \in \mathcal{V}_t, \ N_t < n_t \}$ ensures that at time $t$ the diffusion is trapped in one of the first $n_t$ standard valleys. 
We thus see that, on $\{ V \in \mathcal{V}_t, \ N_t < n_t \} \cap \mathcal{E}^1_t \cap \mathcal{E}^2_t$, we have 
\begin{align*}
\sup_{y \in \mathbb{R}} \mathcal{L}_X(t, y) & \leq \max \left \{ \sup_{y \in [\tilde{L}_{N_t-1}, \tilde{L}_{N_t}]} \left ( \mathcal{L}_X (t,y) - \mathcal{L}_X (H(\tilde m_{N_t}),y) \right ), \right . \\ 
& \left . \sup_{1 \leq j \leq N_t - 1} \sup_{y \in [\tilde{L}_{j-1}, \tilde{L}_j]} \left ( \mathcal{L}_X (H(\tilde L_j),y) - \mathcal{L}_X (H(\tilde m_{j}),y) \right ) \right \} + t e^{(\kappa (1+3\delta )-1)\phi(t)}. 
\end{align*}
The term $\mathbb{P} ( \mathcal{L}^*_X(t) \geq t x_t, \ V \in \mathcal{V}_t, \ N_t < n_t, \ \mathcal{E}^1_t, \mathcal{E}^2_t )$ in the right hand side of \eqref{lsub0} is therefore less than
\begin{align*}
& \mathbb{P} \left ( \sup_{y \in [\tilde{L}_{N_t-1}, \tilde{L}_{N_t}]} \left ( \mathcal{L}_X (t,y) - \mathcal{L}_X (H(\tilde m_{N_t}),y) \right ) \vee \sup_{1 \leq j \leq N_t - 1} \sup_{y \in [\tilde{L}_{j-1}, \tilde{L}_j]} \left ( \mathcal{L}_X (H(\tilde L_j),y) - \mathcal{L}_X (H(\tilde m_{j}),y) \right ) \geq t \tilde x_t^1, \right . \\
& \left . V \in \mathcal{V}_t, \ N_t < n_t \right ), 
\end{align*}
where $\tilde x_t^1 := x_t - e^{(\kappa (1+3\delta )-1)\phi(t)} \sim x_t$. Then, since $\tilde x_t^1$ converge to $+\infty$, we have $\tilde x_t^1 \geq e^{-2 \phi(t)}$ for $t$ large enough. Using the definition of $\mathcal{E}^3_t$, we get that for such large $t$ the above is less than 
\begin{align*}
& \mathbb{P} \left ( \sup_{y \in \mathcal{D}_{N_t}} \left ( \mathcal{L}_X (t,y) - \mathcal{L}_X (H(\tilde m_{N_t}),y) \right ) \vee \sup_{1 \leq j \leq N_t - 1} \sup_{y \in \mathcal{D}_j} \left ( \mathcal{L}_X (H(\tilde L_j),y) - \mathcal{L}_X (H(\tilde m_{j}),y) \right ) \geq t \tilde x_t^1, \right . \\
& \left . V \in \mathcal{V}_t, \ N_t < n_t \right ) + \mathbb{P} \left ( \overline{\mathcal{E}^3_t} \right ), 
\end{align*}
where $\mathcal{D}_j$ is defined in \eqref{defdj}. Putting all this together, we see that for $t$ large enough, $\mathbb{P} ( \mathcal{L}^*_X(t) \geq t x_t )$ is less than
\begin{align}
& \mathbb{P} \left ( \sup_{1 \leq j \leq N_t - 1} \sup_{y \in \mathcal{D}_j} \left ( \mathcal{L}_X (H(\tilde L_j),y) - \mathcal{L}_X (H(\tilde m_{j}),y) \right ) \geq t \tilde x_t^1, \ V \in \mathcal{V}_t, \ N_t < n_t \right ) \nonumber \\
+ & \mathbb{P} \left ( \sup_{y \in \mathcal{D}_{N_t}} \left ( \mathcal{L}_X (t,y) - \mathcal{L}_X (H(\tilde m_{N_t}),y) \right ) \geq t \tilde x_t^1 \right ) \nonumber \\
+ & \mathbb{P} \left ( V \notin \mathcal{V}_t \right ) + \mathbb{P} \left ( N_t \geq n_t \right ) + \mathbb{P} \left ( \overline{\mathcal{E}^1_t} \right ) + \mathbb{P} \left ( \overline{\mathcal{E}^2_t} \right ) + \mathbb{P} \left ( \overline{\mathcal{E}^3_t} \right ). \label{lsub1}
\end{align}
We now deal with the first term. 
%since the other terms are controlled by respectively Lemma \ref{neglectlastvalley} and Fact \ref{bigevents}. 
Using the definition of $\mathcal{E}^4_t$ we get that the first term of \eqref{lsub1} is less than
\[ \mathbb{P} \left ( \sup_{1 \leq j \leq N_t - 1} \mathcal{L}_X (H(\tilde L_j), \tilde m_j) \geq t \tilde x_t^2, \ V \in \mathcal{V}_t, \ N_t < n_t \right ) + \mathbb{P} \left ( \overline{\mathcal{E}^4_t} \right ), \]
where $\tilde x_t^2 := (1+e^{-\tilde c h_t})^{-1} \ \tilde x_t^1 \sim x_t$. Now using the definition of $\mathcal{E}^6_t$, the above is less than
\[ \mathbb{P} \left ( \sup_{1 \leq j \leq N_t - 1} e_j S_j^t \geq \tilde x_t, \ V \in \mathcal{V}_t, \ N_t < n_t \right ) + \mathbb{P} \left ( \overline{\mathcal{E}^4_t} \right ) + \mathbb{P} \left ( \overline{\mathcal{E}^6_t} \right ), \]
where $\tilde x_t := (1+e^{- \tilde c h_t})^{-1} \ \tilde x_t^2 \sim x_t$. Using \eqref{bigeventsnbvalles} with $\eta = a-1$, the above is less than 
\begin{align*}
& \mathbb{P} \left ( \sup_{1 \leq j \leq \mathcal{N}_{at} - 1} e_j S_j^t \geq \tilde x_t \right ) + \mathbb{P} \left ( \overline{\mathcal{E}^4_t} \right ) + \mathbb{P} \left ( \overline{\mathcal{E}^6_t} \right ) + \mathbb{P} \left ( \overline{\mathcal{E}^5_t} \right ) + \mathbb{P} \left ( \overline{\mathcal{E}^7_t} \right ), \\
= & \mathbb{P} \left ( Y_1^{\natural, t} \left ( Y_2^{-1, t} (a) - \right ) \geq \tilde x_t \right ) + \mathbb{P} \left ( \overline{\mathcal{E}^4_t} \right ) + \mathbb{P} \left ( \overline{\mathcal{E}^6_t} \right ) + \mathbb{P} \left ( \overline{\mathcal{E}^5_t} \right ) + \mathbb{P} \left ( \overline{\mathcal{E}^7_t} \right ). 
\end{align*}

Combining with \eqref{lsub1} and the fact that $\tilde x_t^1 \geq \tilde x_t$, we get that $\mathbb{P} ( \mathcal{L}^*_X(t) \geq t x_t )$ is less than
\begin{align*}
& \mathbb{P} \left ( Y_1^{\natural, t} \left ( Y_2^{-1, t}(a) - \right ) \geq \tilde x_t \right ) + \mathbb{P} \left ( \sup_{\mathcal{D}_{N_t}} \left ( \mathcal{L}_X (t,.) - \mathcal{L}_X (H(\tilde m_{N_t}),.) \right ) \geq t \tilde x_t \right ) + \mathbb{P} \left ( V \notin \mathcal{V}_t \right ) \\
+ & \mathbb{P} \left ( N_t \geq n_t \right ) + \mathbb{P} \left ( \overline{\mathcal{E}^1_t} \right ) + \mathbb{P} \left ( \overline{\mathcal{E}^2_t} \right ) + \mathbb{P} \left ( \overline{\mathcal{E}^3_t} \right ) + \mathbb{P} \left ( \overline{\mathcal{E}^4_t} \right ) + \mathbb{P} \left ( \overline{\mathcal{E}^5_t} \right ) + \mathbb{P} \left ( \overline{\mathcal{E}^6_t} \right ) + \mathbb{P} \left ( \overline{\mathcal{E}^7_t} \right ). 
\end{align*}
Applying Lemma \ref{neglectlastvalley} with $u = \sqrt a$ (and $x_t$ replaced by $\tilde x_t$ which does not change anything since $\tilde x_t$ also satisfies \eqref{defxtlim}), Fact \ref{minimacoincide}, \eqref{majonbvalleesvisit} and \eqref{bigeventsleq}, we deduce the existence of a positive constant $c$ such that for $t$ large enough, 
\[ \mathbb{P} \left ( \mathcal{L}^*_X(t) \geq t x_t \right ) \leq \mathbb{P} \left ( Y_1^{\natural, t} \left ( Y_2^{-1, t}(a) - \right ) \geq \tilde x_t \right ) + \mathbb{P} \left ( R_1^t \leq \frac{\sqrt{a}}{\tilde x_t} \right ) + e^{-c \phi(t)}. \]
Since $\tilde x_t \sim x_t$ and $a > 1$ we get the upper bound when $t$ is large enough.

\textit{Lower bound}
\begin{align*}
\mathbb{P} \left ( \mathcal{L}^*_X(t) \geq t x_t \right ) & \geq \mathbb{P} \left ( \sup_{1 \leq j \leq N_t - 1} \mathcal{L}_X(t, \tilde m_j) \geq t x_t \right ) \\
& \geq \mathbb{P} \left ( \sup_{1 \leq j \leq N_t - 1} \mathcal{L}_X(H(\tilde L_j), \tilde m_j) \geq t x_t, \ V \in \mathcal{V}_t, \ N_t < n_t, \ \mathcal{E}^1_t \right ) 
\end{align*}
%\begin{align*}
%\mathbb{P} \left ( \mathcal{L}^*_X(t) \geq t x_t \right ) & \geq \mathbb{P} \left ( \mathcal{L}^*_X(t) \geq t x_t, \ \mathcal{A}_t \right ) \\
%& \geq \mathbb{P} \left ( \sup_{1 \leq j \leq N_t - 1} \mathcal{L}_X(H(\tilde L_j), \tilde m_j) \geq t x_t, \ \mathcal{A}_t \right ) 
%\end{align*}
because, on $\{ V \in \mathcal{V}_t \} \cap \{ N_t < n_t \} \cap \mathcal{E}^1_t$, $\tilde m_j$ (for $j < N_t)$ is no longer reached between times $H(\tilde L_j)$ and $t$, 
\[ \geq \mathbb{P} \left ( \sup_{1 \leq j \leq N_t - 1} e_j S_j^t \geq t \hat x_t, \ V \in \mathcal{V}_t, \ N_t < n_t, \ \mathcal{E}^1_t, \ \mathcal{E}^6_t \right ) \]
%\begin{eqnarray}
%\geq \mathbb{P} \left ( \sup_{1 \leq j \leq N_t - 1} e_j S_j^t \geq t \hat x^1_t, \ N_t < n_t, \ \mathcal{E}^1_t, \ \mathcal{E}^6_t \right ), \label{yeswecan3}
%\end{eqnarray}
because of the definition of $\mathcal{E}^6_t$ and where $\hat x_t := (1-e^{- \tilde c h_t})^{-1} \ x_t \sim x_t$, 
\[ \geq \mathbb{P} \left ( \sup_{1 \leq j \leq \mathcal{N}_{t/a} - 1} e_j S_j^t \geq t \hat x_t, \ V \in \mathcal{V}_t, \ N_t < n_t, \ \mathcal{E}^1_t, \ \mathcal{E}^6_t, \ \mathcal{E}^5_t, \ \mathcal{E}^7_t \right ) \]
where we used \eqref{bigeventsnbvalles} with $\eta = 1 - a^{-1}$, 
\[ \geq \mathbb{P} \left ( Y_1^{\natural, t} \left ( Y_2^{-1, t}(1/a) - \right ) \geq \hat x_t \right ) - \left ( \mathbb{P} \left ( V \notin \mathcal{V}_t \right ) + \mathbb{P} \left ( N_t \geq n_t \right ) + \mathbb{P} \left ( \overline{\mathcal{E}^1_t} \right ) + \mathbb{P} \left ( \overline{\mathcal{E}^5_t} \right ) + \mathbb{P} \left ( \overline{\mathcal{E}^6_t} \right ) + \mathbb{P} \left ( \overline{\mathcal{E}^7_t} \right ) \right ). \]
Applying Fact \ref{minimacoincide}, \eqref{majonbvalleesvisit} and \eqref{bigeventsleq} we get the lower bound since $\hat x_t \sim x_t$ and $a > 1$. 

%Now, still on $\mathcal{A}_t$, we have, for $k \leq n_t$,
%\[ H(\tilde m_k) \leq \tilde{v}_t + \sum_{i=1}^{k-1} U_i \ \text{ and } \ \left | \sum_{i=1}^{k-1} U_i - \sum_{i=1}^{k-1} \mathcal{H}_i \right | \leq \epsilon_t \sum_{i=1}^{k-1} \mathcal{H}_i. \]
%Now, according to the above expression, $k \leq \mathcal{N}_{t(1-\epsilon)}$ implies $\sum_{i=1}^{k-1} U_i \leq t(1-\epsilon)(1+\epsilon_t)$ which implies $H(\tilde m_k) \leq t(1-\epsilon)(1+\epsilon_t) - \tilde{v}_t \leq t$ for $t$ large enough. We have thus proved that, on $\mathcal{A}_t$ and for $t$ large enough, $k \leq \mathcal{N}_{t(1-\epsilon)}$ implies $k \leq N_t$. Putting this into \eqref{yeswecan3} we get
%\begin{align*}
%\mathbb{P} \left ( \mathcal{L}^*_X(t) \geq t x_t \right ) & \geq \mathbb{P} \left ( \sup_{1 \leq j \leq \mathcal{N}_{(1-\epsilon)t} - 1} e_j S_j^t \geq t \hat x_t \right ) \\
%& = \mathbb{P} \left ( Y_1^{\natural, t} \left ( Y_2^{-1, t}(1 - \eta) - \right ) \geq \hat x_t \right )
%\end{align*}

\end{proof}

%\begin{lemme} \label{cvnt}
%
%Let $\mathcal{U}$ denote the potential measure of $\mathcal{Y}_2$, then
%\[ e^{-\kappa \phi(t)} \mathbb{E} \left [ N_t \right ] \underset{t \rightarrow +\infty}{\longrightarrow} \mathcal{U} ([0, 1]). \]
%
%\end{lemme}

\begin{prop} \label{queueloilimsupdiscret}
Let $y_t$ go to infinity with $t$. For any $b > 0$ and $u \in ]0, 1[$ there is a positive constant $C$ (depending on $b$ and $u$) such that for all $t$ large enough, 
\[ C \ \mathbb{P} \left ( R_1^t \leq u b/y_t \right ) / y_t^{\kappa} \leq \mathbb{P} \left ( Y_1^{\natural, t} \left ( Y_2^{-1, t}(b) - \right ) \geq y_t \right ) \leq \mathbb{P} \left ( R_1^t \leq b/y_t \right ). \]

\end{prop}

\begin{proof}

For any $z > 0$, let $k^t(z)$ be the first index $k \geq 1$ such that $e_k S_k^t /t \geq z$. We have
\begin{align}
\left \{ Y_1^{\natural, t} \left ( Y_2^{-1, t}(b) - \right ) \geq y_t \right \} = \left \{ k^t(y_t) < \mathcal{N}_{t b} \right \} = \left \{ \sum_{i=1}^{k^t(y_t)} e_i S_i^t R_i^t \leq tb \right \} \subset \left \{ e_{k^t(y_t)} S_{k^t(y_t)}^t R_{k^t(y_t)}^t \leq tb \right \}, \label{queueloilimsupdiscret1}
\end{align}
and $e_{k^t(y_t)} S_{k^t(y_t)}^t R_{k^t(y_t)}^t$ has the same law as $e_1 S_1^t R_1^t$ conditionally on $e_1 S_1^t \geq t y_t$ so, using \eqref{queueloilimsupdiscret1} and the independence between $e_1 S_1^t$ and $R_1^t$: 
\[ \mathbb{P} \left ( Y_1^{\natural, t} \left ( Y_2^{-1, t}(b) - \right ) \geq y_t \right ) \leq \mathbb{P} \left ( e_1 S_1^t R_1^t \leq tb | e_1 S_1^t \geq t y_t \right ) \leq \mathbb{P} \left ( R_1^t \leq b/y_t \right ), \]
which proves the upper bound. For the lower bound, we fix $\eta < 1-u$. Note that according to \eqref{queueloilimsupdiscret1},
\[ \left \{ e_{k^t(y_t)} S_{k^t(y_t)}^t R_{k^t(y_t)}^t \leq tb(1-\eta) \right \} \cap \left \{ \sum_{i=1}^{k^t(y_t)-1} e_i S_i^t R_i^t < tb\eta \right \} \subset \left \{ Y_1^{\natural, t} \left ( Y_2^{-1, t}(b) - \right ) \geq y_t \right \}, \]
and both events on the left-hand-side are independent so 
\begin{align}
\mathbb{P} \left ( Y_1^{\natural, t} \left ( Y_2^{-1, t}(b) - \right ) \geq y_t \right ) \geq \mathbb{P} \left ( e_1 S_1^t R_1^t \leq tb(1-\eta) | e_1 S_1^t \geq t y_t \right ) \times \mathbb{P} \left ( \sum_{i=1}^{k^t(y_t)-1} e_i S_i^t R_i^t < tb\eta \right ). \label{queueloilimsupdiscret2}
\end{align}
We first deal with the second factor. Let $(Z_{i})_{i \geq 1}$ be \textit{iid} random variables such that $\mathcal{L}(Z_{1}) = \mathcal{L} (e_1 S_1^t R_1^t | e_1 S_1^t \leq t y_t)$ and $T$ be a geometric random variable with parameter $\mathbb{P}(e_1 S_1^t \geq t y_t)$, independent from the sequences $(Z_{i})_{i \geq 1}$ and $(e_i S_i^t R_i^t)_{i \geq 1}$. We then have
\[ \mathbb{P} \left ( \sum_{i=1}^{k^t(y_t)-1} e_i S_i^t R_i^t < tb\eta \right ) = \mathbb{P} \left ( \sum_{i=1}^{T-1} Z_i < tb\eta \right ) \geq \mathbb{P} \left ( \sum_{i=1}^{T-1} e_i S_i^t R_i^t < tb\eta \right ), \]
because the random variable $e_i S_i^t R_i^t$ is stochastically greater than the random variable $Z_i$. Then, 
\[ \mathbb{P} \left ( \sum_{i=1}^{T-1} e_i S_i^t R_i^t < tb\eta \right ) \geq \mathbb{P} \left ( T \leq e^{\kappa \phi(t)} \right ) \times \mathbb{P} \left ( \sum_{1 \leq i \leq e^{\kappa \phi(t)}} e_i S_i^t R_i^t < tb\eta \right ). \]
On the first hand we have
\[ \mathbb{P} \left ( \sum_{1 \leq i \leq e^{\kappa \phi(t)}} e_i S_i^t R_i^t < tb\eta \right ) = \mathbb{P} \left ( Y_2^{t}(1) < b \eta \right ) \underset{t \rightarrow +\infty}{\longrightarrow} \mathbb{P} \left ( \mathcal{Y}_2(1) < b \eta \right ) > 0, \]
where we used Fact \ref{cvsub} and where $\mathcal{Y}_2$ is as in there, the second component of the limit process $(\mathcal{Y}_1, \mathcal{Y}_2)$. On the second hand
\[ \mathbb{P} \left ( T \leq e^{\kappa \phi(t)} \right ) = 1 - \left ( 1 - \mathbb{P}(e_1 S_1^t \geq t y_t) \right )^{\lfloor e^{\kappa \phi(t)} \rfloor} = 1 - e^{\lfloor e^{\kappa \phi(t)} \rfloor \ln (1 - \mathbb{P}(e_1 S_1^t \geq t y_t))}. \]
Using \eqref{cvmeasure7.1} we get $\mathbb{P} \left ( T \leq e^{\kappa \phi(t)} \right ) \underset{t \rightarrow +\infty}{\sim} \mathcal{C}' / y_t^{\kappa}$, where $\mathcal{C}' $ is the constant in Fact \ref{queueiid}. Putting all this together, we get the existence of a positive constant $c_1 > 0$ such that for $t$ large enough, 
\begin{eqnarray}
\mathbb{P} \left ( \sum_{i=1}^{k^t(y_t)-1} e_i S_i^t R_i^t < tb\eta \right ) \geq c_1 / y_t^{\kappa}. \label{queueloilimsupdiscret3}
\end{eqnarray}

We now study the first factor in the right hand side of \eqref{queueloilimsupdiscret2}. From the independence of the two factors $e_1 S_1^t$ and $R_1^t$ in $e_1 S_1^t R_1^t$ we have
\[ \mathbb{P} \left ( e_1 S_1^t R_1^t \leq tb(1-\eta) | e_1 S_1^t \geq t y_t \right ) \geq \mathbb{P} \left ( R_1^t \leq bu /y_t \right ) \times \mathbb{P} \left ( e_1 S_1^t \leq t y_t (1-\eta)/ u | e_1 S_1^t \geq t y_t \right ) \]
and
\begin{align*}
\mathbb{P} \left ( e_1 S_1^t \leq t y_t (1-\eta)/ u | e_1 S_1^t \geq t y_t \right ) & = \frac{\mathbb{P} \left ( t y_t \leq e_1 S_1^t \leq t y_t (1-\eta)/ u \right )}{\mathbb{P} \left ( e_1 S_1^t \geq t y_t \right )} \\
& = 1 - \frac{\mathbb{P} \left ( e_1 S_1^t > t y_t (1-\eta)/ u \right )}{\mathbb{P} \left ( e_1 S_1^t \geq t y_t \right )} \\
& \underset{t \rightarrow +\infty}{\longrightarrow} 1 - (u / (1-\eta))^{\kappa} > 0, 
\end{align*}
where the limit comes from \eqref{cvmeasure7.1}, because $y_t$ goes to infinity. We thus get the existence of a positive constant $c_2 > 0$ such that for $t$ large enough, 
\begin{eqnarray}
\mathbb{P} \left ( e_1 S_1^t R_1^t \leq tb(1-\eta) | e_1 S_1^t \geq t y_t \right ) \geq c_2 \mathbb{P} \left ( R_1^t \leq bu /y_t \right ). \label{queueloilimsupdiscret4}
\end{eqnarray}

Putting \eqref{queueloilimsupdiscret3} and \eqref{queueloilimsupdiscret4} in \eqref{queueloilimsupdiscret2} we get the lower bound. 

\end{proof}

%The combination of Propositions \ref{approxdutl} and \ref{queueloilimdiscret} tells us that the study of $\mathbb{P} ( \mathcal{L}^*_X(t) \geq t x_t )$ is reduced to the study of 

Fix $\theta > 1$. We apply Proposition \ref{approxdutl} (the upper bound with $a=\theta^{1/3}$ and the lower bound with $a=\theta^{1/4}$), Proposition \ref{queueloilimsupdiscret} (the upper bound applied with $b=\theta^{1/3}$, $y_t = \theta^{-1/3} x_t$ and the lower bound with $b = \theta^{-1/4}$, $y_t = \theta^{1/4} x_t$, $u = \theta^{-1/4}$)  and Lemma \ref{approxder2cas} (the upper bounds of \eqref{approxder2caslevy} and \eqref{approxder2casmdb} applied with $a= \theta^{1/3}$, $z_t = \theta^{-2/3} x_t$, and the lower bounds of these same expressions applied with $a=\theta^{1/4}$, $z_t = \theta^{3/4} x_t$). We get: 

\begin{prop} \label{synthese}

Fix $\theta > 1$. If $V$ possesses negative jumps, there are positive constants $\tilde C$ and $c$ such that for $t$ large enough, 
\begin{align}
\frac{e^{ -c  \left ( \log (x_t) \right )^2}}{\tilde C x_t^{\kappa}} \mathbb{P} \left ( I(V^{\uparrow}) \leq 1 / \theta x_t \right ) - e^{-c \phi(t)} \leq \mathbb{P} \left ( \mathcal{L}^*_X(t) \geq t x_t \right ) \leq \tilde C \ \mathbb{P} \left ( I(V^{\uparrow}) \leq \theta / x_t \right ) + e^{-c \phi(t)}. \label{synthesecasjump}
\end{align}
If $V := W_{\kappa}$, the $\kappa$-drifted Brownian motion, there are positive constants $C$ and $c$ such that for $t$ large enough, 
\begin{eqnarray}
\frac1{\tilde C x_t^{\kappa}} \mathbb{P} \left ( \mathcal{R} \leq 1/\theta x_t \right ) - e^{-c \phi(t)} \leq \mathbb{P} \left ( \mathcal{L}^*_X(t) \geq t x_t \right ) \leq \tilde C \ \mathbb{P} \left ( \mathcal{R} \leq \theta /x_t \right ) + e^{-c \phi(t)}. \label{synthesecasmdb}
\end{eqnarray}

%\begin{eqnarray}
%\mathbb{P} \left ( \mathcal{L}^*_X(t) \geq t x_t \right ) \leq 2 \mathbb{P} \left ( I(V^{\uparrow}) \leq (1+\eta) / \tilde x_t \right ) + ..., \label{synthese1}
%\end{eqnarray}
%
%\begin{eqnarray}
%\mathbb{P} \left ( \mathcal{L}^*_X(t) \geq t x_t \right ) \geq \frac{c}{x_t^{1+\kappa}} \mathbb{P} \left ( I(V^{\uparrow}) \leq (1-\eta) \hat x_t \right ) - ... \label{synthese2}
%\end{eqnarray}
\end{prop}

We can now link the asymptotic behavior of the local time with the left tail of $I(V^{\uparrow})$: 

%\begin{prop} \label{limsupenfctdelefttail}
%If there is $\gamma > 1$ and $C > 0$ such that for $x$ small enough we have
%\begin{eqnarray}
%\mathbb{P} \left ( I(V^{\uparrow}) \leq x \right ) \leq \exp \left ( -\frac{C}{x^{\frac1{\gamma-1}}} \right ), \label{condfinie}
%\end{eqnarray}
%then
%\begin{eqnarray}
%\limsup_{t \rightarrow +\infty} \frac{\mathcal{L}^*_X(t)}{t (\log (\log(t)))^{\gamma - 1}} \leq C^{1-\gamma}. \label{limsupfinie}
%\end{eqnarray}
%If there is $\gamma > 1$ and $C > 0$ such that for $x$ small enough we have
%\begin{eqnarray}
%\mathbb{P} \left ( I(V^{\uparrow}) \leq x \right ) \geq \exp \left ( -\frac{C}{x^{\frac1{\gamma-1}}} \right ) \label{condfnonnulle}
%\end{eqnarray}
%then
%\begin{eqnarray}
%\limsup_{t \rightarrow +\infty} \frac{\mathcal{L}^*_X(t)}{t (\log (\log(t)))^{\gamma - 1}} \geq C^{1-\gamma}. \label{limsuppos}
%\end{eqnarray}
%\end{prop}

\begin{proof} of Theorem \ref{limsupenfctdelefttail}

First, we assume that $V$ possesses negative jumps. 

Let us assume that \eqref{condfinie} is satisfied with some constants $\gamma > 1$ and $C > 0$. We now prove \eqref{limsupfinie}. Let $a > 1$ and define the events
\[ \mathcal{A}_n := \left \{ \sup_{t \in [a^n, a^{n+1}]} \frac{\mathcal{L}^*_X(t)}{t (\log (\log(t)))^{\gamma - 1}}  \geq a^3 C^{1-\gamma} \right \}. \]

We define $x_t := C^{1-\gamma} a^{2} (\log (\log(t/a)))^{\gamma-1}$. Note that such a choice of $x_t$ satisfies \eqref{defxtlim} with $\mu = \gamma$ and $D = C^{1-\gamma}$. From the increase of $\mathcal{L}^*_X(.)$, \eqref{synthesecasjump} (the upper bound applied with $t = a^{n+1}$, $\theta = a$) and \eqref{condfinie} we have, for $n$ large enough, 
\begin{align*}
\mathbb{P} \left ( \mathcal{A}_n \right ) & \leq \mathbb{P} \left ( \mathcal{L}^*_X(a^{n+1}) \geq C^{1-\gamma} a^{n+3} (\log (\log(a^n)))^{\gamma-1} \right ) \nonumber \\
& \leq \tilde C \ \mathbb{P} \left ( I(V^{\uparrow}) \leq 1 / a C^{1-\gamma} (\log (\log(a^n)))^{\gamma-1} \right ) + e^{-c \phi(a^{n+1})} \\
& \leq \tilde C \ \exp \left ( -a^{\frac1{\gamma-1}} \log (\log(a^n)) \right ) + e^{-c \phi(a^{n+1})} \\
& = \tilde C \ (\log(a))^{-a^{\frac1{\gamma-1}}} n^{-a^{\frac1{\gamma-1}}} + e^{-c \phi(a^{n+1})}. 
\end{align*}
Since $e^{-c \phi(a^{n+1})} = e^{-c (\log \log(a^{n+1}))^{\omega}} \leq n^{-2}$ for $n$ large enough, the above is the general term of a converging series so, using the Borel-Cantelli lemma, we deduce that $\mathbb{P}$-almost surely, 
\[ \limsup_{t \rightarrow +\infty} \frac{\mathcal{L}^*_X(t)}{t (\log (\log(t)))^{\gamma - 1}} \leq a^3 C^{1-\gamma}, \]
and letting $a$ go to $1$ we get \eqref{limsupfinie}. 

Now, let us assume that \eqref{condfnonnulle} is satisfied with some constants $\gamma > 1$ and $C > 0$. We now prove \eqref{limsuppos}. Let $a>1$ and let $t_n$, $u_n$, $v_n$ (defined from this $a$) and $X^n$ be as in Subsection \ref{decompindeppart} (recall that $X^n - v_n$ is equal in law to $X$ under the annealed probability $\mathbb{P}$). We define
\[ \mathcal{B}_n := \left \{ \frac{\mathcal{L}^{*}_{X^n} (t_n)}{t_n (\log (\log(t_n)))^{\gamma - 1}} \geq \frac{C^{1-\gamma}}{a^{3(\gamma - 1)}} \right \}. \]
We also define $\mathcal{C}_n$ and $\mathcal{D}_n$ to be as in Lemma \ref{lemprindep} and $\mathcal{E}_n := \mathcal{B}_n \cap \mathcal{C}_n$. We define $x_t := C^{1-\gamma} (\log (\log(t)))^{\gamma - 1} / a^{3(\gamma - 1)}$. Note that such a choice of $x_t$ satisfies \eqref{defxtlim} with $\mu = \gamma$ and $D = C^{1-\gamma} / a^{3(\gamma - 1)}$. According to \eqref{synthesecasjump} (the lower bound applied with $t=t_n$, $\theta = a^{\gamma - 1}$), the definition of $t_n$, \eqref{condfnonnulle} and the fact that $n$ is large we have, for some positive constants $c_a$, $K_a$ and $n$ large enough, 
\begin{align*}
\mathbb{P} \left ( \mathcal{B}_n \right ) & \geq K_a e^{ -c_a  \left ( \log (\log (\log (t_n))) \right )^2} \ \mathbb{P} \left ( I(V^{\uparrow}) \leq a^{2(\gamma - 1)} / C^{1-\gamma} (\log (\log(t_n)))^{\gamma - 1} \right )/ (\log (\log(t_n)))^{\kappa(\gamma - 1)} - e^{-c \phi(t_n)} \\
& = K_a e^{ -c_a  \left ( \log (a \log(n)) \right )^2} \ \mathbb{P} \left ( I(V^{\uparrow}) \leq a^{\gamma - 1} / C^{1-\gamma} (\log (n))^{\gamma - 1} \right )/ (a \log (n))^{\kappa(\gamma - 1)} - e^{-c \phi(t_n)} \\
& \geq K_a e^{ -c_a  \left ( \log (a \log(n)) \right )^2} \ e^{- (\log (n)) / a }/ (a \log (n))^{\kappa(\gamma - 1)} - e^{-c \phi(t_n)} \\
& = K_a \ \exp \left ( -c_a  \left ( \log (a \log(n)) \right )^2 - \log (n) /a + \kappa(\gamma - 1) \log \left (a \log (n)) \right ) \right ) - e^{-c \phi(t_n)} \\
& \geq K_a \ \exp \left ( - \log (n) \right ) - e^{-c \phi(t_n)} \\
& = K_a \ n^{-1} - e^{-c \phi(t_n)}. 
\end{align*}
%\begin{align*}
%\mathbb{P} \left ( \mathcal{B}_n \right ) & \geq c a^{2 \kappa(\gamma - 1)} \ \mathbb{P} \left ( I(V^{\uparrow}) \leq a^{2(\gamma - 1)} / C^{1-\gamma} (a \log (n))^{\gamma - 1} \right )/ C^{\kappa(1-\gamma)} (a \log (n))^{\kappa(\gamma - 1)} - e^{-c \phi(t_n)} \\
%& \geq c a^{2 \kappa(\gamma - 1)} \ \exp \left ( - a \log (n) /a^2 \right )/ C^{\kappa(1-\gamma)} (a \log (n))^{\kappa(\gamma - 1)} - e^{-c \phi(t_n)} \\
%& = c a^{2 \kappa(\gamma - 1)} \ n^{-1/a} / C^{\kappa(1-\gamma)} (a \log(n))^{\kappa(\gamma - 1)} - e^{-c \phi(t_n)}
%\end{align*}
Since $e^{-c \phi(t_n)} = e^{-c (\log \log(t_n))^{\omega}} \leq n^{-2}$ for $n$ large enough, we get
\begin{eqnarray}
\sum_{n \geq 1} \mathbb{P} \left ( \mathcal{B}_n \right ) = +\infty. \label{serieinflimsup}
\end{eqnarray}

Then, the combination of \eqref{serieinflimsup} and \eqref{inegindep3} yields
\begin{eqnarray}
\sum_{n \geq 1} \mathbb{P} \left ( \mathcal{E}_n \right ) \geq \sum_{n \geq 1} \mathbb{P} \left ( \mathcal{B}_n \right ) - \sum_{n \geq 1} \mathbb{P} \left ( \overline{\mathcal{C}_n} \right ) = +\infty. \label{sommeenkappa<1}
\end{eqnarray}

Note that each event $\mathcal{E}_n$ belongs to the $\sigma$-field $\sigma ( V(s) - V(u_n), u_n \leq s \leq u_{n+1}, \ X(t), H(v_n) \leq t \leq H(v_n) + T_n)$, in other words, it only depends on the diffusion between times $H(v_n)$ and $H(v_n) + T_n$ and on the environment between positions $u_n$ and $u_{n+1}$. From the Markov property and the independence of the increments of the environment, we get that the events $(\mathcal{E}_n)_{n \geq 1}$ are independent. Combining this independence with \eqref{sommeenkappa<1} and the Borel-Cantelli Lemma we get that $\mathbb{P}$-almost surely, the event $\mathcal{E}_n$ is realized infinitely many often. For $n$ such that this event is realized we have
\begin{eqnarray}
\frac{\mathcal{L}^{*}_{X} (H(v_n) + t_n)}{t_n (\log (\log(t_n)))^{\gamma - 1}} \geq \frac{\mathcal{L}^{*}_{X^n} (t_n)}{t_n (\log (\log(t_n)))^{\gamma - 1}} \geq \frac{C^{1-\gamma}}{a^{3(\gamma - 1)}}. \label{equivh(vn)1}
\end{eqnarray}
According to \eqref{inegindep4} ans the Borel-Cantelli Lemma we have $\mathbb{P}$-almost surely
\[ H(v_n) + t_n \underset{n \rightarrow +\infty}{\sim} t_n, \]
so combining with \eqref{equivh(vn)1} we deduce that $\mathbb{P}$-almost surely, 
\[ \limsup_{t \rightarrow +\infty} \frac{\mathcal{L}^*_X(t)}{t (\log (\log(t)))^{\gamma - 1}} \geq \frac{C^{1-\gamma}}{a^{3(\gamma - 1)}}, \]
and letting $a$ go to $1$ we get \eqref{limsuppos}. 

If $V = W_{\kappa}$, the $\kappa$-drifted Brownian motion with $0 < \kappa < 1$, we proceed the same proof, only replacing $I(V^{\uparrow})$ by $\mathcal{R}$ and \eqref{synthesecasjump} by \eqref{synthesecasmdb}. We thus get that the same result is stil true for $V = W_{\kappa}$, but with $\mathcal{R}$ instead of $I(V^{\uparrow})$.

\end{proof}

The other theorems for the $\limsup$ are now easy to prove. 

\begin{proof} of Theorems \ref{limsupkappa<1} and \ref{limsupkappa<1exact}

In the case where $V$ possesses negative jumps, Theorem \ref{limsupkappa<1} is a direct consequence of the combination of Theorem \ref{limsupenfctdelefttail}, \eqref{encadgene1} and \eqref{encadgene2}. Similarly, the first point of Theorem \ref{limsupkappa<1exact} is obtained from the combination of Theorem \ref{limsupenfctdelefttail} and \eqref{encadreg}. The second point of Theorem \ref{limsupkappa<1exact} is obtained from the combination of Theorem \ref{limsupenfctdelefttail} and \eqref{exactereg}. 

%, we note that, according to the first point $t \log (\log (t))$ is the right renormalization but to get the exact limit
In the case where $V = W_{\kappa}$, the $\kappa$-drifted Brownian motion with $0 < \kappa < 1$, we only need to prove the last point of Theorem \ref{limsupkappa<1exact}, and this requires to determine exactly the left tail of $\mathcal{R}$. This variable is equal in law to the sum of two independent random variables having the same law as $I(W_k^{\uparrow})$. We thus have
\[ - \log \left ( \mathbb{E} [e^{-\lambda \mathcal{R}}] \right ) = - \log \left ( \left ( \mathbb{E} [e^{-\lambda I(W_k^{\uparrow})}] \right )^2 \right ) = - 2 \log \left ( \mathbb{E} [e^{-\lambda I(W_k^{\uparrow})}] \right ) \underset{\lambda \rightarrow + \infty}{\sim} 4 \sqrt{2 \lambda}, \]
where the equivalent comes from $(1.13)$ of \cite{foncexpovech}. Using this together with De Bruijn's Theorem (see Theorem 4.12.9 in \cite{regvar}) we get
\[ - \log \left ( \mathbb{P} \left ( \mathcal{R} \leq x \right ) \right ) \underset{x \rightarrow 0}{\sim} \frac{8}{x}. \]
The last point of Theorem \ref{limsupkappa<1exact} follows from this combined with Theorem \ref{limsupenfctdelefttail}. 
\end{proof}

\subsection{The $\liminf$} \label{preuveliminf}

%\[ \mathbb{P} \left ( \mathcal{L}^*_X(t) \leq t / x_t \right ) \geq \mathbb{P} \left ( \mathcal{L}^*_X(t) \leq t / x_t, \ \mathcal{A}_t \right ) \geq \mathbb{P} \left ( \sup_{1 \leq j \leq N_t} \sup_{y \in \mathcal{D}_j} \mathcal{L}_X(H(\tilde L_j), y) \leq t / x_t, \ \mathcal{A}_t \right ), \]
%and, thanks to \eqref{yeswecan2.11}, this is more than
%\[ \mathbb{P} \left ( \sup_{1 \leq j \leq N_t} \mathcal{L}_X(H(\tilde L_j), \tilde m_j) \leq t / \tilde x_t, \ \mathcal{A}_t \right ), \]
%where $\tilde x_t := ... x_t \underset{t \rightarrow +\infty}{\sim} x_t$, 
%\[ \geq \mathbb{P} \left ( \sup_{1 \leq j \leq N_t} e_j S_j^t \leq t / \hat x_t, \ \mathcal{A}_t \right ), \]
%%\begin{eqnarray}
%%\geq \mathbb{P} \left ( \sup_{1 \leq j \leq N_t} e_j S_j^t \leq t / \hat x_t, \ \mathcal{A}_t \right ), \label{minoliminf1}
%%\end{eqnarray}
%where $\hat x_t := (1 + \epsilon_t) \tilde x_t \underset{t \rightarrow +\infty}{\sim} x_t$, because on $\mathcal{A}_t$ we have both $N_t \leq n_t$ and $\forall j \leq n_t, | \mathcal{L}_X(H(\tilde L_j), \tilde m_j) - e_j S_j^t | \leq \epsilon_t e_j S_j^t$. Now, using the fact that on $\mathcal{A}_t$ and for $t$ large enough \eqref{yeswecan4} is true, we deduce that for $t$ large enough, 
%\begin{eqnarray}
%\mathbb{P} \left ( \mathcal{L}^*_X(t) \leq t / x_t \right ) \geq \mathbb{P} \left ( \sup_{1 \leq j \leq \mathcal{N}_{(1+\epsilon)t}} e_j S_j^t \leq t / \hat x_t \right ) = \mathbb{P} \left ( Y_1^{\natural, t} \left ( Y_2^{-1, t}(1 + \epsilon) \right ) \leq 1/ \hat x_t \right ). \label{minoliminf1}
%\end{eqnarray}

For the $\liminf$ we study the asymptotic of the quantity $\mathbb{P} (\mathcal{L}^*_X(t)/t \leq 1/x_t)$. Recall that $x_t$ is defined in \eqref{defxtlim} where $D > 0$ and $\mu \in ]1, 2]$ are fixed constants. In all this subsection we take $\omega := 2$ for the parameter in \eqref{paramtaillevallees}. 

\begin{prop} \label{approxdupetittl}
Recall the $\lambda_0$ defined in Fact \ref{queueiid}. There is a positive constant $c$ such that for all $a > 1$ and $t$ large enough we have, 
\begin{align*}
\mathbb{P} \left ( Y_1^{\natural, t} \left ( Y_2^{-1, t}(a) \right ) \leq 1/a x_t \right ) - e^{-c \phi(t)} & \leq \mathbb{P} \left( \mathcal{L}^*_X(t) \leq t /x_t \right ) \\
& \leq 2 \mathbb{P} \left ( Y_1^{\natural, t} \left ( Y_2^{-1, t}(1/4) \right ) \leq 2/x_t \right ) + e^{- \lambda_0 x_t /8} + e^{-c \phi(t)}. 
\end{align*}
\end{prop}

Here, the functional of $(Y_1^t, Y_2^t)$ involved is $Y_1^{\natural, t} ( Y_2^{-1, t}(.))$ which represents the supremum of the local time after leaving the last valley. 

\begin{proof}

\textit{Lower bound}

For the lower bound we first have to prove that, if the local time at the bottom of the valleys is small, then the overall supremum of the local time is small. For this our argument starts similarly as in the proof of the lower bound of Proposition 4.1 in \cite{caslevyvech}, by using the localization in the bottom of the valleys of the main contributions of the local time. 

From the definition of $N_t$, we have $H(\tilde m_{N_t +1}) \geq t$ on $\{ V \in \mathcal{V}_t, \ N_t < n_t \}$ so
\begin{eqnarray}
\mathbb{P} \left ( \mathcal{L}^*_X(t) \leq t /x_t \right ) \geq \mathbb{P} \left ( \mathcal{L}^*_X( H(\tilde m_{N_t + 1})) \leq t /x_t, \ V \in \mathcal{V}_t, \ N_t < n_t, \ \mathcal{E}^1_t, \mathcal{E}^2_t \right ). \label{lb1}
\end{eqnarray}
The event $\{ N_t < n_t \} \cap \mathcal{E}^1_t$ ensures that for $j \leq N_t < n_t$, $\tilde m_{j}$ is no longer reached after $H(\tilde L_{j})$ and $\tilde{L}_{j}$ is no longer reached after $H(\tilde m_{j+1})$ so that for each $j \in \{1, ..., N_t \}$ we have 
\begin{align*}
\forall x \in [\tilde{L}_{j-1}, \tilde{m}_j], \ \mathcal{L}_X (H(\tilde m_{N_t + 1}),x) = \mathcal{L}_X (H(\tilde L_j),x) & = \left ( \mathcal{L}_X (H(\tilde L_j),x) - \mathcal{L}_X (H(\tilde m_{j}),x) \right ) \\
& + \left ( \mathcal{L}_X (H(\tilde m_{j}),x) - \mathcal{L}_X (H(\tilde L_{j-1}),x) \right ), \\
\forall x \in [\tilde{m}_j, \tilde{L}_{j}], \ \mathcal{L}_X (H(\tilde m_{N_t + 1}),x) = \mathcal{L}_X (H(\tilde m_{j+1}),x) & = \left ( \mathcal{L}_X (H(\tilde m_{j+1}),x) - \mathcal{L}_X (H(\tilde L_{j}),x) \right ) \\
& + \left ( \mathcal{L}_X (H(\tilde L_{j}),x) - \mathcal{L}_X (H(\tilde m_{j}),x) \right ), \\
\forall x \leq \tilde L_0 =0, \ \mathcal{L}_X (H(\tilde m_{N_t + 1}),x) = \mathcal{L}_X (H(\tilde m_{1}),x) & = \left ( \mathcal{L}_X (H(\tilde m_{1}),x) - \mathcal{L}_X (H(\tilde L_{0}),x) \right ). 
%\forall x \geq \tilde{L}_{N_t -1}, \ \mathcal{L}_X (t,x) \leq \left ( \mathcal{L}_X (t,x) - \mathcal{L}_X (H(\tilde m_{N_t}),x) \right ) %& + \left ( \mathcal{L}_X (H(\tilde m_{N_t}),x) - \mathcal{L}_X (H(\tilde L_{N_t -1}),x) \right ). 
\end{align*}
Moreover, 
\[ \forall x \in [\tilde{L}_{N_t}, \tilde{m}_{N_t +1}], \ \mathcal{L}_X (H(\tilde m_{N_t + 1}),x) = \left ( \mathcal{L}_X (H(\tilde m_{N_t+1}),x) - \mathcal{L}_X (H(\tilde L_{N_t}),x) \right ). \]
The event $\{ N_t < n_t \} \cap \mathcal{E}^2_t$ ensures that the local time does not grow too much between $H(\tilde L_{j-1})$ and $H(\tilde m_{j})$ for $j \in \{1, ..., N_t + 1 \}$: $\sup_{y \in \mathbb{R}} ( \mathcal{L}_X (H(\tilde m_{j}),y) - \mathcal{L}_X (H(\tilde L_{j-1}),y) ) \leq t e^{(\kappa (1+3\delta )-1)\phi(t)}$. 
%The event $\{ V \in \mathcal{V}_t, \ N_t < n_t \}$ ensures that at time $t$ the diffusion is trapped in one of the first $n_t$ standard valleys. 
We thus see that, on $\{ V \in \mathcal{V}_t, \ N_t < n_t \} \cap \mathcal{E}^1_t \cap \mathcal{E}^2_t$, we have 
\[ \sup_{y \in \mathbb{R}} \mathcal{L}_X(H(\tilde m_{N_t + 1}), y) \leq \sup_{1 \leq j \leq N_t} \sup_{y \in [\tilde{L}_{j-1}, \tilde{L}_j]} \left ( \mathcal{L}_X (H(\tilde L_j),y) - \mathcal{L}_X (H(\tilde m_{j}),y) \right ) + t e^{(\kappa (1+3\delta )-1)\phi(t)}. \]
The right hand side of \eqref{lb1} is therefore more than
\[ \mathbb{P} \left ( \sup_{1 \leq j \leq N_t} \sup_{y \in [\tilde{L}_{j-1}, \tilde{L}_j]} \left ( \mathcal{L}_X (H(\tilde L_j),y) - \mathcal{L}_X (H(\tilde m_{j}),y) \right ) \leq t /\tilde x_t^1, \ V \in \mathcal{V}_t, \ N_t < n_t, \ \mathcal{E}^1_t, \mathcal{E}^2_t \right ), \]
where $\tilde x_t^1 := 1/((1/x_t) - e^{(\kappa (1+3\delta )-1)\phi(t)})$. Then, since $\tilde x_t^1 \sim x_t$, we have $1/\tilde x_t^1 \geq e^{-2 \phi(t)}$ for $t$ large enough. Using the definition of $\mathcal{E}^3_t$, we get that for such large $t$ the above is more than 
\[ \mathbb{P} \left ( \sup_{1 \leq j \leq N_t} \sup_{y \in \mathcal{D}_j} \left ( \mathcal{L}_X (H(\tilde L_j),y) - \mathcal{L}_X (H(\tilde m_{j}),y) \right ) \leq t / \tilde x_t^1, \ V \in \mathcal{V}_t, \ N_t < n_t, \ \mathcal{E}^1_t, \mathcal{E}^2_t, \mathcal{E}^3_t \right ), \]
where $\mathcal{D}_j$ is defined in \eqref{defdj}. From the definition of $\mathcal{E}^4_t$ the above is more than
\[ \mathbb{P} \left ( \sup_{1 \leq j \leq N_t} \mathcal{L}_X (H(\tilde L_j), \tilde m_j) \leq t /\tilde x_t^2, \ V \in \mathcal{V}_t, \ N_t < n_t, \ \mathcal{E}^1_t, \mathcal{E}^2_t, \mathcal{E}^3_t, \mathcal{E}^4_t \right ), \]
where $\tilde x_t^2 := (1+e^{-\tilde c h_t}) \tilde x_t^1 \sim x_t$. Now using the definition of $\mathcal{E}^6_t$, the above is more than
\[ \mathbb{P} \left ( \sup_{1 \leq j \leq N_t} e_j S_j^t \leq 1/\tilde x_t, \ V \in \mathcal{V}_t, \ N_t < n_t, \ \mathcal{E}^1_t, \mathcal{E}^2_t, \mathcal{E}^3_t, \mathcal{E}^4_t, \mathcal{E}^6_t \right ), \]
where $\tilde x_t := (1+e^{- \tilde c h_t}) \tilde x_t^2 \sim x_t$. Let $a > 1$. Using \eqref{bigeventsnbvalles} with $\eta = a-1$, the above is more than
\begin{align*}
& \mathbb{P} \left ( \sup_{1 \leq j \leq \mathcal{N}_{at}} e_j S_j^t \leq 1/\tilde x_t, \ V \in \mathcal{V}_t, \ N_t < n_t, \ \mathcal{E}^1_t, \mathcal{E}^2_t, \mathcal{E}^3_t, \mathcal{E}^4_t, \mathcal{E}^5_t, \mathcal{E}^6_t, \mathcal{E}^7_t \right ) \\
\geq & \mathbb{P} \left ( Y_1^{\natural, t} \left ( Y_2^{-1, t}(a) \right ) \leq 1/\tilde  x_t \right ) \\
- & \left ( \mathbb{P} \left ( V \notin \mathcal{V}_t \right ) + \mathbb{P} \left ( N_t \geq n_t \right ) + \mathbb{P} (\overline{\mathcal{E}^1_t}) + \mathbb{P} (\overline{\mathcal{E}^2_t}) + \mathbb{P} (\overline{\mathcal{E}^3_t}) + \mathbb{P} (\overline{\mathcal{E}^4_t}) + \mathbb{P} (\overline{\mathcal{E}^5_t}) + \mathbb{P} (\overline{\mathcal{E}^6_t}) + \mathbb{P} (\overline{\mathcal{E}^7_t}) \right ), 
\end{align*}
where we used the definition of $(Y_1^t, Y_2^t)$. Applying Fact \ref{minimacoincide}, \eqref{majonbvalleesvisit} and \eqref{bigeventsleq} we get the asserted lower bound for a suitably chosen constant $c$ and $t$ large enough, since $\tilde x_t \sim x_t$ and $a>1$. 

\textit{Upper bound}
\begin{align}
\mathbb{P} \left ( \mathcal{L}^*_X(t) \leq t / x_t \right ) & \leq \mathbb{P} \left ( \sup_{1 \leq j \leq N_t} \mathcal{L}_X(t, \tilde m_j) \leq t / x_t, \ V \in \mathcal{V}_t, \ N_t < n_t \right ) + \mathbb{P} \left ( V \notin \mathcal{V}_t \right ) + \mathbb{P} \left ( N_t \geq n_t \right ) \nonumber \\
& \leq \mathbb{P} \left ( \mathcal{L}_X(t, \tilde m_{N_t}) \leq t / x_t, \ \sup_{1 \leq j \leq N_t - 1} \mathcal{L}_X(H(\tilde L_j), \tilde m_j) \leq t / x_t, \ V \in \mathcal{V}_t, \ N_t < n_t, \ \mathcal{E}^1_t \right ) \nonumber \\
& + \mathbb{P} \left ( \overline{\mathcal{E}^1_t} \right ) + \mathbb{P} \left ( V \notin \mathcal{V}_t \right ) + \mathbb{P} \left ( N_t \geq n_t \right ), \label{lastestim0}
\end{align}
because, on $\{ V \in \mathcal{V}_t, \ N_t < n_t \} \cap \mathcal{E}^1_t$, $\tilde m_j$ (for $j < N_t)$ is no longer reached between times $H(\tilde L_j)$ and $t$. 
%Lemma \ref{neglectlastvalley} QU'EST CE QUE CA VIENT FAIRE LA ???
We fix $v \in \mathcal{G}_t$, a realization of the environment. Let us define
\[ \mathcal{E}_t(v, k, z) := \left \{ \mathcal{L}_{X_{\tilde m_k}}(t(1-z), \tilde m_k) \leq t /x_t, \ H_{X_{\tilde m_k}}(\tilde m_{k+1}) \geq t(1-z), \ H_{X_{\tilde m_k}}(\tilde L_k) < H_{X_{\tilde m_k}}(\tilde L_{k-1}) \right \}, \]
and $\nu_t(v, k, z) := P^v(\mathcal{E}_t(v, k, z))$. 
%\begin{align*}
%\nu_t(v, k, z) := P^v & \left ( \mathcal{L}_{X(H(\tilde m_k) + .)}(t(1-z), \tilde m_k) \leq t /x_t, \ H(\tilde m_{k+1}) - H(\tilde m_k) \geq t(1-z), \right. \\
%& \left. H_{X_{\tilde m_k}}(\tilde L_k) < H_{X_{\tilde m_k}}(\tilde L_{k-1}) \right ). 
%\end{align*}
The event $\mathcal{E}_t(v, k, z)$ belongs to the $\sigma$-field $\sigma \left ( X(t), t \geq H(\tilde m_k) \right )$.
In other words, it only depends on the diffusion after time $H(\tilde m_k)$. On the other hand, $H(\tilde m_{k})$ and $\sup_{1 \leq j \leq k - 1} \mathcal{L}_X(H(\tilde L_j), \tilde m_j)$ are measurable with respect to the $\sigma$-field $\sigma \left ( X(t), 0 \leq t \leq H(\tilde m_{k}) \right )$. From the Markov property applied to $X$ at $H(\tilde m_k)$, we get that $H(\tilde m_k)$ is independent from the event $\mathcal{E}_t(v, k, z)$. As a consequence, $P^v ( \mathcal{L}_X(t, \tilde m_{N_t}) \leq t / x_t, \ \sup_{1 \leq j \leq N_t - 1} \mathcal{L}_X(H(\tilde L_j), \tilde m_j) \leq t / x_t, \ N_t < n_t, \ \mathcal{E}^1_t ) $ is less than
%The event in $\nu_t(v, k, z)$ belongs to the $\sigma$-field 
%\[ \sigma \left ( V(s) - V(\tilde L_{k-1}), s \geq \tilde L_{k-1}, \ X(t), H(\tilde m_k) \leq t \leq \min (H_{X_{\tilde m_k}}(\tilde L_k), H_{X_{\tilde m_k}}(\tilde L_{k-1})) \right ). \]
%In other words, it only depends on the diffusion between times $H(\tilde m_k)$ and 
%
%\noindent $\min (H_{X_{\tilde m_k}}(\tilde L_k), H_{X_{\tilde m_k}}(\tilde L_{k-1}))$ and on the environment after the stopping time $\tilde L_{k-1}$. 
%
%On the other hand, $H(\tilde m_k)$ is measurable with respect to the $\sigma$-field 
%\[ \sigma \left ( V(s), 0 \leq s \leq \tilde m_k, \ X(t), 0 \leq t \leq H(\tilde m_k) \right ). \]
%From the Markov property applied to $V$ at $\tilde L_{k-1}$ and applied to $X$ at $H(\tilde m_k)$, we get that $H(\tilde m_k)$ is independent from the event in $\nu_t(v, k, z)$ so the first term of the right hand side of \eqref{lastestim0} is less than
\begin{eqnarray}
\sum_{k=1}^{n_t} \int_0^1 \nu_t(v, k, z) \times P^v \left ( \sup_{1 \leq j \leq k - 1} \mathcal{L}_X(H(\tilde L_j), \tilde m_j) \leq t / x_t, \ H(\tilde m_k)/ t \in dz \right ). \label{lastestim1}
\end{eqnarray}
%The first term is less than
%\[ \sum_{k=1}^{n_t} \int_0^1 \mathbb{P} \left ( \mathcal{L}_X(t(1-z), \tilde m_k) \leq t / x_t, \ H(\tilde m_{k+1}) - H(\tilde m_{k}) > \geq t(1-z), ...\right ) \times \mathbb{P} \left ( \sup_{1 \leq j \leq k - 1} \mathcal{L}_X(H(\tilde L_j), \tilde m_j) \leq t / x_t, \ H(\tilde m_k)/ t \in dz \right ) \]
%BREF, DEFINIR UN $\nu_t(v, k, z)$. 
%From the definition of $\mathcal{E}^9_t$, the fact that $H(\tilde m_{k+1}) - H(\tilde m_k) \geq t(1-z)$, the definition of $\mathcal{E}^5_t(k)$ and the definition of $\mathcal{E}^7_t$, we get that $\nu_t(v, k, z)$ is less than
The fact that the sum stops at $n_t$ comes from $N_t < n_t$ together with the fact that $v \in \mathcal{G}_t \subset \mathcal{V}_t$. From the definition of $\mathcal{E}^9_t(v, k, z)$ we get 
%\[ P^v \left ( \frac{1-z}{R_k^t} (1- ...) \leq \frac1{x_t}, \ H(\tilde m_{k+1}) - H(\tilde m_k) \geq t(1-z), \ H_{X_{\tilde m_k}}(\tilde L_k) < H_{X_{\tilde m_k}}(\tilde L_{k-1}) \right ) + \mathbb{P} \left ( ... \right ) \]
\begin{eqnarray}
\mathcal{E}_t(v, k, z) \subset \left \{ \frac{1-z}{R_k^t} (1- e^{- \tilde c h_t}) \leq \frac1{x_t}, \ H_{X_{\tilde m_k}}(\tilde m_{k+1}) \geq t(1-z) \right \} \cup \mathcal{E}^9_t(v, k, z). \label{lastestim1.1}
\end{eqnarray}
%\[ P^v \left ( \frac{1-z}{R_k^t} (1- ...) \leq \frac1{x_t}, \ (1+e^{- \tilde c h_t}) e_k S_k^t R_k^t \geq t \left (1-z - e^{-\phi(t) /2} \right ) \right ) + ... \]
We deduce that \eqref{lastestim1} is less than 
\begin{align*}
& \sum_{k=1}^{n_t} \int_0^1 P^v \left ( \frac{1-z}{R_k^t} (1- e^{- \tilde c h_t}) \leq \frac1{x_t}, \ H_{X_{\tilde m_k}}(\tilde m_{k+1}) \geq t(1-z) \right ) \\
& \times P^v \left ( \sup_{1 \leq j \leq k - 1} \mathcal{L}_X(H(\tilde L_j), \tilde m_j) \leq t / x_t, \ H(\tilde m_k)/ t \in dz \right ) \\
+ & \sum_{k=1}^{n_t} \int_0^1 P^v \left ( \mathcal{E}^9_t(v, k, z) \right ) \times P^v \left ( \sup_{1 \leq j \leq k - 1} \mathcal{L}_X(H(\tilde L_j), \tilde m_j) \leq t / x_t, \ H(\tilde m_k)/ t \in dz \right ). 
\end{align*}
$v$ being fixed, the events in the right hand side of \eqref{lastestim1.1} are independent from the $\sigma$-field $\sigma ( X(t), 0 \leq t \leq H(\tilde m_k) )$. Using this independence and the fact that the sums over $k$ above corresponds to disjoint events (so they can be written as probabilites of unions of events), and integrating with respect to $v \in D(\mathbb{R}, \mathbb{R})$ equipped with the probability measure $P$, we get that 
%$\sum_{k=1}^{n_t} \int_0^1 \nu_t(v, k, z) \times P^v ( \sup_{1 \leq j \leq k - 1} \mathcal{L}_X(H(\tilde L_j), \tilde m_j) \leq t / x_t, \ H(\tilde m_k)/ t \in dz)$ 
the first term in the right hand side of \eqref{lastestim0} is less than
\begin{align*}
& \mathbb{P} \left ( \frac{1- H(\tilde m_{N_t})/ t}{R_{N_t}^t} (1- e^{- \tilde c h_t}) \leq 1/x_t, \ \sup_{1 \leq j \leq N_t - 1} \mathcal{L}_X(H(\tilde L_j), \tilde m_j) \leq t / x_t \right ) \\
+ & E \left [ \mathds{1}_{V \in \mathcal{G}_t} P^V \left ( \cup_{k=1}^{n_t} \{ N_t \geq k \} \cap \mathcal{E}^9_t(V, k, H(\tilde m_{k})/t) \right) \right ] + P(V \notin \mathcal{G}_t) \\
%+ & \mathbb{E} \left [ \sum_{k=1}^{n_t} \int_0^1 P^v \left ( \mathcal{E}^9_t(v, k, z) \right) \times P^v \left ( \sup_{1 \leq j \leq k - 1} \mathcal{L}_X(H(\tilde L_j), \tilde m_j) \leq t / x_t, \ H(\tilde m_k)/ t \in dz \right ) \right ], \mathbb{P} \left ( \cup_{k=1}^{n_t} \mathcal{E}^9_t(v, k, H(\tilde m_k)/ t) \right ) \\
\leq & \mathbb{P} \left ( (1- e^{- \tilde c h_t})^{-1} \frac{R_{N_t}^t}{x_t} + \frac{H(\tilde m_{N_t})}{t} \geq 1, \ \sup_{1 \leq j \leq N_t - 1} \mathcal{L}_X(H(\tilde L_j), \tilde m_j) \leq t / x_t \right ) + e^{-c \phi(t)}, 
\end{align*}
where $c$ is a positive constant and where we have used \eqref{majoe9} for the second term, Lemma \ref{measofgoodenv} for the third term, and the fact that $t$ is large enough, 
\begin{align*}
\leq \mathbb{P} & \left ( (1- e^{- \tilde c h_t})^{-1} \frac{R_{N_t}^t}{x_t} + \frac{(1+e^{- \tilde c h_t})}{t} \sum_{j=1}^{N_t - 1} e_j S_j^t R_j^t \geq 1-2/\log h_t, \ \sup_{1 \leq j \leq N_t - 1} e_j S_j^t \leq (1 - e^{- \tilde c h_t})^{-1} t / x_t, \right. \\
& \left. V \in \mathcal{V}_t, \ N_t < n_t, \ \mathcal{E}^5_t, \mathcal{E}^6_t, \mathcal{E}^7_t \right ) + \mathbb{P} \left ( V \notin \mathcal{V}_t \right ) + \mathbb{P} \left ( N_t \geq n_t \right ) + \mathbb{P} \left ( \overline{\mathcal{E}^5_t} \right ) + \mathbb{P} \left ( \overline{\mathcal{E}^6_t} \right ) + \mathbb{P} \left ( \overline{\mathcal{E}^7_t} \right ) + e^{-c \phi(t)}, 
\end{align*}
where we used the definitions of $\mathcal{E}^5_t$, $\mathcal{E}^6_t$ and $\mathcal{E}^7_t$, 
\[ \leq \mathbb{P} \left ( \frac{R_{N_t}^t}{x_t} + \frac1{t} \sum_{j=1}^{N_t - 1} e_j S_j^t R_j^t \geq 1/2, \ \sup_{1 \leq j < N_t - 1} e_j S_j^t < 2 t / x_t, \ V \in \mathcal{V}_t, \ N_t < n_t, \ \mathcal{E}^5_t, \mathcal{E}^6_t, \mathcal{E}^7_t \right ) + e^{-c \phi(t)} \]
where we used the fact that $t$ is large enough, Fact \ref{minimacoincide}, \eqref{majonbvalleesvisit} and \eqref{bigeventsleq}, and where the constant $c$ has been suitably decreased, 
\begin{align*}
\leq & \mathbb{P} \left ( \frac{R_{N_t}^t}{x_t} + \frac1{t} \sum_{j=1}^{N_t - 1} e_j S_j^t R_j^t \geq 1/2, \ \sup_{1 \leq j \leq N_t - 1} e_j S_j^t < 2 t / x_t, \ e_{N_t} S_{N_t}^t \geq 2 t / x_t \right ) \\
+ & \mathbb{P} \left ( \sup_{1 \leq j \leq N_t} e_j S_j^t < 2 t / x_t, \ V \in \mathcal{V}_t, \ N_t < n_t, \ \mathcal{E}^5_t, \mathcal{E}^6_t, \mathcal{E}^7_t \right ) + e^{-c \phi(t)}. 
\end{align*}
On the event in the probability of the first term, we have $N_t = k^t(2/x_t)$. For the second term we use \eqref{bigeventsnbvalles} with $\eta = 3/4$. The above is thus less than
\begin{align*}
& \mathbb{P} \left ( \frac{R_{k^t(2/x_t)}^t}{x_t} + \frac1{t} \sum_{j=1}^{k^t(2/x_t) - 1} e_j S_j^t R_j^t \geq 1/2 \right ) + \mathbb{P} \left ( \sup_{1 \leq j \leq \mathcal{N}_{t/4}} e_j S_j^t < 2 t / x_t \right ) + e^{-c \phi(t)} \\
\leq & \mathbb{P} \left ( \frac{R_{k^t(2/x_t)}^t}{x_t} \geq 1/4 \right ) + \mathbb{P} \left ( \frac1{t} \sum_{j=1}^{k^t(2/x_t) - 1} e_j S_j^t R_j^t \geq 1/4 \right ) + \mathbb{P} \left ( \sup_{1 \leq j \leq \mathcal{N}_{t/4}} e_j S_j^t < 2 t / x_t \right ) + e^{-c \phi(t)} \\
\leq & \mathbb{P} \left ( R_1^t \geq x_t /4 \right ) + 2 \mathbb{P} \left ( Y_1^{\natural, t} \left ( Y_2^{-1, t}(1/4) \right ) \leq 2/x_t \right ) + e^{-c \phi(t)}
\end{align*}
where, for the last inequality, we used the fact that the sequence $(R_j^t)_{j \geq 1}$ is \textit{iid} and independent from the random index $k^t(2/x_t)$, \eqref{queueloilimdiscretbis1} with $b = 1/4$, $y_t = x_t/2$, and the definition of $(Y_1^t, Y_2^t)$. 
%for the third term we used Fact \ref{minimacoincide}, and for the fourth term we used \eqref{majonbvalleesvisit}. 
Then, note that according to \eqref{cvrlapl} we have for all $t$ large enough, 
\[ \mathbb{P} \left ( R_1^t \geq x_t /4 \right ) \leq e^{- \lambda_0 x_t /8}. \]
Bounding the three terms in the right hand side of \eqref{lastestim0} thanks to the above, Fact \ref{minimacoincide}, \eqref{majonbvalleesvisit} and \eqref{bigeventsleq} we get the upper bound for a suitably chosen constant $c$ and $t$ large enough. 

\end{proof}

We now study the functional involved in Proposition \ref{approxdupetittl}. For this we need two lemmas. 

In the remaining part of this subsection we fix $\eta \in ]0, 1/3[$. Let $y_t$ go to infinity with $t$ satisfying $\log (y_t) << \phi(t)$. Let $p_t := \mathbb{P} (e_1 S_1^t > t / y_t)$ and $(\overline{\mathcal{H}_{i}})_{i \geq 1}$ be \textit{iid} random variables such that $\overline{\mathcal{H}_{1}}$ has the same law as $e_1 S_1^t R_1^t$ conditionally to $\{ e_1 S_1^t \leq t / y_t \}$: $\mathcal{L}(\overline{\mathcal{H}_{1}}) = \mathcal{L} (e_1 S_1^t R_1^t|e_1 S_1^t \leq t / y_t)$. Since we have $\log (y_t) << \phi(t)$, \eqref{cvmeasure7.1} gives $p_t \sim \mathcal{C}' e^{-\kappa \phi(t)} y_t^{\kappa}$. We have

\begin{lemme} \label{etudelaplace}

\begin{eqnarray}
1 - \mathbb{E} \left [ e^{-\lambda \overline{\mathcal{H}_{1}} /t} \right ] \underset{t \rightarrow +\infty}{\sim} \lambda \frac{\mathcal{C}' \kappa \mathbb{E} \left [ \mathcal{R} \right ]}{(1-\kappa) e^{\kappa \phi(t)} y_t^{1-\kappa}}, \label{queueloilimdiscretbis8lp}
\end{eqnarray}
where $\mathcal{C}' $ is the constant in Fact \ref{queueiid} and $\mathcal{R}$ is defined in the beginning of Section \ref{results}. 

\end{lemme}

\begin{proof}

For any $\lambda \geq 0$ we have
\begin{align*}
\mathbb{E} \left [ e^{-\lambda \overline{\mathcal{H}_{1}} /t} \right ] & = \mathbb{E} \left [ e^{-\lambda e_1 S_1^t R_1^t /t} | e_1 S_1^t \leq t / y_t \right ] = (1-p_t)^{-1} \mathbb{E} \left [ e^{-\lambda e_1 S_1^t R_1^t /t} \mathds{1}_{e_1 S_1^t \leq t / y_t} \right ] \\
& = (1-p_t)^{-1} \int_0^{+\infty} \mathbb{E} \left [ e^{-\lambda u e_1 S_1^t /t} \mathds{1}_{e_1 S_1^t \leq t / y_t} \right ] \times \mathbb{P}(R_1^t \in du) \\
& = (1-p_t)^{-1} \int_0^{+\infty} \left ( 1 - e^{-\lambda u / y_t} p_t - \lambda \int_0^{1 / y_t} u e^{-\lambda x u} \mathbb{P}(e_1 S_1^t/t > x) dx \right ) \times \mathbb{P}(R_1^t \in du), 
\end{align*}
where we used iteration by parts, 
\[ = (1-p_t)^{-1} \left ( 1 - \mathbb{E} \left [ e^{-\lambda R_1^t / y_t} \right ] p_t - \lambda \int_0^{1 / y_t} \mathbb{E} \left [ R_1^t e^{-\lambda x R_1^t} \right ] \mathbb{P}(e_1 S_1^t/t > x) dx \right ), \]
where we used Fubini's Theorem, 
\begin{eqnarray}
= (1-p_t)^{-1} \left ( 1 - p_t + p_t(1-\mathbb{E} \left [ e^{-\lambda R_1^t / y_t} \right ]) - \lambda \int_0^{1 / y_t} \mathbb{E} \left [ R_1^t e^{-\lambda x R_1^t} \right ] \mathbb{P}(e_1 S_1^t/t > x) dx \right ). \label{queueloilimdiscretbis3lp}
\end{eqnarray}
We now study the second and third term in $\eqref{queueloilimdiscretbis3lp}$. Using the fact that the difference between two points of a continuously differentiable function is the integral of its derivative, the last part of Fact \ref{queueiid} and the equivalent for $p_t$ we get
\begin{eqnarray}
p_t(1-\mathbb{E} \left [ e^{-\lambda R_1^t / y_t} \right ]) = p_t \frac{\lambda}{y_t} \int_0^1 \mathbb{E} \left [ R_1^t e^{-\lambda u R_1^t / y_t} \right ] du \underset{t \rightarrow +\infty}{\sim} \lambda \frac{p_t \mathbb{E}  [ \mathcal{R} ]}{y_t} \underset{t \rightarrow +\infty}{\sim} \lambda \frac{\mathcal{C}' \mathbb{E}  [ \mathcal{R} ]}{e^{\kappa \phi(t)} y_t^{1-\kappa}}, \label{queueloilimdiscretbis4lp}
\end{eqnarray}
where $\mathcal{C}' $ is the constant in Fact \ref{queueiid}. Then, from the last part of Fact \ref{queueiid} again, 
\begin{eqnarray}
\int_0^{1 / y_t} \mathbb{E} \left [ R_1^t e^{-\lambda x R_1^t} \right ] \mathbb{P}(e_1 S_1^t/t > x) dx \underset{t \rightarrow +\infty}{\sim} \mathbb{E} \left [ \mathcal{R} \right ] \int_0^{1 / y_t} \mathbb{P}(e_1 S_1^t/t > x) dx. \label{queueloilimdiscretbis5lp}
\end{eqnarray}
Recall that $\eta \in ]0, 1/3[$. $\int_0^{1 / y_t} \mathbb{P}(e_1 S_1^t/t > x) dx$ equals
\begin{align}
& \int_0^{e^{-(1-2 \eta)\phi(t)}} \mathbb{P}(e_1 S_1^t/t > x) dx + \ e^{-\kappa \phi(t)} \int_{e^{-(1-2 \eta)\phi(t)}}^{1 / y_t} x^{-\kappa} x^{\kappa} e^{\kappa \phi(t)} \mathbb{P}(e_1 S_1^t/t > x) dx \nonumber \\
= & \int_0^{e^{-(1-2 \eta)\phi(t)}} \mathbb{P}(e_1 S_1^t/t > x) dx + \ e^{-\kappa \phi(t)} \int_{e^{-(1-2 \eta)\phi(t)}}^{1 / y_t} x^{-\kappa} \left ( x^{\kappa} e^{\kappa \phi(t)} \mathbb{P}(e_1 S_1^t/t > x) - \mathcal{C}' \right ) dx \nonumber \\
+ & \ \mathcal{C}' e^{-\kappa \phi(t)} \int_0^{1 / y_t} x^{-\kappa} dx - \ \mathcal{C}' e^{-\kappa \phi(t)} \int_0^{e^{-(1-2 \eta)\phi(t)}} x^{-\kappa} dx. \label{queueloilimdiscretbis6lp}
\end{align}
Now, the absolute values of the first and fourth terms of \eqref{queueloilimdiscretbis6lp} are respectively less than $e^{-(1-2 \eta)\phi(t)}$ and $\mathcal{C}' e^{-(2 \eta \kappa +(1-2 \eta))\phi(t)}/(1-\kappa)$. In particular, thanks to $2 \eta \kappa +(1-2 \eta) > \kappa$ (which is trivial) and $\log (y_t) << \phi(t)$, both are negligible with respect to $e^{-\kappa \phi(t)} / y_t^{1-\kappa}$. From \eqref{cvmeasure7.1} we also have
\[ e^{-\kappa \phi(t)} \int_{e^{-(1-2 \eta)\phi(t)}}^{1 / y_t} x^{-\kappa} \left ( x^{\kappa} e^{\kappa \phi(t)} \mathbb{P}(e_1 S_1^t/t > x) - \mathcal{C}' \right ) dx = \mathcal{O} (e^{-\kappa \phi(t)} / y_t^{1-\kappa}), \]
and the third term of \eqref{queueloilimdiscretbis6lp} equals $\mathcal{C}' e^{-\kappa \phi(t)} / (1-\kappa) y_t^{1-\kappa}$. Combining with \eqref{queueloilimdiscretbis5lp} we get
\begin{eqnarray}
-\lambda \int_0^{1 / y_t} \mathbb{E} \left [ R_1^t e^{-\lambda x R_1^t} \right ] \mathbb{P}(e_1 S_1^t/t > x) dx \underset{t \rightarrow +\infty}{\sim} -\lambda \frac{\mathcal{C}' \mathbb{E} \left [ \mathcal{R} \right ]}{(1-\kappa) e^{\kappa \phi(t)} y_t^{1-\kappa}}. \label{queueloilimdiscretbis7lp}
\end{eqnarray}

Putting together \eqref{queueloilimdiscretbis4lp} and \eqref{queueloilimdiscretbis7lp} in \eqref{queueloilimdiscretbis3lp} we obtain \eqref{queueloilimdiscretbis8lp}. 

\end{proof}

\begin{lemme} \label{etudelaplacege}

%Let $(\overline{\mathcal{H}_{i}})_{i \geq 1}$ be \textit{iid} random variables such that $\mathcal{L}_X(\overline{\mathcal{H}_{1}}) = \mathcal{L} (e_1 S_1^t R_1^t|e_1 S_1^t \leq t / y_t)$. 
Recall the $\lambda_0$ defined in Fact \ref{queueiid}. For any $\lambda \in [0, \lambda_0[$ we have
\begin{eqnarray}
\left ( \mathbb{E} \left [ e^{\lambda y_t \overline{\mathcal{H}_{1}} /t} \right ] - 1 \right ) / p_t \underset{t \rightarrow +\infty}{\longrightarrow} 1-\mathbb{E} \left [ e^{\lambda \mathcal{R}} \right ] + \lambda \int_0^{1} x^{-\kappa} \mathbb{E} \left [ \mathcal{R} e^{\lambda x \mathcal{R}} \right ] dx, \label{queueloilimdiscretbis8lpge}
\end{eqnarray}
and this limit is positive for all $\lambda \in ]0, \lambda_0[$. 
\end{lemme}

\begin{proof}

For any $\lambda \in [0, \lambda_0[$ we have $\lambda y_t e_1 S_1^t /t \in [0, \lambda_0[$ on the event $\{ e_1 S_1^t /t \leq 1/y_t \}$, so
\begin{align*}
\mathbb{E} \left [ e^{\lambda y_t \overline{\mathcal{H}_{1}} /t} \right ] & = \mathbb{E} \left [ e^{\lambda y_t e_1 S_1^t R_1^t /t} | e_1 S_1^t \leq t / y_t \right ] = (1-p_t)^{-1} \mathbb{E} \left [ e^{\lambda y_t e_1 S_1^t R_1^t /t} \mathds{1}_{e_1 S_1^t \leq t / y_t} \right ] \\
& = (1-p_t)^{-1} \int_0^{+\infty} \mathbb{E} \left [ e^{\lambda y_t u e_1 S_1^t /t} \mathds{1}_{e_1 S_1^t \leq t / y_t} \right ] \times \mathbb{P}(R_1^t \in du) \\
& = (1-p_t)^{-1} \int_0^{+\infty} \left ( 1 - e^{\lambda y_t u / y_t} p_t + \lambda y_t \int_0^{1 / y_t} u e^{\lambda y_t x u} \mathbb{P}(e_1 S_1^t/t > x) dx \right ) \times \mathbb{P}(R_1^t \in du) \\
& = (1-p_t)^{-1} \int_0^{+\infty} \left ( 1 - e^{\lambda u} p_t + \lambda \int_0^{1} u e^{\lambda y u} \mathbb{P}(e_1 S_1^t/t > y/y_t) dy \right ) \times \mathbb{P}(R_1^t \in du), 
\end{align*}
where we used iteration by parts and made the change of variable $y = y_t x$, 
\[ = (1-p_t)^{-1} \left ( 1 - \mathbb{E} \left [ e^{\lambda R_1^t} \right ] p_t + \lambda \int_0^{1} \mathbb{E} \left [ R_1^t e^{\lambda y R_1^t} \right ] \mathbb{P}(e_1 S_1^t/t > y/y_t) dy \right ) \]
where we used Fubini's Theorem, 
\begin{eqnarray}
= (1-p_t)^{-1} \left ( 1 - p_t + p_t(1-\mathbb{E} \left [ e^{\lambda R_1^t} \right ]) + \lambda \int_0^{1} \mathbb{E} \left [ R_1^t e^{\lambda y R_1^t} \right ] \mathbb{P}(e_1 S_1^t/t > y/y_t) dy \right ). \label{queueloilimdiscretbis3lpge}
\end{eqnarray}
According to \eqref{cvrlapl} the second term is equivalent to $p_t(1-\mathbb{E} [ e^{\lambda \mathcal{R}}])$. We now study the third term in $\eqref{queueloilimdiscretbis3lpge}$. Recall that $\eta \in ]0, 1/3[$. This term equals 
\begin{align}
& \lambda \int_0^{y_t e^{-(1-2 \eta)\phi(t)}} \mathbb{E} \left [ R_1^t e^{\lambda y R_1^t} \right ] \mathbb{P}(e_1 S_1^t/t > y/y_t) dy \nonumber \\
+ & \lambda y_t^{\kappa} e^{-\kappa \phi(t)} \int_{e^{-(1-2 \eta)\phi(t)}}^{1} y^{-\kappa} \mathbb{E} \left [ R_1^t e^{\lambda y R_1^t} \right ] (y/y_t)^{\kappa} e^{\kappa \phi(t)} \mathbb{P}(e_1 S_1^t/t > y/y_t) dy \nonumber \\
= & \lambda \int_0^{y_t e^{-(1-2 \eta)\phi(t)}} \mathbb{E} \left [ R_1^t e^{\lambda y R_1^t} \right ] \mathbb{P}(e_1 S_1^t/t > y/y_t) dy \nonumber \\
+ & \lambda y_t^{\kappa} e^{-\kappa \phi(t)} \int_{y_t e^{-(1-2 \eta)\phi(t)}}^{1} y^{-\kappa} \mathbb{E} \left [ R_1^t e^{\lambda y R_1^t} \right ] \left ( (y/y_t)^{\kappa} e^{\kappa \phi(t)} \mathbb{P}(e_1 S_1^t/t > y/y_t) - \mathcal{C}' \right ) dy \nonumber \\
+ & \lambda \mathcal{C}' y_t^{\kappa} e^{-\kappa \phi(t)} \int_0^{1} y^{-\kappa} \mathbb{E} \left [ R_1^t e^{\lambda y R_1^t} \right ] dy - \lambda \mathcal{C}' y_t^{\kappa} e^{-\kappa \phi(t)} \int_0^{y_t e^{-(1-2 \eta)\phi(t)}} y^{-\kappa} \mathbb{E} \left [ R_1^t e^{\lambda y R_1^t} \right ] dy. \label{queueloilimdiscretbis6lpge}
\end{align}
In the above, $\mathcal{C}' $ is the constant in Fact \ref{queueiid}. Now, thanks to \eqref{cvrlapl}, the absolute values of the first and fourth terms of \eqref{queueloilimdiscretbis6lpge} are ultimately less than $2 \lambda \mathbb{E} [\mathcal{R}] y_t e^{-(1-2 \eta)\phi(t)}$ and \\ $2 \lambda \mathcal{C}' \mathbb{E} [\mathcal{R}] y_t e^{-(2 \eta \kappa +(1-2 \eta))\phi(t)} /(1-\kappa)$. In particular, thanks to $2 \eta \kappa +(1-2 \eta) > \kappa$ (which is trivial) and $\log (y_t) << \phi(t)$, both are negligible with respect to $e^{-\kappa \phi(t)} y_t^{\kappa}$. From \eqref{cvmeasure7.1} and \eqref{cvrlapl} we also have
\[ \lambda y_t^{\kappa} e^{-\kappa \phi(t)} \int_{y_t e^{-(1-2 \eta)\phi(t)}}^{1} y^{-\kappa} \mathbb{E} \left [ R_1^t e^{\lambda y R_1^t} \right ] \left ( (y/y_t)^{\kappa} e^{\kappa \phi(t)} \mathbb{P}(e_1 S_1^t/t > y/y_t) - \mathcal{C}' \right ) dy = \mathcal{O} (e^{-\kappa \phi(t)} y_t^{\kappa}), \]
and, thanks to \eqref{cvrlapl}, the third term of \eqref{queueloilimdiscretbis6lpge} is equivalent to $\lambda \mathcal{C}' y_t^{\kappa} e^{-\kappa \phi(t)} \int_0^{1} y^{-\kappa} \mathbb{E} [ \mathcal{R} e^{\lambda y \mathcal{R}} ] dy$. We thus get 
\begin{align*}
\lambda \int_0^{1} \mathbb{E} \left [ R_1^t e^{\lambda y R_1^t} \right ] \mathbb{P}(e_1 S_1^t/t > y/y_t) dy & \underset{t \rightarrow + \infty}{\sim} \lambda \mathcal{C}' y_t^{\kappa} e^{-\kappa \phi(t)} \int_0^{1} y^{-\kappa} \mathbb{E} [ \mathcal{R} e^{\lambda y \mathcal{R}} ] dy \\
& \underset{t \rightarrow + \infty}{\sim} \lambda p_t \int_0^{1} y^{-\kappa} \mathbb{E} [ \mathcal{R} e^{\lambda y \mathcal{R}} ] dy. 
\end{align*}
Putting into \eqref{queueloilimdiscretbis3lpge} we obtain \eqref{queueloilimdiscretbis8lpge}. 

We justify the positivity of the limit as follows: we see that the right hand side of \eqref{queueloilimdiscretbis8lpge} is equivalent to $\lambda \kappa \mathbb{E} [ \mathcal{R} ] / (1-\kappa)$ when $\lambda$ goes to $0$. The limit in \eqref{queueloilimdiscretbis8lpge} is therefore positive for small $\lambda$. On the other hand, $( \mathbb{E} [ e^{\lambda y_t \overline{\mathcal{H}_{1}} /t} ] - 1) / p_t$ increases with $\lambda$ so the limit in \eqref{queueloilimdiscretbis8lpge} is non-decreasing on $[0, \lambda_0[$. We thus get the positivity of the limit for all $\lambda \in ]0, \lambda_0[$. 

\end{proof}

We can now study the lower and upper bounds given by Proposition \ref{approxdupetittl}: 

\begin{prop} \label{queueloilimdiscret}
Let $y_t$ be chosen as before (that is, $y_t \rightarrow +\infty$ and $\log (y_t) << \phi(t)$). There is a positive constant $L$ (not depending on the choice of $y_t$), such that for any $b > 0$, $u > 1$ and $t$ large enough we have 
\[ e^{- u b(1-\kappa) y_t /\kappa \mathbb{E}[\mathcal{R}]} \leq \mathbb{P} \left ( Y_1^{\natural, t} \left ( Y_2^{-1, t}(b) \right ) \leq 1/y_t \right ) \leq e^{- L b y_t}. \]
\end{prop}

\begin{proof}

\textit{Lower bound}

Let us fix $\alpha \in ]b(1-\kappa)/\kappa \mathbb{E}[\mathcal{R}], u b(1-\kappa)/\kappa \mathbb{E}[\mathcal{R}] [$. For any $z > 0$, $k^t(z)$ still denotes the first index $k \geq 1$ such that $e_k S_k^t /t \geq z$. We have
\begin{eqnarray}
\left \{ Y_1^{\natural, t} \left ( Y_2^{-1, t}(b) \right ) \leq 1/y_t \right \} = \left \{ k^t(1 / y_t) > \mathcal{N}_{tb} \right \} = \left \{ \sum_{i=1}^{k^t(1 / y_t)-1} e_i S_i^t R_i^t > tb \right \}. \label{queueloilimdiscretbis1}
\end{eqnarray}

Now, recall that $(\overline{\mathcal{H}_{i}})_{i \geq 1}$ are \textit{iid} random variables such that $\mathcal{L}(\overline{\mathcal{H}_{1}}) = \mathcal{L} (e_1 S_1^t R_1^t|e_1 S_1^t \leq t / y_t)$ and let $T$ be a geometric random variable with parameter $p_t = \mathbb{P}(e_1 S_1^t > t / y_t)$, independent from the sequence $(\overline{\mathcal{H}_{i}})_{i \geq 1}$. Recall that $p_t \sim \mathcal{C}' e^{-\kappa \phi(t)} y_t^{\kappa}$, where $\mathcal{C}' $ is the constant in Fact \ref{queueiid}. We have
\begin{align}
\mathbb{P} \left ( \sum_{i=1}^{k^t(1 / y_t)-1} e_i S_i^t R_i^t > tb \right ) = \mathbb{P} \left ( \sum_{i=1}^{T-1} \overline{\mathcal{H}_{i}} > tb \right ) \geq \mathbb{P} \left ( T > \alpha h(t) + 1 \right ) \times \mathbb{P} \left ( \sum_{i=1}^{\lfloor \alpha h(t) \rfloor} \overline{\mathcal{H}_{i}} > tb \right ), \label{queueloilimdiscretbis2}
\end{align}
where we put $h(t) := y_t/p_t \sim e^{\kappa \phi(t)} y_t^{1-\kappa} / \mathcal{C}'$. We give a lower bound for the two factors in the left hand side of \eqref{queueloilimdiscretbis2}. We first study the Laplace transform of the normalized sum of the second factor to prove its convergence to a constant number. For any $\lambda \geq 0$, we have
\[ \mathbb{E} \left [ e^{-\lambda \sum_{i=1}^{\lfloor \alpha h(t) \rfloor} \overline{\mathcal{H}_{i}}/t} \right ] = \left ( \mathbb{E} \left [ e^{-\lambda \overline{\mathcal{H}_{1}} /t} \right ] \right )^{\lfloor \alpha h(t) \rfloor} = e^{\lfloor \alpha h(t) \rfloor \log \left (1 + \mathbb{E} \left [ e^{-\lambda \overline{\mathcal{H}_{1}} /t} \right ] - 1 \right )}. \]
According to Lemma \ref{etudelaplace}, the exponent is equivalent to
\[ -\lambda \mathcal{C}' \kappa \mathbb{E} \left [ \mathcal{R} \right ] \alpha h(t) / (1-\kappa) e^{\kappa \phi(t)} y_t^{1-\kappa}, \]
and since $h(t) \sim e^{\kappa \phi(t)} y_t^{1-\kappa} / \mathcal{C}'$ we get
\[ \mathbb{E} \left [ e^{-\lambda \sum_{i=1}^{\lfloor \alpha h(t) \rfloor} \overline{\mathcal{H}_{i}}/t} \right ] \underset{t \rightarrow +\infty}{\longrightarrow} e^{-\lambda \alpha \kappa \mathbb{E} \left [ \mathcal{R} \right ] / (1-\kappa)}, \]
so $\sum_{i=1}^{\lfloor \alpha h(t) \rfloor} \overline{\mathcal{H}_{i}}/t$ converges in probability to $\alpha \kappa \mathbb{E} [ \mathcal{R} ] / (1-\kappa)$ which yields 
\begin{eqnarray}
\mathbb{P} \left ( \sum_{i=1}^{\lfloor \alpha h(t) \rfloor} \overline{\mathcal{H}_{i}} > tb \right ) \underset{t \rightarrow +\infty}{\longrightarrow} 1, \label{queueloilimdiscretbis9}
\end{eqnarray}
since $\alpha > b (1-\kappa) / \kappa \mathbb{E} [ \mathcal{R} ]$. 

We now study the first factors in the left hand side of \eqref{queueloilimdiscretbis2}. Since $T$ is geometric with parameter $p_t$, we have
\[ \mathbb{P} \left ( T > \alpha h(t) + 1 \right ) = (1-p_t)^{\lfloor \alpha h(t) + 1 \rfloor} = e^{\lfloor \alpha h(t) + 1 \rfloor \log \left (1 - p_t \right )}. \]
Now, since $h(t) \sim e^{\kappa \phi(t)} y_t^{1-\kappa} / \mathcal{C}'$ and $p_t \sim \mathcal{C}' e^{-\kappa \phi(t)} y_t^{\kappa}$ we get
\begin{eqnarray}
\log \left ( \mathbb{P} \left ( T > \alpha h(t) + 1 \right ) \right ) \underset{t \rightarrow +\infty}{\sim} \alpha y_t. \label{queueloilimdiscretbis10}
\end{eqnarray}

Now, putting \eqref{queueloilimdiscretbis9} and \eqref{queueloilimdiscretbis10} into \eqref{queueloilimdiscretbis2}, and combining the latter with \eqref{queueloilimdiscretbis1}, we get the result for $t$ large enough since $\alpha < u b(1-\kappa)/\kappa \mathbb{E}[\mathcal{R}]$. 

\textit{Upper bound}

%SEPARER EN DEUX OU FAIRE DIRECTEMENT LA LAPLACE DU TRUC AVEC LA GEOM ? mettre la geom obligerait a parler de cv des zeros...
Recall \eqref{queueloilimdiscretbis1} and the definitions of $(\overline{\mathcal{H}_{i}})_{i \geq 1}$, $p_t$, $T$ and $h(t)$. Let us fix $\alpha > 0$ that will be chosen latter. We have
\begin{eqnarray}
\mathbb{P} \left ( \sum_{i=1}^{k^t(t / y_t)-1} e_i S_i^t R_i^t > tb \right ) = \mathbb{P} \left ( \sum_{i=1}^{T-1} \overline{\mathcal{H}_{i}} > tb \right ) \leq \mathbb{P} \left ( T > \alpha h(t) \right ) + \mathbb{P} \left ( \sum_{i=1}^{\lfloor \alpha h(t) \rfloor} \overline{\mathcal{H}_{i}} > tb \right ). \label{queueloilimdiscrettris2}
\end{eqnarray}

Let us choose $\lambda \in ]0, \lambda_0[$ where $\lambda_0$ defined in Fact \ref{queueiid}. The second term equals 
\begin{align*}
\mathbb{P} \left ( \sum_{i=1}^{\lfloor \alpha h(t) \rfloor} y_t \overline{\mathcal{H}_{i}}/t > b y_t \right ) & \leq e^{- \lambda b y_t} \mathbb{E} \left [ \exp \left ( \lambda \sum_{i=1}^{\lfloor \alpha h(t) \rfloor} y_t \overline{\mathcal{H}_{i}}/t \right ) \right ] \\
& = e^{- \lambda b y_t} \left ( 1 + \mathbb{E} \left [ \exp \left ( \lambda y_t \overline{\mathcal{H}_{i}}/t \right ) \right ] - 1 \right )^{\lfloor \alpha h(t) \rfloor} \\
& = e^{- \lambda b y_t + \lfloor \alpha h(t) \rfloor \log \left ( 1 + \mathbb{E} \left [ \exp \left ( \lambda y_t \overline{\mathcal{H}_{i}}/t \right ) \right ] - 1 \right )}. 
\end{align*}
According to Lemma \ref{etudelaplacege}, the fact that $h(t) \sim e^{\kappa \phi(t)} y_t^{1-\kappa} / \mathcal{C}'$ and $p_t \sim \mathcal{C}' e^{-\kappa \phi(t)} y_t^{\kappa}$ we have 
\[ \lfloor \alpha h(t) \rfloor \log \left ( 1 + \mathbb{E} \left [ \exp \left ( \lambda y_t \overline{\mathcal{H}_{i}}/t \right ) \right ] - 1 \right ) \underset{t \rightarrow +\infty}{\sim} \alpha y_t \left ( 1-\mathbb{E} \left [ e^{\lambda \mathcal{R}} \right ] + \lambda \int_0^{1} x^{-\kappa} \mathbb{E} \left [ \mathcal{R} e^{\lambda x \mathcal{R}} \right ] dx \right ). \]
Thanks to the positivity of the limit in Lemma \ref{etudelaplacege} we can choose $\alpha$ such that $0 < \alpha < b \lambda / 2 ( 1-\mathbb{E} [ e^{\lambda \mathcal{R}} ] + \lambda \int_0^{1} x^{-\kappa} \mathbb{E} [ \mathcal{R} e^{\lambda x \mathcal{R}} ] dx )$. We thus get for $t$ large enough
\begin{eqnarray}
\mathbb{P} \left ( \sum_{i=1}^{\lfloor \alpha h(t) \rfloor} y_t \overline{\mathcal{H}_{i}}/t > b y_t \right ) \leq e^{- \lambda b y_t/2}. \label{queueloilimdiscrettris3}
\end{eqnarray}

Since $T$ is geometric with parameter $p_t$, we have
\begin{eqnarray}
\mathbb{P} \left ( T > \alpha h(t) \right ) = (1-p_t)^{\lfloor \alpha h(t) \rfloor} = e^{\lfloor \alpha h(t) \rfloor \log \left (1 - p_t \right )} \underset{t \rightarrow +\infty}{\approx} e^{-\alpha y_t} \leq e^{-\alpha y_t/2}, \label{queueloilimdiscrettris4}
\end{eqnarray}
where we used the equivalents for $h(t)$ and $p_t$ and where the last inequality holds for $t$ large enough. 

Now, putting \eqref{queueloilimdiscrettris3} and \eqref{queueloilimdiscrettris4} into \eqref{queueloilimdiscrettris2}, and combining the latter with \eqref{queueloilimdiscretbis1}, we get the result for $t$ large enough. 

\end{proof}

%We define the sequence $(t_n)_{n \geq 1}$ of real numbers by $t_n := r^n$ and for any $\sigma > 0$ we define the sequence $(\mathcal{B}^{\sigma}_n)_{n \geq 1}$ of events by
%\[ \mathcal{B}^{\sigma}_n := \left \{ \inf_{t \in [t_n, t_{n+1}]} \frac{\log (\log(t)) \mathcal{L}^*_X(t)}{t} \leq \sigma \right \} \]

Fix $\theta > 1$. We apply Proposition \ref{approxdupetittl} with $a=\theta^{1/3}$ and Proposition \ref{queueloilimdiscret} (the lower bound with $b=\theta^{1/3}$, $u=\theta^{1/3}$, $y_t = \theta^{1/3} x_t$ and the upper bound with $b = 1/4$, $y_t = x_t / 2$). We get: 

\begin{prop} \label{pretpourbcindepliminf}
Let $\tilde L \in ]0, \min \{ L/ 8, \lambda_0 /8 \}[$ where $L$ is the positive constant defined in Proposition \ref{queueloilimdiscret} and $\lambda_0$ is defined in Fact \ref{queueiid}. There is a positive constant $c$ such that for any $\theta > 1$ and $t$ large enough we have
\[ e^{- \theta (1-\kappa) x_t /\kappa \mathbb{E}[\mathcal{R}]} - e^{-c \phi(t)} \leq \mathbb{P} \left( \mathcal{L}^*_X(t) \leq t /x_t \right ) \leq e^{- \tilde L x_t} + e^{-c \phi(t)}. \]
\end{prop}

%\begin{prop} \label{majoparfonctliminf}
%For all $t$ large enough we have
%\[ \mathbb{P} \left( \mathcal{L}^*_X(t) \leq t /x_t \right ) \leq e^{- \lambda_0 x_t /8} + 2 \mathbb{P} \left ( Y_1^{\natural, t} \left ( Y_2^{-1, t}(1/4) \right ) \leq 2/x_t \right ) + e^{-c \phi(t)}. \]
%\end{prop}
%
%\begin{proof}
%
%
%\end{proof}

%\begin{prop} \label{queueloilimdiscretupperbound}
%There is a positive constant $c$ such that for any positive $a$ and $t$ large enough we have
%\[ \mathbb{P} \left ( Y_1^{\natural, t} \left ( Y_2^{-1, t}(a) \right ) \leq 1/x_t \right ) \leq e^{- c a x_t}. \]
%\end{prop}
%
%\begin{proof}
%
%
%
%\end{proof}

%Combining Propositions \ref{majoparfonctliminf} and \ref{queueloilimdiscretupperbound} we get: 
%\begin{prop} \label{pretpourbcfacilliminf}
%There is a positive constant $c$ such that for all $t$ large enough, 
%\[ \mathbb{P} \left( \mathcal{L}^*_X(t) \leq t /x_t \right ) \leq e^{-c x_t} + e^{-c \phi(t)}, \]
%\end{prop}

We can now prove Theorem \ref{liminfkappa<1}. 

\begin{proof} of Theorem \ref{liminfkappa<1}

Recall that $\tilde L$ is the positive constant defined in Proposition \ref{pretpourbcindepliminf} and let $x_t := 4 \log (\log(t)) /2\tilde L$. Note that such a choice of $x_t$ satisfies \eqref{defxtlim} with $\mu = 2$ and $D = 4/2\tilde L$. We define the events
\[ \mathcal{A}_n := \left \{ \inf_{t \in [2^n, 2^{n+1}]} \frac{\mathcal{L}^*_X(t)}{t /\log (\log(t))} \leq \tilde L/4 \right \}. \]

From the increase of $\mathcal{L}^*_X(.)$ and Proposition \ref{pretpourbcindepliminf} (the upper bound applied with $t = 2^n$) we have for $n$ large enough, 
\begin{align*}
\mathbb{P} \left ( \mathcal{A}_n \right ) \leq \mathbb{P} \left ( \mathcal{L}^*_X(2^{n}) \leq 2^{n+1} \tilde L / 4 \log (\log(2^n)) \right ) & \leq \exp \left ( -2 \log (\log(2^n)) \right ) + e^{-c \phi(2^n)} \\
& = (\log(2))^{-2} n^{-2} + e^{-c \phi(2^n)}. 
\end{align*}
Since $e^{-c \phi(2^{n})} = e^{-c (\log \log(2^{n}))^2} \leq n^{-2}$ for $n$ large enough, the above is the general term of a converging series so, using the Borel-Cantelli lemma we deduce that $\mathbb{P}$-almost surely, 
\[ \liminf_{t \rightarrow +\infty} \frac{\mathcal{L}^*_X(t)}{t /\log (\log(t))} \geq \tilde L/4, \]
so the $\liminf$ is positive. 

We now prove the upper bound for the $\liminf$. Let $a>1$ and let $t_n$, $u_n$, $v_n$ (defined from this $a$) and $X^n$ be as in Subsection \ref{decompindeppart} (recall that $X^n - v_n$ is equal in law to $X$ under the annealed probability $\mathbb{P}$). We define
\[ \mathcal{B}_n := \left \{ \frac{\mathcal{L}^{*}_{X^n} (t_n)}{t_n /\log (\log(t_n))} \leq \frac{a^2 (1-\kappa)}{\kappa \mathbb{E}[\mathcal{R}]} \right \}, \ \ \ \mathcal{B}_n' := \left \{ \frac{\mathcal{L}^{*}_{X} (H(v_n))}{t_n /\log (\log(t_n))} \leq \frac1{n} \right \}. \]
We also define $\mathcal{C}_n$ and $\mathcal{D}_n$ to be as in Lemma \ref{lemprindep} and $\mathcal{E}_n := \mathcal{B}_n \cap \mathcal{C}_n$. 
We define $x_t := \kappa \mathbb{E}[\mathcal{R}] \log (\log(t)) /a^2 (1-\kappa)$. Note that such a choice of $x_t$ satisfies \eqref{defxtlim} with $\mu = 2$ and $D = \kappa \mathbb{E}[\mathcal{R}] /a^2 (1-\kappa)$. Recall the notation $\mathcal{L}^{*, +}_X$ defined in Subsection \ref{factnot}. Let us choose $K$ as in Lemma \ref{infsupdiff}, $\eta$ and $C$ as in Lemma \ref{tlenpetittpsatt} of the next section and $Q$ be as defined in the next section. According to \eqref{queuetlneg} and Lemma \ref{tlenpetittpsatt} applied with $u = t_n /n \log (\log(t_n)) $, $v = v_{n}$ we get for all $n$ large enough, 
%ATTENTION ATTENTION ATTENTION: ON NE MAJORE QUE LE TEMPS LOCAL SUR LES POSITIFS AVEC CE LEMME -> IL FAUT COMBINER AVEC UN AUTRE SUR LE TL EN LES NEGATIFS 
\begin{align*}
\mathbb{P} \left ( \overline{\mathcal{B}_n'} \right ) & \leq \mathbb{P} \left ( \sup_{]-\infty, 0]} \mathcal{L}_X(+\infty, .) > t_n / n \log (\log(t_n)) \right ) + \mathbb{P} \left ( \frac{\mathcal{L}^{*, +}_{X} (H(v_n))}{t_n /\log (\log(t_n))} > \frac1{n} \right ) \\
& \leq K (t_n / n \log (\log(t_n)))^{-\kappa/(2+\kappa)} + C \left ( v_n /Q + v_n^{7/8} \right ) ( n \log (\log(t_n)) )^{\kappa} / (t_n)^{\kappa} + v_n^{- \eta}. 
\end{align*}
From the definition of $t_n$, the fact that $n^{\kappa} v_n / t_n^{\kappa} = n^{\kappa} e^{- \kappa 2 a n^{a-1}/3}$ and $\log (\log(t_n)) = a \log(n)$, we get
\begin{eqnarray}
\sum_{n \geq 1} \mathbb{P} \left ( \overline{\mathcal{B}_n'} \right ) < +\infty. \label{negtpsloctpsatt}
\end{eqnarray}

According to Proposition \ref{pretpourbcindepliminf} (the lower bound applied with $t=t_n$, $\theta = a$) and the definition of $t_n$ we have 
\[ \mathbb{P} \left ( \mathcal{B}_n \right ) \geq e^{- \log (\log(t_n))/a} - e^{-c \phi(t_n)} = e^{- a \log (n)/a} - e^{-c \phi(t_n)}= n^{-1} - e^{-c \phi(t_n)}. \]
Since $e^{-c \phi(t_n)} = e^{-c (\log \log(t_n))^2} \leq n^{-2}$ for $n$ large enough, we get
\begin{eqnarray}
\sum_{n \geq 1} \mathbb{P} \left ( \mathcal{B}_n \right ) = +\infty. \label{serieinfliminf}
\end{eqnarray}

Then, the combination of \eqref{serieinfliminf} and \eqref{inegindep3} yields
\begin{eqnarray}
\sum_{n \geq 1} \mathbb{P} \left ( \mathcal{E}_n \right ) \geq \sum_{n \geq 1} \mathbb{P} \left ( \mathcal{B}_n \right ) - \sum_{n \geq 1} \mathbb{P} \left ( \overline{\mathcal{C}_n} \right ) = +\infty. \label{sommeenkappa<12}
\end{eqnarray}

As in the proof of Theorem \ref{limsupenfctdelefttail}, we see that each event $\mathcal{E}_n$ belongs to the $\sigma$-field $\sigma ( V(s) - V(u_n), u_n \leq s \leq u_{n+1}, \ X(t), H(v_n) \leq t \leq H(v_n) + T_n)$ so the events $(\mathcal{E}_n)_{n \geq 1}$ are independent. Combining this independence with \eqref{sommeenkappa<12} and the Borel-Cantelli Lemma we get that $\mathbb{P}$-almost surely, the event $\mathcal{E}_n$ is realized infinitely many often. Combining with \eqref{negtpsloctpsatt} and the Borel-Cantelli Lemma we get that $\mathbb{P}$-almost surely, the event $\mathcal{B}_n' \cap \mathcal{B}_n \cap \mathcal{C}_{n}$ is realized infinitely many often. For $n$ such that this event is realized we have
\begin{eqnarray}
\frac{\mathcal{L}^{*}_{X} (H(v_n) + t_n)}{t_n /\log (\log(t_n))} \leq \frac{\mathcal{L}^{*}_{X^n} (t_n)}{t_n /\log (\log(t_n))} + \frac{\mathcal{L}^{*}_{X} (H(v_n))}{t_n /\log (\log(t_n))} \leq \frac{a^2 (1-\kappa)}{\kappa \mathbb{E}[\mathcal{R}]} + \frac1{n}. \label{equivh(vn)2}
\end{eqnarray}
%IL FAUT UN LEMME POUR MAJORER LE TL EN UN TEMPS D'ATTEINTE: ON PEUT LE FAIRE EN SUIVANT LA MEME LOGIQUE QUE DANS LA SECTION $\kappa > 1$. D'AILLEURS ON MONTRE EXACTEMENT CA POUR CONTROLER $\overline{\mathcal{B}_n}$ DANS LE BOREL CANTELLI SUR LA LIMINF QUAND $\kappa > 1$ -> FAIRE UN LEMME QUI SERT POUR LES 2
Recall that according to \eqref{inegindep4} and the Borel-Cantelli Lemma we have $\mathbb{P}$-almost surely
\[ H(v_n) + t_n \underset{n \rightarrow +\infty}{\sim} t_n, \]
so combining with \eqref{equivh(vn)2} we deduce that $\mathbb{P}$-almost surely, 
\[ \liminf_{t \rightarrow +\infty} \frac{\mathcal{L}^{*}_{X} (t)}{t /\log (\log(t))} \leq \frac{a^2 (1-\kappa)}{\kappa \mathbb{E}[\mathcal{R}]}, \]
and letting $a$ go to $1$ we get the asserted upper bound for the $\liminf$.

%Let us define $M := (1-\kappa)/\kappa \mathbb{E}[\mathcal{R}]$, the expected upper bound for the $\liminf$. Let $R > 2$. We define the events
%\begin{align*}
%\mathcal{C}_n & := \left \{ \frac{\log (\log(R^{n+1})) \mathcal{L}^{*}_{X(H(2 R^{n}) + .)} \left (\tau \left ( X(H(2 R^{n}) + .), R^{n+1} \right ) \right )}{R^{n+1}} \leq M \right \}, \\
%\mathcal{D}_n & := \left \{ \tau \left ( X(H(2 R^{n}) + .), R^{n+1} \right ) < \tau \left ( X(H(2 R^{n}) + .), R^{n} \right ) \right \}, \\
%\mathcal{E}_n & := \mathcal{C}_n \cap \mathcal{D}_n. 
%\end{align*}
%
%Let us define $s_n := (R-2) \log (\log(R^{n+1})) / M R$. From the Markov property at time $H(2 R^{n})$, the definition of $s_n$, Proposition \ref{pretpourbc} applied with ..., we have
%\begin{align*}
%\mathbb{P} \left ( \mathcal{C}_n \right ) & = \mathbb{P} \left ( \mathcal{L}^*_X((R-2) R^{n}) \leq M R^{n+1} / (\log (\log(R^{n+1}))) \right ) \nonumber \\
%& = \mathbb{P} \left ( \mathcal{L}^*_X((R-2) R^{n}) \leq (R-2) R^{n} / s_n \right ) \nonumber \\
%& \geq e^{- truc (1-\kappa) s_n /\kappa \mathbb{E}[\mathcal{R}]} - ... \\
%& = (\log(R))^{- truc \frac{R-2}{R}} \times (n+1)^{- truc \frac{R-2}{R}} - ...
%\end{align*}
%
%ON PEUT REMPLACER $R$ PAR $2$ PAR EXEMPLE. PAREIL POUR LA LIMSUP ?

\end{proof}

\section{Almost sure behavior when $\kappa > 1$} \label{bigkappa}

In this section we prove Theorems \ref{limsupkappa>1} and \ref{liminfkappa>1}. Let us first recall some facts and notations from \cite{Singh} and \cite{caslevyvech}. Our proof is based on the study of the so-called generalized Ornstein-Uhlenbeck process defined by
\[ Z(x) := e^{V(x)} R \left ( \int_0^x e^{-V(y)} dy \right ), \]
where $R$ is a two-dimensional squared Bessel process independent from $V$. Let $L$ be the local time of $Z$ for the position $1$, $n$ the associated excursion measure, and $L^{-1}$ the right continuous inverse of $L$. We denote by $\xi$ a generic excursion. Let us denote by $Q$ the positive constant denoted by $n [\zeta]$ in Section 5 of \cite{Singh}. Recall also the notations $K$ and $m$ defined in the Introduction ($m$ only appears in the context where $\kappa > 1$, where it is properly defined). We have: 
 
\begin{fact} \label{supdeZ}

There is $\eta > 0$ and $r_0 > 0$ such that for all $r \geq r_0$ and $h>1$ we have
\begin{eqnarray}
e^{-(r/Q + r^{7/8}) n(\sup \xi > h)} - r^{-\eta} \leq \mathbb{P} \left ( \sup_{x \in [0,r]} Z(x) \leq h \right ) \leq e^{-(r/Q - r^{7/8}) n(\sup \xi > h)} + r^{-\eta}, \label{supz}
\end{eqnarray}
\begin{eqnarray}
n(\sup \xi > h) \underset{h \rightarrow +\infty}{\sim} Q 2^{\kappa} \Gamma(\kappa) \kappa^2 K/ h^{\kappa}. \label{singh}
\end{eqnarray}
\end{fact}

The first point is Lemma 2.3 of \cite{caslevyvech} while the second point is Proposition 5.1 of \cite{Singh}. Note that Fact \ref{supdeZ} is true for a general positive $\kappa$ and not only for $\kappa > 1$. 

Let us recall the link between $Z$ and the local time until the hitting times. The local time at point $x$ and within the hitting time $H(r)$ is given by: 
%\begin{equation}
\[ \mathcal{L}_X(H(r),x)= e^{- V(x)}\mathcal{L}_B(\tau (B, A_V(r)),A_V(x)). \]
%\label{expretl2}
%\end{equation}
$\mathcal{M}_1(r)$ and $\mathcal{M}_2(r)$ denote respectively the supremum of the above expression for $x \in ]-\infty, 0[$ and $x \in [0, +\infty[$. The supremum of the local time until instant $H(r)$ can be written
\begin{equation}
\mathcal{L}^*_X(H(r))= \max \{ \mathcal{M}_1(r), \mathcal{M}_2(r) \}, \label{expretl3}
\end{equation}
where 
\begin{eqnarray}
\mathcal{M}_1(r) \leq \mathcal{M}_1(+\infty) < +\infty \text{ and, as in (2.10) of \cite{caslevyvech}, } \mathcal{M}_2(r) \overset{\mathcal{L}}{=} \sup_{x \in [0,r]} Z(x). \label{tlneg}
\end{eqnarray}
%and 
%\[ \mathcal{M}_2(r) := \sup_{x \in [0,r]} \mathcal{L}_X(H(r),x) = \sup_{x \in [0,r]} e^{- V(x)}\mathcal{L}_B(\tau (B, A_V(r)),A_V(x)). \]
%
%and, as in \cite{caslevyvech}, 
%\begin{eqnarray}
%\mathcal{M}_2(r) \overset{\mathcal{L}}{=} \sup_{x \in [0,r]} Z(x). \label{expretl4}
%\end{eqnarray}This allows to give a general bound for the local time at a hitting-time: 
The boundedness of $\mathcal{M}_1(r)$ is justified in (2.9) of \cite{caslevyvech}. 

We can now study the behavior of the local time at a hitting time. This allows to prove the following useful lemma. 

\begin{lemme} \label{tlenpetittpsatt}

There exist $\eta > 0$, a positive constant $C$, $u_0 > 0$ and $v_0 > 0$ such that
\[ \forall u \geq u_0, v \geq v_0, \ \mathbb{P} \left ( \mathcal{L}^{*, +}_X(H(v)) > u \right ) \leq C ( v/Q + v^{7/8} ) u^{-\kappa} + v^{- \eta}, \]
where $\mathcal{L}^{*, +}_X$ is as defined in Subsection \ref{factnot}. 
\end{lemme}

\begin{proof}

From the definition of $\mathcal{M}_2$ and \eqref{tlneg} we have
\begin{eqnarray}
\mathbb{P} \left ( \mathcal{L}^{*, +}_X(H(v)) \geq u \right ) = \mathbb{P} \left ( \mathcal{M}_2(v) > u \right ) = \mathbb{P} \left ( \sup_{x \in [0, v]} Z(x) > u \right ). \label{tlenpetittpsatt1}
\end{eqnarray}
Now let us choose $\eta$ and $v_0$ such that \eqref{supz} is true for all $r \geq v_0$ and $h > 1$ with this $\eta$. We choose $C > Q 2^{\kappa} \Gamma(\kappa) \kappa^2 K$ and $u_0$ such that $n ( \sup \xi > u ) \leq C u^{-\kappa}$ for all $u \geq u_0$. Such a $u_0$ exists thanks to \eqref{singh}. For $u \geq u_0$ and $v \geq v_0$ we have
\begin{align*}
\mathbb{P} \left ( \sup_{x \in [0, v]} Z(x) > u \right ) & \leq 1 - e^{- ( v/Q + v^{7/8} ) \times n ( \sup \xi > u )} + v^{- \eta} \\
& \leq ( v/Q + v^{7/8} ) \times n ( \sup \xi > u ) + v^{- \eta} \\
& \leq C ( v/Q + v^{7/8} ) u^{-\kappa} + v^{- \eta}. 
\end{align*}
Putting into \eqref{tlenpetittpsatt1} we get the result. 

\end{proof}

\begin{remarque} \label{allkappa}
Neither Lemma \ref{tlenpetittpsatt} nor its proof require the hypothesis that $\kappa > 1$. The lemma is thus true whatever is the value of $\kappa$. 
\end{remarque}

We need to study the supremum of the local time until a deterministic time. The following fact, from Lemma 2.1 of \cite{caslevyvech}, says that we can replace a deterministic time by a hitting time when $\kappa > 1$. \textbf{We now assume $\kappa > 1$ until the end of this section. } We have: 

\begin{fact} \label{chgtvar}

%%Recall the definition of $H_{-}$ in Subsection \ref{factsandnotations}. 
%For all $r$ large enough we have
%%\[ \mathbb{P} \left ( H_-(+\infty) > r \right ) \leq C r^{-\kappa/(2+\kappa)} \ \ \ \text{and} \ \ \ \mathbb{P} \left ( \inf_{]-\infty, 0]} \mathcal{L}_X(+\infty, .) > r \right ) \leq 3 r^{-\kappa/(2+\kappa)}. \]
%\begin{eqnarray}
%\mathbb{P} \left ( \inf_{]-\infty, 0]} \mathcal{L}_X(+\infty, .) > r \right ) \leq 3 r^{-\kappa/(2+\kappa)}. \label{queuetlneg}
%\end{eqnarray}
%%C'EST PLUS SIMPLE DE DIRE (APRES AVOIR DIT QUE $\mathcal{M}_1(+\infty) < + \infty$) QU'ON ETUDIE QUE $\mathcal{L}^{*, +}_X$: LE SUP DU TL SUR LES POSITIFS -> REMPLACER PAR CA: OUI MAIS ON NE PEUT FAIRE CA QUE POUR $\mathcal{B}_n$. 
For any $\alpha \in ]\max \{ 3/4, 1/ \kappa \}, 1[$ their exists $\eta > 0$ such that for $r$ large enough we have
\begin{eqnarray}
\mathbb{P} \left ( H(r/m - r^{\alpha}) \leq r \leq H(r/m + r^{\alpha}) \right ) & \geq 1 - r^{-\eta}. \label{tpsatt}
\end{eqnarray}

\end{fact}

Let $\alpha \in ]\max \{ 3/4, 1/ \kappa \}, 1[$ be fixed until the end of this section and $\eta > 0$ be small enough so that both \eqref{supz} and \eqref{tpsatt} are satisfied (with this $\alpha$). 

\subsection{The $\liminf$}

In this subsection we prove Theorem \ref{liminfkappa>1}. Let us define $J := 2 (\Gamma(\kappa) \kappa^2 K/m)^{1/\kappa}$, the expected $\liminf$. We begin to prove that 
\begin{eqnarray}
\liminf_{t \rightarrow +\infty} \frac{\mathcal{L}^*_X(t)}{(t/ \log (\log(t)))^{1/\kappa}} \geq J. \label{minoliminf}
\end{eqnarray}

Let $a > 1$ and define the events
\[ \mathcal{A}_n := \left \{ \inf_{t \in [a^n, a^{n+1}]} \frac{\mathcal{L}^*_X(t)}{(t/ \log (\log(t)))^{1/\kappa}}  \leq \frac{J}{a^{3/\kappa}} \right \}. \]

From the increase of $\mathcal{L}^*_X(.)$, \eqref{tpsatt}, \eqref{expretl3}, \eqref{tlneg}, and \eqref{supz}, we have
\begin{align}
\mathbb{P} \left ( \mathcal{A}_n \right ) & \leq \mathbb{P} \left ( \mathcal{L}^*_X(a^n) \leq J (a^{n-2}/ \log (\log(a^n)))^{1/\kappa} \right ) \nonumber \\
& \leq \mathbb{P} \left ( \mathcal{L}^*_X(H(a^n/m - a^{\alpha n})) \leq J (a^{n-2}/ \log (\log(a^n)))^{1/\kappa} \right ) + a^{-n \eta} \nonumber \\
& \leq \mathbb{P} \left ( \mathcal{M}_2(a^n/m - a^{\alpha n}) \leq J (a^{n-2}/ \log (\log(a^n)))^{1/\kappa} \right ) + a^{-n \eta} \nonumber \\
& = \mathbb{P} \left ( \sup_{x \in [0, a^n/m - a^{\alpha n}]} Z(x) \leq J (a^{n-2}/ \log (\log(a^n)))^{1/\kappa} \right ) + a^{-n \eta} \nonumber \\
& \leq e^{- \big ( a^n/Q m - a^{\alpha n}/Q - (a^n/m - a^{\alpha n})^{7/8} \big ) \times n \big (\sup \xi > J (a^{n-2}/ \log (\log(a^n)))^{1/\kappa} \big )} + a^{-n \eta} \nonumber \\
& +  (a^n/m - a^{\alpha n})^{-\eta}. \label{seriecv}
\end{align}
According to the equivalent given by \eqref{singh}, the exponent in the above expression is, for $n$ large enough, less than $-a \log (\log(a^n))$, so for such large $n$, 
\[ e^{- \big ( a^n/Q m - a^{\alpha n}/Q - (a^n/m - a^{\alpha n})^{7/8} \big ) \times n \big (\sup \xi > J (a^{n-2}/ \log (\log(a^n)))^{1/\kappa} \big )} \leq (n \log(a))^{-a}. \]
The other two terms in the right hand side of \eqref{seriecv} are also general terms of converging series so we obtain, 
\[ \sum_{n \geq 1} \mathbb{P} ( \mathcal{A}_n ) < +\infty. \]

According to the Borel-Cantelli lemma we get
\[ \liminf_{t \rightarrow +\infty} \frac{\mathcal{L}^*_X(t)}{(t/ \log (\log(t)))^{1/\kappa}} \geq J/a^{3/\kappa}, \]
in which we can let $a$ go to $1$ which yields \eqref{minoliminf}. We now prove that 
\begin{eqnarray}
\liminf_{t \rightarrow +\infty} \frac{\mathcal{L}^*_X(t)}{(t/ \log (\log(t)))^{1/\kappa}} \leq J. \label{majoliminf}
\end{eqnarray}

Let us fix $a > 0$, $u_n := n^{2n}$, $v_n := u_n / m + u_n^{\alpha} = n^{2n} / m + n^{2 \alpha n}$ and $X^n := X(H(2 v_n) + .)$, the diffusion shifted by the hitting time of $2 v_n$. Note that from the Markov property for $X$ at time $H(2 v_n)$ and the stationarity of the increments of $V$, $X^n - 2 v_n$ is equal in law to $X$ under the annealed probability $\mathbb{P}$. We take $n$ so large such that $2 v_n < v_{n+1}$ and define the events
\begin{align*}
\mathcal{B}_n & := \left \{ \frac{\mathcal{L}^{*, +}_X(H(2 v_n))}{(u_{n+1}/ \log (\log(u_{n+1})))^{1/\kappa}}  \leq a J \right \}, \\
\mathcal{C}_n & := \left \{ \frac{\mathcal{L}^{*}_{X^n} \left (\tau \left ( X^n, v_{n+1} \right ) \right )}{(u_{n+1}/ \log (\log(u_{n+1})))^{1/\kappa}}  \leq (1+a) J \right \}, \\
\mathcal{D}_n & := \left \{ \tau \left ( X^n, v_{n+1} \right ) < \tau \left ( X^n, v_{n} \right ) \right \}, \\
\mathcal{E}_n & := \mathcal{C}_n \cap \mathcal{D}_n, \\
\mathcal{F}_n & := \left \{ \mathcal{L}^{*, +}_X(u_n) \leq \mathcal{L}^{*, +}_X(H(v_n)) \right \}. 
\end{align*}

Recall that $\eta >$ has been fixed so that \eqref{supz} is satisfied. For this $\eta$ and for $C> Q 2^{\kappa} \Gamma(\kappa) \kappa^2 K$, the inequality of Lemma \ref{tlenpetittpsatt} is true for $u$ and $v$ large enough. According to this lemma applied with $u = a J (u_{n+1}/ \log (\log(u_{n+1})))^{1/\kappa}$, $v = 2v_{n}$ we get for all $n$ large enough, 
\[ \mathbb{P} \left ( \overline{\mathcal{B}_n} \right ) \leq C \left ( 2 v_n /Q + (2 v_n)^{7/8} \right ) \log (\log(u_{n+1})) / J^{\kappa} a^{\kappa} u_{n+1} + (2 v_n)^{- \eta}. \]

%First, from \eqref{expretl3} and Fact \ref{supdeZ}, we have
%\begin{align*}
%\mathbb{P} \left ( \overline{\mathcal{B}_n} \right ) & = \mathbb{P} \left ( \mathcal{M}_2(2v_{n}) > \eta J (u_{n+1}/ \log (\log(u_{n+1})))^{1/\kappa} \right ) \\
%& = \mathbb{P} \left ( \sup_{[0, 2v_{n}]} Z > \eta J (u_{n+1}/ \log (\log(u_{n+1})))^{1/\kappa} \right ) \\
%& \leq 1 - e^{- \big ( 2 v_n /Q + (2 v_n)^{\alpha} \big ) \times n \big ( \sup \xi > \eta J (u_{n+1}/ \log (\log(u_{n+1})))^{1/\kappa} \big )} + (2 v_n)^{- \eta} \\
%& \leq \big ( 2 v_n /Q + (2 v_n)^{\alpha} \big ) \times n \big ( \sup \xi > \eta J (u_{n+1}/ \log (\log(u_{n+1})))^{1/\kappa} \big ) + (2 v_n)^{- \eta} \\
%& \leq 4 m v_{n} \log (\log(u_{n+1})) / \eta^{\kappa} u_{n+1} + (2 v_n)^{- \eta}
%\end{align*}
%where we use the equivalent \eqref{singh} for the last inequality, which is thus true for $n$ large enough. 
Since, for $n$ large enough, $v_{n} / u_{n+1} \leq 1/m n^2$ and $\log (\log(u_{n+1})) \sim \log(n)$ we can deduce that
\begin{eqnarray}
\sum_{n \geq 1} \mathbb{P} \left ( \overline{\mathcal{B}_n} \right ) < +\infty. \label{sommebn}
\end{eqnarray}

From the equality in law between $X^n - 2 v_n$ and $X$ under $\mathbb{P}$, , \eqref{expretl3}, \eqref{tlneg}, \eqref{queuetlneg}, \eqref{supz}
\begin{align}
\mathbb{P} \left ( \mathcal{C}_n \right ) & = \mathbb{P} \left ( \mathcal{L}^*_X \left (H(v_{n+1} - 2v_n) \right ) \leq (1+a) J \left ( u_{n+1} / \log (\log(u_{n+1})) \right )^{1/\kappa} \right ) \nonumber \\
& \geq \mathbb{P} \left ( \mathcal{M}_2(H(v_{n+1} - 2v_n)) \leq (1+a) J \left ( u_{n+1} / \log (\log(u_{n+1})) \right )^{1/\kappa} \right ) \nonumber \\
& - \mathbb{P} \left ( \sup_{]-\infty, 0]} \mathcal{L}_X(+\infty, .) > (1+a) J \left ( u_{n+1} / \log (\log(u_{n+1})) \right )^{1/\kappa} \right ) \nonumber \\
& \geq \mathbb{P} \left ( \sup_{x \in [0, v_{n+1} - 2v_n]} Z(x) \leq (1+a) J \left ( u_{n+1} / \log (\log(u_{n+1})) \right )^{1/\kappa} \right ) \nonumber \\
& - K ((1+a) J)^{-\kappa/(2 +\kappa)} \times \left ( u_{n+1} / \log (\log(u_{n+1})) \right )^{-1/(2+\kappa)} \nonumber \\
& \geq e^{- \big ( (v_{n+1} - 2v_n)/Q + (v_{n+1} - 2v_n)^{7/8} \big ) \times n \big (\sup \xi > (1+a) J \left ( u_{n+1} / \log (\log(u_{n+1})) \right )^{1/\kappa} \big )} \nonumber \\
& - (v_{n+1} - 2v_n)^{-\eta} - K ((1+a) J)^{-\kappa/(2 +\kappa)} \times \left ( u_{n+1} / \log (\log(u_{n+1})) \right )^{-1/(2+\kappa)}. \label{sommecn1}
\end{align}

According to the equivalent given by \eqref{singh} and the definitions of $u_n$ and $v_n$, the exponent in the above expression is equivalent to $(1+a)^{-\kappa} \log (\log(u_{n+1})) \sim (1+a)^{-\kappa} \log(n)$ so for $n$ large enough, 
\[ e^{- \big ( (v_{n+1} - 2v_n)/Q + (v_{n+1} - 2v_n)^{7/8} \big ) \times n \big (\sup \xi > (1+a) J \left ( u_{n+1} / \log (\log(u_{n+1})) \right )^{1/\kappa} \big )} \geq \frac1{n}, \]
and the remaining terms in the right hand side of \eqref{sommecn1} are the general terms of converging series. We thus get
\begin{eqnarray}
\sum_{n \geq 1} \mathbb{P} \left ( \mathcal{C}_n \right ) = +\infty. \label{sommecn2}
\end{eqnarray}

\begin{align*}
\mathbb{P} \left ( \overline{\mathcal{D}_n} \right ) & = \mathbb{P} \left ( \tau \left ( X^n, v_{n+1} \right ) > \tau \left ( X^n, v_{n} \right ) \right ) \nonumber \\
& \leq \mathbb{P} \left ( \tau \left ( X^n, v_{n} \right ) < \tau \left ( X^n, +\infty \right ) \right ) \nonumber \\
& \leq \mathbb{P} \left ( \tau \left ( X, -v_{n} \right ) < \tau \left ( X, +\infty \right ) \right ) \label{majocompdn0} \\
& = \mathbb{P} \left ( \inf_{[0, +\infty[} X < -v_{n} \right ), 
%& \leq \mathbb{P} \left ( \inf_{[0, +\infty[} X(H(2v_{n}) + .) < v_{n} \right ) = \mathbb{P} \left ( \inf_{[0, +\infty[} X < - v_{n} \right ), 
\end{align*}
where we used the equality in law between $X^n - 2 v_n$ and $X$ under $\mathbb{P}$. Combining with \eqref{minoinfndiff} applied with $r = v_n$ we get
\begin{eqnarray}
\sum_{n \geq 1} \mathbb{P} \left ( \overline{\mathcal{D}_n} \right ) < +\infty. \label{sommedn}
\end{eqnarray}

Then, the combination of \eqref{sommecn2} and \eqref{sommedn} yields
\begin{eqnarray}
\sum_{n \geq 1} \mathbb{P} \left ( \mathcal{E}_n \right ) \geq \sum_{n \geq 1} \mathbb{P} \left ( \mathcal{C}_n \right ) - \sum_{n \geq 1} \mathbb{P} \left ( \overline{\mathcal{D}_n} \right ) = +\infty. \label{sommeen}
\end{eqnarray}

According to the definitions of the sequences $(u_n)_{n \geq 1}$ and $(u_n)_{n \geq 1}$, to \eqref{tpsatt}, and to the increase of $\mathcal{L}^{*, +}_X$, we have, for $n$ large enough, $\mathbb{P} (\mathcal{F}_n) \leq u_n^{-\eta}$, so
\begin{eqnarray}
\sum_{n \geq 1} \mathbb{P} \left ( \overline{\mathcal{F}_n} \right ) < +\infty. \label{sommefn}
\end{eqnarray}

%The events $\mathcal{E}_n$ are independent since for each $n$, $\mathcal{E}_n$ belongs to the $\sigma$-field $\sigma ( V(s) - V(v_n), v_n \leq s \leq v_{n+1}, \ X(t), H(2v_n) \leq t \leq \min (\tau(X(H(2v_n) + .), v_n), \tau(X(H(2v_n) + .), v_{n+1}) ))$. Combining this independence with \eqref{sommeen} and the Borel-Cantelli Lemma we get that almost surely, the event $\mathcal{E}_n$ is realized infinitely many often. 
Note that each event $\mathcal{E}_n$ belongs to the $\sigma$-field $\sigma ( V(s) - V(v_n), v_n \leq s \leq v_{n+1}, \ X(t), H(2v_n) \leq t \leq \min (\tau(X^n, v_n), \tau(X^n, v_{n+1}) ))$, in other words, it only depends on the diffusion between times $H(2v_n)$ and $\min (\tau(X^n, v_n), \tau(X^n, v_{n+1}))$ and on the environment between positions $v_n$ and $v_{n+1}$. From the Markov property and the independence of the increments of the environment, we get that the events $(\mathcal{E}_n)_{n \geq 1}$ are independent. Combining this independence with \eqref{sommeen} and the Borel-Cantelli Lemma we get that $\mathbb{P}$-almost surely, the event $\mathcal{E}_n$ is realized infinitely many often. 

Combining \eqref{sommebn}, \eqref{sommefn} and the Borel-Cantelli Lemma we get that $\mathbb{P}$-almost surely the event $\mathcal{B}_n \cap \mathcal{F}_n$ is realized for all large $n$. We deduce that $\mathbb{P}$-almost surely, the event $\mathcal{B}_n \cap \mathcal{C}_n \cap \mathcal{D}_n \cap \mathcal{F}_{n+1}$ is realized infinitely many often. Then, for $n$ such that this event is realized we have, 
\begin{align*}
\frac{\mathcal{L}^{*, +}_X(u_{n+1})}{(u_{n+1}/ \log (\log(u_{n+1})))^{1/\kappa}} & \leq \frac{\mathcal{L}^{*, +}_X(H(v_{n+1}))}{(u_{n+1}/ \log (\log(u_{n+1})))^{1/\kappa}} \\
& \leq \frac{\mathcal{L}^{*, +}_X(H(2 v_n))}{(u_{n+1}/ \log (\log(u_{n+1})))^{1/\kappa}} + \frac{\mathcal{L}^{*}_{X^n} \left (\tau \left ( X^n, v_{n+1} \right ) \right )}{(u_{n+1}/ \log (\log(u_{n+1})))^{1/\kappa}} \\
& \leq a J + (1+a) J, 
\end{align*}
so 
\[ \liminf_{t \rightarrow +\infty} \frac{\mathcal{L}^{*, +}_X(t)}{(t/ \log (\log(t)))^{1/\kappa}} \leq (1 + 2 a)J. \]

Now, letting $a$ go to $0$ and combining with the finiteness of $\sup_{]-\infty, 0[} \mathcal{L}_X (+\infty)$ (see \eqref{expretl3} and \eqref{tlneg}) we obtain \eqref{majoliminf} so Theorem \ref{liminfkappa>1} is proved. 

\subsection{The $\limsup$}

In this subsection we prove Theorem \ref{limsupkappa>1}. First, let us assume that $\int_1^{+\infty} \frac{(f(t))^{\kappa}}{t} dt < + \infty$ and prove that
\begin{eqnarray}
\limsup_{t \rightarrow +\infty} \frac{f(t) \mathcal{L}^*_X(t)}{t} = 0. \label{majolimsup}
\end{eqnarray}

According to Remark \ref{serieint} the condition $\int_1^{+\infty} \frac{(f(t))^{\kappa}}{t} dt < + \infty$ is equivalent to
\begin{eqnarray}
\sum_{n=1}^{+\infty} (f(2^n))^{\kappa} < + \infty. \label{majolimsuphyp}
\end{eqnarray}

Let us fix $a > 0$ and define the events
\[ \mathcal{A}_n := \left \{ \sup_{t \in [2^n, 2^{n+1}]} \frac{f(t) \mathcal{L}^{*, +}_X(t)}{t^{1/\kappa}}  \geq a \right \}. \]

From the increase of $\mathcal{L}^{*, +}_X(.)$, \eqref{tpsatt}, \eqref{expretl3}, \eqref{tlneg}, and \eqref{supz}, we have
\begin{align*}
\mathbb{P} \left ( \mathcal{A}_n \right ) & \leq \mathbb{P} \left ( \mathcal{L}^{*, +}_X(2^{n+1}) \geq 2^{n/\kappa}a / f(2^n) \right ) \nonumber \\
& \leq \mathbb{P} \left ( \mathcal{L}^{*, +}_X \left ( H(2^{n+1}/m + 2^{\alpha (n+1)}) \right ) \geq 2^{n/\kappa}a / f(2^n) \right ) + 2^{-\eta (n+1)} \nonumber \\
& = \mathbb{P} \left ( \mathcal{M}_2 \left ( 2^{n+1}/m + 2^{\alpha (n+1)} \right ) \geq 2^{n/\kappa}a / f(2^n) \right ) + 2^{-\eta (n+1)} \nonumber \\
& = \mathbb{P} \left ( \sup_{x \in [0, 2^{n+1}/m + 2^{\alpha (n+1)}]} Z(x) \geq 2^{n/\kappa}a / f(2^n) \right ) + 2^{-\eta (n+1)} \nonumber \\
& \leq 1 - e^{- \big ( 2^{n+1}/Q m + 2^{\alpha (n+1)}/Q + \left (2^{n+1} /m + 2^{\alpha (n+1)} \right )^{7/8} \big ) \times n \big (\sup \xi > 2^{n/\kappa}a / f(2^n) \big )} \nonumber \\
& + 2^{-\eta (n+1)} + \left (2^{n+1}/m + 2^{\alpha (n+1)} \right )^{-\eta} \nonumber \\
& \leq \big ( 2^{n+1}/Q m + 2^{\alpha (n+1)}/Q + \left (2^{n+1} /m + 2^{\alpha (n+1)} \right )^{7/8} \big ) \times n \big (\sup \xi > 2^{n/\kappa}a / f(2^n) \big ) \nonumber \\
& + 2^{-\eta (n+1)} + \left (2^{n+1}/m + 2^{\alpha (n+1)} \right )^{-\eta}. 
\end{align*}

According to \eqref{singh}, the first term in the right hand side is equivalent to 
\[2^{1 +\kappa} \Gamma(\kappa) \kappa^2 K (f(2^n))^{\kappa} /m a^{\kappa}\] 
which is the general term of a convergent series, according to \eqref{majolimsuphyp}. The two remaining terms are also the general terms of convergent series so we get
\[ \sum_{n \geq 1} \mathbb{P} ( \mathcal{A}_n ) < +\infty, \]
and applying the Borel-Cantelli lemma we deduce
\[ \limsup_{t \rightarrow +\infty} \frac{f(t) \mathcal{L}^{*, +}_X(t)}{t} \leq a. \]
Now, letting $a$ go to $0$ and combining with the finiteness of $\sup_{]-\infty, 0[} \mathcal{L}_X (+\infty)$ (see \eqref{expretl3} and \eqref{tlneg}) we obtain \eqref{majolimsup}. 

Let us now assume that $\int_1^{+\infty} \frac{(f(t))^{\kappa}}{t} dt = + \infty$ and prove that
\begin{eqnarray}
\limsup_{t \rightarrow +\infty} \frac{f(t) \mathcal{L}^*_X(t)}{t} = + \infty. \label{minolimsup}
\end{eqnarray}

According to Remark \ref{serieint} the condition $\int_1^{+\infty} \frac{(f(t))^{\kappa}}{t} dt = + \infty$ is equivalent to
\begin{eqnarray}
\sum_{n=1}^{+\infty} (f(2^n))^{\kappa} = + \infty. \label{minolimsuphyp}
\end{eqnarray}

Let $M > 0$, $u_n := 2^n/m - 2^{\alpha n}$ and $X^n := X(H(\sqrt{2} u_n) + .)$, the diffusion shifted by the hitting time of $\sqrt{2} u_n$. Note that from the Markov property for $X$ at time $H(\sqrt{2} u_n)$ and the stationarity of the increments of $V$, $X^n - \sqrt{2} u_n$ is equal in law to $X$ under the annealed probability $\mathbb{P}$. We take $n$ so large such that $\sqrt{2} u_n < u_{n+1}$ and define the events 
\begin{align*}
\mathcal{C}_n & := \left \{ \frac{f(2^{n+1}) \mathcal{L}^*_{X^n} \left (\tau \left ( X^n, u_{n+1} \right ) \right )}{2^{(n+1)/\kappa}}  \geq M \right \}, \\
\mathcal{D}_n & := \left \{ \tau \left ( X^n, u_{n+1} \right ) < \tau \left ( X^n, u_{n} \right ) \right \}, \\
\mathcal{E}_n & := \mathcal{C}_n \cap \mathcal{D}_n, \\
\mathcal{F}_n & := \left \{ \mathcal{L}^{*}_X(H(u_n)) \leq \mathcal{L}^{*}_X(2^n) \right \}. 
\end{align*}

From the equality in law between $X^n - \sqrt{2} u_n$ and $X$ under $\mathbb{P}$, \eqref{expretl3}, \eqref{tlneg}, and \eqref{supz}, we have
%\begin{align}
%\mathbb{P} \left ( \mathcal{C}_n \right ) & \geq \mathbb{P} \left ( \mathcal{L}^{*, +}_X(H((2-\sqrt{2})2^{n} - (2^{\alpha} - \sqrt{2})2^{\alpha n})) \geq 2^{(n+1)/\kappa} M / f(2^{n+1}) \right ) \nonumber \\
%& = \mathbb{P} \left ( \mathcal{M}_2 \left ( (2-\sqrt{2})2^{n} - (2^{\alpha} - \sqrt{2})2^{\alpha n} \right ) \geq 2^{(n+1)/\kappa} M / f(2^{n+1}) \right ) \nonumber \\
%& = \mathbb{P} \left ( \sup_{[0, (2-\sqrt{2})2^{n} - (2^{\alpha} - \sqrt{2})2^{\alpha n}]} Z \geq 2^{(n+1)/\kappa} M / f(2^{n+1}) \right ) \nonumber \\
%& \geq 1 - e^{- \big ( ((2-\sqrt{2})2^{n} - (2^{\alpha} - \sqrt{2})2^{\alpha n})/Q - ((2-\sqrt{2})2^{n} - (2^{\alpha} - \sqrt{2})2^{\alpha n})^{7 /8} \big ) \times n \big (\sup \xi > 2^{(n+1)/\kappa} M / f(2^{n+1}) \big )} \nonumber \\
%& - ((2-\sqrt{2})2^{n} - (2^{\alpha} - \sqrt{2})2^{\alpha n})^{-\eta} \label{minocnlimsup}
%\end{align}
\begin{align}
\mathbb{P} \left ( \mathcal{C}_n \right ) & = \mathbb{P} \left ( \mathcal{L}_X(H(u_{n+1} - \sqrt{2} u_n)) \geq 2^{(n+1)/\kappa} M / f(2^{n+1}) \right ) \nonumber \\
& \geq \mathbb{P} \left ( \mathcal{L}^{*, +}_X(H(u_{n+1} - \sqrt{2} u_n)) \geq 2^{(n+1)/\kappa} M / f(2^{n+1}) \right ) \nonumber \\
& = \mathbb{P} \left ( \mathcal{M}_2 \left ( u_{n+1} - \sqrt{2} u_n \right ) \geq 2^{(n+1)/\kappa} M / f(2^{n+1}) \right ) \nonumber \\
& = \mathbb{P} \left ( \sup_{x \in [0, u_{n+1} - \sqrt{2} u_n]} Z(x) \geq 2^{(n+1)/\kappa} M / f(2^{n+1}) \right ) \nonumber \\
& \geq 1 - e^{- \big ( (u_{n+1} - \sqrt{2} u_n)/Q - (u_{n+1} - \sqrt{2} u_n)^{7 /8} \big ) \times n \big (\sup \xi > 2^{(n+1)/\kappa} M / f(2^{n+1}) \big )} \nonumber \\
& - (u_{n+1} - \sqrt{2} u_n)^{-\eta}. \label{minocnlimsup}
\end{align}

According to \eqref{singh} and the definition of $u_n$, the exponent in the right hand side is equivalent to 
\[ (2-\sqrt{2}) 2^{\kappa - 1} \Gamma(\kappa) \kappa^2 K (f(2^{n+1}))^{\kappa} / m M^{\kappa}. \] 
Since $f$ converges to $0$ at infinity, the above is also an equivalent for the term $1-e^{-(...)}$ in the right hand side of \eqref{minocnlimsup}. Then, combining with \eqref{minolimsuphyp} and the fact that the other term is the general term of a covering series we get
\begin{eqnarray}
\sum_{n \geq 1} \mathbb{P} \left ( \mathcal{C}_n \right ) = +\infty. \label{minocnlimsup2}
\end{eqnarray}

Reasoning as in the proof of \eqref{sommedn} we can prove that
\begin{eqnarray}
\sum_{n \geq 1} \mathbb{P} \left ( \overline{\mathcal{D}_n} \right ) < +\infty, \label{sommednsup}
\end{eqnarray}
so
\begin{eqnarray}
\sum_{n \geq 1} \mathbb{P} \left ( \mathcal{E}_n \right ) \geq \sum_{n \geq 1} \mathbb{P} \left ( \mathcal{C}_n \right ) - \sum_{n \geq 1} \mathbb{P} \left ( \overline{\mathcal{D}_n} \right ) = +\infty. \label{sommeensup}
\end{eqnarray}

According to \eqref{tpsatt}, and to the increase of $\mathcal{L}_X^*$, we also prove that
\begin{eqnarray}
\sum_{n \geq 1} \mathbb{P} \left ( \overline{\mathcal{F}_n} \right ) < +\infty. \label{sommefnsup}
\end{eqnarray}

The events $\mathcal{E}_n$ are independent since for each $n$, $\mathcal{E}_n$ belongs to the $\sigma$-field $\sigma ( V(s) - V(u_n), u_n \leq s \leq u_{n+1}, \ X(t), H(\sqrt{2} u_n) \leq t \leq H(\sqrt{2} u_n) + \min (\tau(X^n, u_n), \tau(X^{n+1}, u_{n+1}) )$. Combining this independence with \eqref{sommeensup} and the Borel-Cantelli Lemma we get that $\mathbb{P}$-almost surely, the event $\mathcal{E}_n$ is realized infinitely many often. Combining with \eqref{sommefnsup} and the Borel-Cantelli Lemma we get that $\mathbb{P}$-almost surely, the event $\mathcal{C}_n \cap \mathcal{D}_n \cap \mathcal{F}_{n+1}$ is realized infinitely many often. Then, for $n$ such that this event is realized we have, 
\begin{align*}
f(2^{n+1}) \mathcal{L}^{*}_X(2^{n+1})/2^{(n+1)/\kappa} & \geq f(2^{n+1}) \mathcal{L}^{*}_X(H(u_{n+1}))/2^{(n+1)/\kappa} \\
& \geq f(2^{n+1}) \mathcal{L}^{*}_{X^n} \left (\tau \left ( X^n, u_{n+1} \right ) \right ) / 2^{(n+1)/\kappa} \\
& \geq M, 
\end{align*}
so
\[ \limsup_{t \rightarrow +\infty} \frac{f(t) \mathcal{L}^*_X(t)}{t} \geq M. \]
Now, letting $M$ go to infinity we get \eqref{minolimsup} so Theorem \ref{limsupkappa>1} is proved. 

\section{Some lemmas} \label{toolbox}

In this section, we justify some technical facts and lemmas for $V$, $V^{\uparrow}$ and the diffusion in $V$. Some of them are known or can be easily obtained from results of \cite{caslevyvech}, and we give some details for their justification when it is necessary. Some of these facts are new, like the approximation of the contributions of the valleys to the traveled distance by an \textit{iid} sequence. 

\subsection{Properties of $V$, $V^{\uparrow}$ and $\hat V^{\uparrow}$} \label{toutsurlenv}

\begin{lemme} (Lemma 5.5 of \cite{caslevyvech}) \label{tpsatteinthatv}

There are two positive constants $c_1, c_2$ such that
\[ \forall y, r > 0, \ P \left ( \tau(V, ]-\infty, -y]) > r \right ) \leq e^{c_1 y -c_2 r}. \]

\end{lemme}

\begin{lemme} \label{supdev}

There are positive constants $c_1, c_2$ such that, 
\[ \forall t, a > 0, \ P \left ( \sup_{[t, +\infty[} V > - a \right ) \leq e^{c_1 a - c_2 t} + e^{-\kappa a}. \]

\end{lemme}

\begin{proof}

Let us choose $\gamma \in ]0, \kappa[$, we have
\begin{align*}
P \left ( \sup_{[t, +\infty[} V > - a \right ) & \leq P \left ( V(t) > -2a \right ) + P \left ( V(t) \leq -2a, \sup_{[0, +\infty[} V(t + .) - V(t) > a \right ) \\
& \leq P \left ( e^{\gamma V(t)} > e^{-2\gamma a} \right ) + P \left ( \sup_{[0, +\infty[} V(t + .) - V(t) > a \right ) \\
& = e^{2\gamma a} E \left [ e^{\gamma V(t)} \right ] + e^{-\kappa a}, 
\end{align*}
where we used Markov's inequality for the first term and the Markov property at time $t$ for the second term, together with the fact that the supremum of $V$ on $[0, +\infty[$ follows an exponential distribution with parameter $\kappa$. Since $E [ e^{\gamma V(t)} ] = e^{t \Psi_V(\gamma)}$ and $\Psi_V(\gamma) < 0$ (because $0 < \gamma < \kappa$), we get the result with $c_1 := 2 \gamma$ and $c_2 := -\Psi_V(\gamma)$. 

\end{proof}

\begin{lemme} (Lemma 5.3 of \cite{caslevyvech}) \label{foncexpov}

There is a positive constant $\mathcal{C}$ such that
\[ \mathbb{P} \left ( \int_0^{+\infty} e^{V(u)} du \geq x \right ) \underset{x \rightarrow +\infty}{\sim} \mathcal{C} x^{-\kappa}. \]
%where $\mathcal{C}$ is the constant of Corollary 5 of Bertoin, Yor \cite{Bertoinyor} applied to $-V$. 

\end{lemme}

We now state some Lemmas about $V^{\uparrow}$. First, we recall how $V^{\uparrow}$ and $\hat V^{\uparrow}$ are defined. 

$V$ being spectrally negative, the Markov family $( V^{\uparrow}_x, x \geq 0 )$ may be defined by $h$-transform as in \cite{Bertoin}, Section VII.3. For any $x \geq 0$, the process $V^{\uparrow}_x$ must be seen as $V$ conditioned to stay positive and starting from $x$. We denote $V^{\uparrow}$ for the process $V^{\uparrow}_0$. It is known that $V^{\uparrow}_x$ converges in the Skorokhod space to $V^{\uparrow}$ when $x$ goes to $0$. Also, as well as $V$, $V^{\uparrow}$ has no positive jumps so it reaches every positive level continuously. 

We now define $\hat V^{\uparrow}$, that is, $\hat V$ conditioned to stay positive. Since $V$ is spectrally negative and not the opposite of a subordinator, it is regular for $]0, +\infty[$ (see \cite{Bertoin}, Theorem VII.1), so $\hat{V}$ is for $]-\infty, 0[$. Moreover, $\hat V$ drifts to $+\infty$. We can thus define the Markov family $( \hat V^{\uparrow}_x, \ x \geq 0 )$ as in Doney \cite{Doney}, Chapter 8. It can be seen from there that for any $x > 0$ the process $\hat V^{\uparrow}_x$ is Markovian and has infinite life-time. If moreover $V$ has unbounded variation then $\hat V$ is regular for $]0, +\infty[$, and from Theorem 24 of \cite{Doney}, we have that $\hat V^{\uparrow}_0$, that we denote by $\hat V^{\uparrow}$, is well defined. 
%and that $\hat V^{\uparrow}_0$, that we denote by $\hat V^{\uparrow}$, is indeed well defined, and that $\hat{V}^{\uparrow}_x$ converges in the Skorokhod space to $\hat{V}^{\uparrow}$ when $x$ goes to $0$. 

Here again, for any $x \geq 0$, the process $\hat V^{\uparrow}_x$ must be seen as $\hat V$ conditioned to stay positive and starting from $x$. Note that, since $\hat V$ converges almost surely to infinity, for $x > 0$, $\hat V^{\uparrow}_x$ is only $\hat V_x$ conditioned in the usual way to remain positive. 

\begin{lemme} (Lemma 5.8 of \cite{caslevyvech}) \label{vuprestegrand}

There are two positive constants $c_1, c_2$ such that, for all $1 < a < b$, we have 
\[ \mathbb{P} \left ( \inf_{[0, +\infty[} V_b^{\uparrow} < a \right ) \leq c_2 e^{-c_1 (b-a)}. \]

\end{lemme}

\begin{lemme} \label{dernierpassage}

There are positive constants $c_3, c_4, c_5, c_6$ such that, 
\begin{align}
\forall \ 0 \leq x < y, \ r > 0, \ \mathbb{P} \left ( \tau(V^{\uparrow}_x, y) > r \right ) & \leq e^{c_3 y - c_4 r}, \label{tpsattup} \\
\forall z, r > 0, \ \mathbb{P} \left ( \mathcal{K}(V^{\uparrow}_z, z) > r \right ) & \leq e^{2 c_3 z - c_4 r} + c_6 e^{-c_5 z}. \label{boundlastpassage} 
\end{align}

\end{lemme}

\begin{proof}

Let us fix $0 \leq x < y$. From the Markov property applied at $\tau(V^{\uparrow}, x)$, the hitting time of $x$ by $V^{\uparrow}$, we have
\[ \mathbb{P} \left ( \tau(V^{\uparrow}_x, y) > r \right ) = \mathbb{P} \left ( \tau(V^{\uparrow}(\tau(V^{\uparrow}, x) + .), y) > r \right ) \leq \mathbb{P} \left ( \tau(V^{\uparrow}, y) > r \right ), \]
so \eqref{tpsattup} follows from Lemma 5.7 of \cite{caslevyvech}. For the second point, we have
\[ \left \{ \tau(V^{\uparrow}_z, 2z) \leq r \right \} \cap \left \{ \inf_{[\tau(V^{\uparrow}_z, 2z), +\infty[} V^{\uparrow}_z > z \right \} \subset \left \{ \mathcal{K}(V^{\uparrow}_z, z) \leq r \right \}, \]
so taking the complementary, 
\begin{align*}
\mathbb{P} \left ( \mathcal{K}(V^{\uparrow}_z, z) > r \right ) & \leq \mathbb{P} \left ( \tau(V^{\uparrow}_z, 2z) > r \right ) + \mathbb{P} \left ( \inf_{[\tau(V^{\uparrow}_z, 2z), +\infty[} V^{\uparrow}_z \leq z \right ) \\
& \leq e^{2 c_3 z - c_4 r} + \mathbb{P} \left ( \inf_{[0, +\infty[} V^{\uparrow}_{2z} \leq z \right ), 
\end{align*}
where, for the first term, we used \eqref{tpsattup} with $x = z$, $y = 2 z$ and, for the second term, we used the Markov property at time $\tau(V^{\uparrow}_z, 2z)$. Combining with Lemma \ref{vuprestegrand} applied with $a = z$, $b = 2z$, we get \eqref{boundlastpassage}. 

\end{proof}

\begin{lemme} \label{fonctronque}

There is a positive constant $c$ such that for $t$ large enough, 
\[ \mathbb{P} \left ( \int_{\tau(V^{\uparrow}, h_t/2)}^{+\infty} e^{- V^{\uparrow}(x)} dx \geq e^{-h_t /4} \right ) \leq e^{-c h_t}. \]

\end{lemme}

\begin{proof}

\begin{align*}
\int_{\tau(V^{\uparrow}, h_t/2)}^{\mathcal{K}(V^{\uparrow}, h_t/2)} e^{- V^{\uparrow}(x)} dx & \leq \left ( \mathcal{K}(V^{\uparrow}, h_t/2) - \tau(V^{\uparrow}, h_t/2) \right ) \times \sup_{[\tau(V^{\uparrow}, h_t/2), \mathcal{K}(V^{\uparrow}, h_t/2)]} e^{- V^{\uparrow}} \\
& \overset{\mathcal{L}}{=} \mathcal{K}(V^{\uparrow}_{h_t/2}, h_t/2) \times \sup_{[0, \mathcal{K}(V^{\uparrow}_{h_t/2}, h_t/2)]} e^{- V^{\uparrow}_{h_t/2}}, 
\end{align*}
where we used the Markov property at time $\tau(V^{\uparrow}, h_t/2)$ for the equality in law. We thus get
\begin{align}
\mathbb{P} \left ( \int_{\tau(V^{\uparrow}, h_t/2)}^{\mathcal{K}(V^{\uparrow}, h_t/2)} e^{- V^{\uparrow}(x)} dx \geq e^{-h_t /4}/2 \right ) & \leq \mathbb{P} \left ( \mathcal{K}(V^{\uparrow}_{h_t/2}, h_t/2) \geq e^{h_t /8}/2 \right ) + \mathbb{P} \left ( \inf_{[0, +\infty[} V^{\uparrow}_{h_t/2} \leq 3 h_t /8 \right ) \nonumber \\
& \leq e^{c_3 h_t - c_4 e^{h_t/8}/2} + c_6 e^{-c_5 h_t/2} + c_2 e^{-c_1 h_t/8}, \label{fonctronque1}
\end{align}
where, for the first term, we applied \eqref{boundlastpassage} with $z = h_t/2$, $r = e^{h_t/8} /2$ and, for the second term, we applied Lemma \ref{vuprestegrand} with $a = 3 h_t/8$, $b = h_t/2$. 

Then, according to Corollary VII.19 of \cite{Bertoin} we have 
\[ \int_{\mathcal{K}(V^{\uparrow}, h_t/2)}^{+\infty} e^{- V^{\uparrow}(x)} dx \overset{\mathcal{L}}{=} e^{-h_t/2} \int_{0}^{+\infty} e^{- V^{\uparrow}(x)} dx = e^{-h_t/2} I(V^{\uparrow}), \]
and, according to Theorem 1.1 of \cite{foncexpovech}, $I(V^{\uparrow})$ admits some finite exponential moments, so in particular it has finite expectation. 
%we can choose $\lambda > 0$ such that
%\[ \mathbb{E} \left [ e^{\lambda I(V^{\uparrow})} \right ] < +\infty. \]
We thus get
\begin{eqnarray}
\mathbb{P} \left ( \int_{\mathcal{K}(V^{\uparrow}, h_t/2)}^{+\infty} e^{- V^{\uparrow}(x)} dx \geq e^{-h_t /4} /2 \right ) = \mathbb{P} \left ( I(V^{\uparrow}) \geq e^{h_t /4}/2 \right ) \leq 2 e^{-h_t /4} \mathbb{E} \left [ I(V^{\uparrow}) \right ]. \label{fonctronque2}
\end{eqnarray}
%\begin{eqnarray}
%\mathbb{P} \left ( \int_{\mathcal{K}(V^{\uparrow}, h_t/2)}^{+\infty} e^{- V^{\uparrow}(x)} dx \geq e^{-h_t /4} /2 \right ) = \mathbb{P} \left ( I(V^{\uparrow}) \geq e^{h_t /4}/2 \right ) \leq e^{- \lambda e^{h_t /4}/2} \mathbb{E} \left [ e^{\lambda I(V^{\uparrow})} \right ]. \label{fonctronque2}
%\end{eqnarray}
The result follows from the combination of \eqref{fonctronque1} and \eqref{fonctronque2}. 

\end{proof}

The next fact gives the law of the bottom of the valleys in terms of the laws of $V^{\uparrow}$ and $\hat V^{\uparrow}$. It is a combination of Propositions 3.2 and 3.6 of \cite{caslevyvech}. 

\begin{fact} \label{loidesval}

Assume $V$ has unbounded variation. For all $i \geq 1$ let 
\begin{align*}
P^{(i)} & := (V^{(i)}(m_i-x),\ 0 \leq x \leq m_i-\tau_i^-(h_t)) \\
\tilde P^{(i)} & := (\tilde V^{(i)}(\tilde m_i-x),\ 0 \leq x \leq \tilde m_i- \tilde \tau_i^-(h_t)). 
\end{align*}

For all $i \geq 1$ we have 
\begin{eqnarray}
d_{VT} \left ( \tilde P^{(i)}, P^{(2)} \right ) \leq 2 e^{- \delta \kappa h_t /3} \label{vartot}
\end{eqnarray}
where $d_{VT}$ is the total variation distance. Moreover, the law of $P^{(2)}$ is absolutely continuous with respect to the law of the process $(\hat{V}^{\uparrow}(x), \ 0 \leq x \leq \tau(\hat{V}^{\uparrow}, h_t+))$ and has density $c_{h_t} /(1-e^{-\kappa \hat{V}^{\uparrow}(\tau(\hat{V}^{\uparrow}, h_t+))})$ with respect to this law, where $c_{h_t}$ is a constant increasing with $h_t$ and converging to $1$ when $t$ (and hence $h_t$) goes to infinity. 

For all $i \geq 1$, the two processes $(\tilde V^{(i)}(\tilde m_i-x),\ 0 \leq x \leq \tilde m_i- \tilde \tau_i^-(h_t)) = \tilde P^{(i)}$ and $(\tilde V^{(i)}(\tilde m_i+x),\  0\leq x \leq \tilde \tau_i(h_t)- \tilde m_i)$ are independent and the second is equal in law to $(V^{\uparrow}(x), \ 0 \leq x \leq \tau(V^{\uparrow}, h_t))$. 

\end{fact}

Let us now recall a fact from \cite{caslevyvech}: 

\begin{fact} \label{3integrals}
Assume that the hypotheses of Theorems \ref{limsupkappa<1} and \ref{liminfkappa<1} are satisfied. Fix $\epsilon$ small enough. There is a positive constant $c$ (depending on $\epsilon$) such that for all $t$ large enough 
\begin{align}
\forall j \geq 1, \ & P \left( \int_{\tilde{L}_{j-1}}^{\tilde \tau_j^-(h_t / 2)} e^{-\tilde V^{(j)}(u)} du \leq e^{-\epsilon h_t} \right) \geq 1 - e^{-c h_t}, \label{neggauche} \\
\forall j \geq 1, \ & P \left( \int_{\tilde \tau_j^+(h_t / 2)}^{\tilde L_j} e^{-\tilde V^{(j)}(u)} du \leq e^{-\epsilon h_t} \right) \geq 1 - e^{-c h_t}, \label{negdroite} \\
%\forall j \geq 1, \ & P \left ( \sup_{u \in [\tilde \tau_j^-(h_t /2), \tilde \tau_j^+(h_t /2)]} \left | A^j(u) / A^j(\tilde{L}_j) \right | \leq e^{-h_t/3} \right ) \geq 1 - e^{-c h_t}, \label{mainmilieu} \\
\forall j \geq 1, \ & P \left ( \sup_{u \in [\tilde \tau_j^-(h_t /2), \tilde \tau_j^+(h_t /2)]} \left | A^j(u) \right | \leq e^{5 h_t/8} \right ) \geq 1 - e^{-c h_t}, \label{tildeamilieu}
%\forall j \geq 1, \ & P \left ( S_j^t \leq A^j(\tilde{L}_j) \leq (1 + e^{-h_t / 7}) S_j^t \right ) \geq 1 - e^{-c h_t}. \label{compa-s}
\end{align}
where $A^j$ is defined in Subsection \ref{trapsdiff}. 
\end{fact}

\begin{proof}
\eqref{neggauche} and \eqref{negdroite} 
%and \eqref{mainmilieu} 
are respectively Lemmas 4.6 and 4.7 
%and Lemma 4.8 (applied with $\epsilon = 1/6$) 
from \cite{caslevyvech}. For \eqref{tildeamilieu}, note that
\[ P \left ( \sup_{u \in [\tilde \tau_j^-(h_t /2), \tilde \tau_j^+(h_t /2)]} \left | A^j(u) \right | \leq e^{5 h_t/8} \right ) \leq P \left ( \left | A^j(\tilde \tau_j^-(h_t / 2)) \right | \vee A^j(\tilde \tau_j^+(h_t / 2)) \geq e^{5 h_t /8} \right ). \]
Then, $| A^j(\tilde \tau_j^-(h_t / 2)) |$ and $A^j(\tilde \tau_j^+(h_t / 2))$ can be bounded as in the proof of Lemma 4.8 of \cite{caslevyvech} which yields \eqref{tildeamilieu}. 
%\eqref{compa-s} is only Lemma 4.10 of \cite{caslevyvech}. 

\end{proof}

\subsection{Contribution of the valleys to the traveled distance}

We need to approximate the contribution to the distance traveled by the diffusion in each valley. We have: 

\begin{prop} \label{contribdist}
On an enlarged probability space, there is an \textit{iid} sequence $(D_j^t)_{j \geq 1}$, independent from the sequence $(e_j, S_j^t, R_j^t)_{j \geq 1}$ and such that for $t$ large enough, 
\[ \forall j \geq 1, \ \mathbb{P} \left ( \left | (\tilde L_{j} - \tilde L_{j-1}) - D_j^t \right | \leq e^{-c_1 h_t} D_j^t \right ) \geq 1 - e^{-c_2 h_t} \ \text{and} \ D_j^t \overset{\mathcal{L}}{\sim} \mathcal{E}( e^{-\kappa h_t} q), \]
where $q$ is the constant in Theorem 1.4 of \cite{caslevyvech} and $c_1, c_2$ are positive constants. 
\end{prop}

\begin{proof}

%Let $t$ be fixed so $h_t$ is also fixed, and large enough (as it will be specified later). 
From the definitions of valleys we have
\begin{eqnarray}
\tilde L_{j} - \tilde L_{j-1} = (\tilde L^{\sharp}_{j} - \tilde L_{j-1}) + (\tilde m_{j} - \tilde L^{\sharp}_{j}) + (\tilde \tau_{j}(h_t) -\tilde m_{j}) + (\tilde L_{j} - \tilde \tau_{j}(h_t)). \label{suitevalley}
\end{eqnarray}
Recall the definition of $m^*(h_t)$ from Section 3.5 in \cite{caslevyvech}: 
\[ \tau^*(h_t) := \inf \left \{ u \geq 0, \ V(u) - \inf_{[0, u]}V = h_t \right \}, \ m^*(h_t) := \inf \left \{ u \geq 0, \ V(u) = \inf_{[0, \tau^*(h_t)]}V \right \}. \]
From the Markov property at the stopping times $\tilde L_{j-1}$, $\tilde L^{\sharp}_{j}$ and $\tilde \tau_{j}(h_t)$, and Fact \ref{loidesval} we get that the terms in the right hand side of \eqref{suitevalley} are respectively equal in law to $\tau(V, ]-\infty, -e^{(1-\delta)\kappa h_t}])$, $m^*(h_t)$, $\tau(V^{\uparrow}, h_t)$ and $\tau(V, ]-\infty, -h_t/2])$. For the first and fourth term, applying Lemma \ref{tpsatteinthatv} with $y =e^{(1-\delta)\kappa h_t}$ and $r=e^{(1-\delta/2)\kappa h_t}/2$ we get for $t$ large enough, 
\begin{eqnarray}
\forall j \geq 1, \ \mathbb{P} \left ( (\tilde L^{\sharp}_{j} - \tilde L_{j-1}) + (\tilde L_{j} - \tilde \tau_{j}(h_t)) > e^{(1-\delta/2)\kappa h_t} \right ) \leq e^{-h_t}. \label{suitevalley1}
\end{eqnarray}
For the third term, applying \eqref{tpsattup} with $x =0$, $y =h_t$ and $r = e^{(1-\delta/2)\kappa h_t}$ we get for $t$ large enough, 
\begin{eqnarray}
\forall j \geq 1, \ \mathbb{P} \left ( (\tilde \tau_{j}(h_t) -\tilde m_{j}) > e^{(1-\delta/2)\kappa h_t} \right ) \leq e^{-h_t}. \label{suitevalley2}
\end{eqnarray}
The most important term in \eqref{suitevalley} is the second one, its law is given by Lemma 3.9 of \cite{caslevyvech}. Before studying it, let us recall some of the notations used in Section 3.5 of \cite{caslevyvech}. 

Let $\mathcal{F}$ denote the space of excursions, that is,  c\`ad-l\`ag functions from $[0, +\infty[$ to $\mathbb{R}$, starting at zero and killed at the first positive instant when they reach $0$ (this instant can possibly be infinite). For $\xi \in \mathcal{F}$, let us denote $\zeta(\xi) := \inf \{ s > 0, \ \xi(s) = 0 \}$ for the length of the excursion $\xi$. For $h > 0$, let $\mathcal{F}_{h, -}$ and $\mathcal{F}_{h, +}$ denote respectively the set of excursions whose height is strictly less than $h$ and the set of excursions higher than $h$: 
\[ \mathcal{F}_{h, -} := \left \{ \xi \in \mathcal{F}, \ \sup_{[0, \zeta]} \xi < h \right \}, \ \ \ \ \ \mathcal{F}_{h, +} := \left \{ \xi \in \mathcal{F}, \ \sup_{[0, \zeta]} \xi \geq h \right \}. \]

%$\{0\}$ is instantaneous for $V - \underline{V}$ (see Theorem ... in \cite{Bertoin}), excursion theory above $0$ can thus be applied for $V - \underline{V}$ (see \cite{Bertoin}). Let $L$ be a local time at $0$ of $V - \underline{V}$, $\mathcal{N}$ its associated excursion measure and $L^{-1}$ its right continuous inverse. Then, the excursions above $0$ of $V - \underline{V}$ form a Poisson point process on $\mathcal{F}$ with intensity measure $\mathcal{N}$. In the irregular case (when $V$ has bounded variation) the local time $L$ has to be defined artificially as in \cite{Bertoin}, Section IV.5. In this case, the excursion measure is proportional to the law of the first excursion and in particular the total mass of the excursion measure is finite. 

Let $\mathcal{N}$ be the measure defined on $\mathcal{F}$ as in Subsection 3.5 of \cite{caslevyvech}. Let $S^{h, -}$ and $S^{h, +}$ be two independent pure jumps subordinators with L\'evy measures respectively $\zeta \mathcal{N}(\mathcal{F}_{h, -} \cap .)$ and $\zeta \mathcal{N}(\mathcal{F}_{h, +} \cap .)$, the image measures of respectively $\mathcal{N}(\mathcal{F}_{h, -} \cap .)$ and $\mathcal{N}(\mathcal{F}_{h, +} \cap .)$ by $\zeta$. The sum of these measures equals $\zeta \mathcal{N}$ so $S := S^{h, -} + S^{h, +}$ is a pure jumps subordinator with L\'evy measure $\zeta \mathcal{N}$. Let also $T_{h}$ be an exponential random variable with parameter $\mathcal{N}(\mathcal{F}_{h, +})$, independent from $S^{h, -}$ and $S^{h, +}$. According to Lemma 3.9 of \cite{caslevyvech} we have
\[ \forall j \geq 1, \ \tilde m_{j} - \tilde L^{\sharp}_{j} \overset{\mathcal{L}}{=} S^{h_t, -}(T_{h_t}). \]
Recall that the sequence $(\tilde m_{j} - \tilde L^{\sharp}_{j})_{j \geq 1}$ is \textit{iid} because of Remark \ref{iid}. Therefore, on an enlarged probability space, there is an \textit{iid} sequence $(S^{h_t, -}_j, T_{h_t, j})_{j \geq 1}$ such that for all $j \geq 1$, 
\begin{eqnarray}
S^{h_t, -}_j \perp\!\!\!\perp T_{h_t, j}, \ \ \ S^{h_t, -}_j \overset{\mathcal{L}}{=}  S^{h_t, -}, \ \ \ T_{h_t, j} \overset{\mathcal{L}}{=}  T_{h_t} \ \ \ \text{and} \ \ \ S^{h_t, -}_j (T_{h_t, j}) = \tilde m_{j} - \tilde L^{\sharp}_{j}. \label{intervexpo0}
\end{eqnarray}
Let $d := \mathcal{N}(\mathcal{F}_{h_t, +}) T_{h_t}$ and more generally $d_j := \mathcal{N}(\mathcal{F}_{h_t, +}) T_{h_t, j}$. Then, $d_j \overset{\mathcal{L}}{\sim} \mathcal{E}(1)$ and the sequence $(d_j)_{j \geq 1}$ is \textit{iid}. For large $t$, it is natural to approximate $S^{h_t, -}(T_{h_t})$ by a multiple of $d$, for this we write
\begin{align}
S^{h_t, -}(T_{h_t}) - e^{\kappa h_t} d /q = \left ( S^{h_t, -}(T_{h_t}) - S(T_{h_t}) \right ) + \left ( S(T_{h_t}) - \mathbb{E}[S(1)] T_{h_t} \right ) + \left ( \mathbb{E}[S(1)] T_{h_t} - e^{\kappa h_t} d /q \right ). \label{intervexpo1}
\end{align}
For the first term, using the expression of $S$ in terms of $S^{h_t, +}$ and $S^{h_t, -}$, the independence between $S^{h_t, +}$ and $T_{h_t}$, the definition of $S^{h_t, +}$, Cauchy-Schwarz's inequality and the definition of $S$, we have
\begin{align}
\mathbb{E} \left [ \left | S^{h_t, -}(T_{h_t}) - S(T_{h_t}) \right | \right ] & = \mathbb{E} \left [ S^{h_t, +}(T_{h_t}) \right ] = \frac1{\mathcal{N}(\mathcal{F}_{h_t, +})} \mathbb{E} \left [ S^{h_t, +}(1) \right ] \nonumber \\
&= \frac1{\mathcal{N}(\mathcal{F}_{h_t, +})} \int_{\mathcal{F}_{h_t, +}} \zeta(\xi) \mathcal{N}(d \xi) \leq \frac1{\sqrt{\mathcal{N}(\mathcal{F}_{h_t, +})}} \sqrt{ \int \zeta^2(\xi) \mathcal{N}(d \xi)}, \nonumber \\
& = \frac1{\sqrt{\mathcal{N}(\mathcal{F}_{h_t, +})}} \sqrt{ Var(S(1))}. \label{intervexpo2}
\end{align}

$Var(S(1))$ is indeed finite thanks to Lemma 3.10 of \cite{caslevyvech}. For the second term in the right hand side of \eqref{intervexpo1}, using Cauchy-Schwarz's inequality and the independence between $S$ and $T_{h_t}$, 
\begin{align}
\mathbb{E} \left [ \left | S(T_{h_t}) - \mathbb{E}[S(1)] T_{h_t} \right | \right ] \leq \sqrt{\mathbb{E} \left [ \left ( S(T_{h_t}) - \mathbb{E}[S(1)] T_{h_t} \right )^2 \right ]} = \frac1{\sqrt{\mathcal{N}(\mathcal{F}_{h_t, +})}} \sqrt{ Var(S(1))}. \label{intervexpo3}
\end{align}
For the third term in the right hand side of \eqref{intervexpo1}, using the definition of $d$, 
\[ \mathbb{E}[S(1)] T_{h_t} - e^{\kappa h_t} d /q = (\mathbb{E}[S(1)] / \mathcal{N}(\mathcal{F}_{h_t, +}) - e^{\kappa h_t} /q) d. \]
Then, recall the expression of $q$ in Proposition 3.13 of \cite{caslevyvech} (see also the proof of that proposition for the expression of $q$ involving $\mathbb{E}[S(1)]$) and Lemma 3.11 of \cite{caslevyvech}: 
\[ q := \mathcal{N}(\mathcal{F}_{1, +}) \times (e^{\kappa} - \mathbb{E} [ e^{\kappa V_1(\tau(V_1, ]-\infty, 0]))} ] ) / \mathbb{E}[S(1)], \]
and 
\begin{eqnarray}
\mathcal{N}(\mathcal{F}_{h_t, +}) = e^{-\kappa h_t} \mathcal{N}(\mathcal{F}_{1, +}) \times \left ( e^{\kappa} - \mathbb{E} \left [ e^{\kappa V_1(\tau(V_1, ]-\infty, 0]))} \right ] \right ) + \underset{t \rightarrow +\infty}{\mathcal{O}} (e^{-2\kappa h_t}). \label{intervexpo3.1}
\end{eqnarray}
We thus get
\begin{align}
\text{a.s.} \ \mathbb{E}[S(1)] T_{h_t} - e^{\kappa h_t} d /q = d \times \underset{t \rightarrow +\infty}{\mathcal{O}} (e^{-\kappa h_t}). \label{intervexpo4}
\end{align}
Let us define $D^t := e^{\kappa h_t} d /q$ and more generally $D_j^t := e^{\kappa h_t} d_j /q$. We then have indeed $D_j^t \overset{\mathcal{L}}{\sim} \mathcal{E}( e^{-\kappa h_t} q)$ and the fact that the sequence sequence $(D_j^t)_{j \geq 1}$ is \textit{iid}. \eqref{intervexpo4} becomes 
\begin{align}
\text{a.s.} \ \left | \mathbb{E}[S(1)] T_{h_t} - D^t \right | \leq D^t \times C e^{-2\kappa h_t}, \label{intervexpo5}
\end{align}
where $C$ is a positive constant. Using Markov's inequality, \eqref{intervexpo2} and \eqref{intervexpo3}, and the asymptotic of $\mathcal{N}(\mathcal{F}_{h_t, +})$ given by \eqref{intervexpo3.1}, we get the existence of a positive constant $C$ such that for $t$ large enough, 
\[ \mathbb{P} \left ( \left | S^{h_t, -}(T_{h_t}) - S(T_{h_t}) \right | + \left | S(T_{h_t}) - \mathbb{E}[S(1)] T_{h_t} \right | \geq e^{6 \kappa h_t /10} \right ) \leq C e^{-\kappa h_t /10}. \]

\[ \mathbb{P} \left ( D^t \leq e^{9 \kappa h_t /10} \right ) = \mathbb{P} \left ( d \leq q e^{- \kappa h_t /10} \right ) \leq q e^{- \kappa h_t /10}. \]

As a consequence there is a positive constant $C$ such that for $h_t$ large enough, 
\[ \mathbb{P} \left ( \left | S^{h_t, -}(T_{h_t}) - S(T_{h_t}) \right | + \left | S(T_{h_t}) - \mathbb{E}[S(1)] T_{h_t} \right | \geq e^{-3 \kappa h_t /10} D^t \right ) \leq C e^{-\kappa h_t /10}. \]
Combining with \eqref{intervexpo5} and putting into \eqref{intervexpo1} we get 
\[ \mathbb{P} \left ( \left | S^{h_t, -}(T_{h_t}) - D^t \right | \geq e^{-2 \kappa h_t /10} D^t \right ) \leq C e^{-\kappa h_t /10}. \]
Combining with \eqref{intervexpo0} we get
\begin{eqnarray}
\forall j \geq 1, \ \mathbb{P} \left ( \left | \tilde m_{j} - \tilde L^{\sharp}_{j} - D^t_j \right | \geq e^{-2 \kappa h_t /10} D^t_j \right ) \leq C e^{-\kappa h_t /10}. \label{intervexpo6}
\end{eqnarray}
Then, we have $\mathbb{P} ( D^t_j \leq e^{(1-\delta/2)\kappa h_t} ) = \mathbb{P} ( d_j \leq q e^{-\delta \kappa h_t /2} ) \leq q e^{-\delta \kappa h_t /2}$ so putting into \eqref{suitevalley1} and \eqref{suitevalley2} we get for all $j \geq 1$
\begin{eqnarray}
\mathbb{P} \left ( (\tilde L^{\sharp}_{j} - \tilde L_{j-1}) + (\tilde \tau_{j}(h_t) -\tilde m_{j}) + (\tilde L_{j} - \tilde \tau_{j}(h_t)) > 2 e^{-\delta \kappa h_t /2} D^t_j \right ) \leq 2 e^{-h_t} + q e^{-\delta \kappa h_t /2}. \label{intervexpo7}
\end{eqnarray}
Combining \eqref{intervexpo6} and \eqref{intervexpo7}, and putting into \eqref{suitevalley} we get the result for $t$ large enough. 

\end{proof}

\subsection{Proof of some facts and lemmas} \label{preuvedesfacts}

This subsection is devoted to the justification of the facts stated in Subsection \ref{trapsdiff}, which mainly come from \cite{caslevyvech}. As these results are included in Section \ref{smallkappa}, we prove them under the hypotheses of Theorems \ref{limsupkappa<1} and \ref{liminfkappa<1}: $0 < \kappa < 1$, $V$ has unbounded variation and there exists $p>1$ such that $V (1) \in L^p$. 

In the facts and lemmas considered here, the value of the constant $c$ is not important (as long as it is positive) so we allow it to decrease from line to line and it is implicit that the estimates are true for $t$ large enough. 

\begin{proof} of Lemma \ref{minmajnbvalleesvisit}

\eqref{majonbvalleesvisit} is only Lemma 4.17 of \cite{caslevyvech}. For \eqref{minonbvalleesvisit}, note that $ \{ N_t \leq \tilde n_t \} = \{ H(m_{\tilde n_t + 1}) > t \}$ and $H(\tilde m_{\tilde n_t + 1}) = H(m_{\tilde n_t + 1})$ on $\{ V \in \mathcal{V}_t \}$ (since $\tilde n_t < n_t$). Let $t$ be large enough so that $(1- 2 / \log(h_t)) (1+ e^{- \tilde c h_t})^{-1} \geq 1/2$. Using the definitions of $\mathcal{E}^5_t$ and $\mathcal{E}^7_t$ (which is possible here since $\tilde n_t < n_t$) we get that $\mathbb{P} ( \{ N_t \leq \tilde n_t \} \cap \{ V \in \mathcal{V}_t \} \cap \mathcal{E}^5_t \cap \mathcal{E}^7_t )$ is less than
\begin{align*}
& \mathbb{P} \left( H(\tilde m_{\tilde n_t + 1}) >  t, \ \mathcal{E}^5_t, \ \mathcal{E}^7_t \right) \leq \mathbb{P} \left( \sum_{i=1}^{\tilde n_t} e_i S_i^t R_i^t \geq t /2 \right) \\
\leq & \mathbb{P} \left( \sup_{1 \leq i \leq \tilde n_t} e_i S_i^t R_i^t \geq t /2 \tilde n_t \right) = 1 - \left [ \mathbb{P} \left( e_1 S_1^t R_1^t /t < 1/2 \tilde n_t \right) \right ]^{\tilde n_t} \\
= & 1 - \left [ 1 - \mathbb{P} \left( e_1 S_1^t R_1^t /t \geq 1/ 2 \tilde n_t \right) \right ]^{\tilde n_t}. 
\end{align*}
According to \eqref{cvmeasure9.1} applied with some $\eta \in ]0, (1 - \rho)/3[$ we have that 
\[ \mathbb{P} \left ( e_1 S_1^t R_1^t /t \geq 1/2 \tilde n_t \right) \underset{t \rightarrow +\infty}{\sim} 2^{\kappa} \mathcal{C}' \mathbb{E}\left [ \mathcal{R}^{\kappa} \right ] \tilde n_t^{\kappa} e^{- \kappa \phi(t)} \underset{t \rightarrow +\infty}{\sim} 2^{\kappa} \mathcal{C}' \mathbb{E}\left [ \mathcal{R}^{\kappa} \right ] e^{\kappa (\rho - 1) \phi(t)}, \]
where $\mathcal{C}' $ is the constant in Fact \ref{queueiid}. Since $\rho < 1$ the later converges to $0$ so 
\begin{align*}
\left [ 1 - \mathbb{P} \left( e_1 S_1^t R_1^t /t \geq 1/2 \tilde n_t \right) \right ]^{\tilde n_t} & \underset{t \rightarrow +\infty}{\approx} \exp \left ( -2^{\kappa} \mathcal{C}' \mathbb{E}\left [ \mathcal{R}^{\kappa} \right ] \tilde n_t e^{\kappa (\rho - 1) \phi(t)} \right ) \\
& \underset{t \rightarrow +\infty}{\approx} \exp \left ( - 2^{\kappa} \mathcal{C}' \mathbb{E}\left [ \mathcal{R}^{\kappa} \right ] e^{(\rho + \kappa (\rho - 1)) \phi(t)} \right ). 
\end{align*}
Let $\rho' := - (\rho + \kappa (\rho - 1))$. $\rho'$ is positive thanks to the hypothesis $\rho \in ]0, \kappa/(1+\kappa)[$ so $e^{(\rho + \kappa (\rho - 1)) \phi(t)}$ converges to $0$ and we deduce that for $t$ large enough, 
\[ \mathbb{P} ( \{ N_t \leq \tilde n_t \} \cap \{ V \in \mathcal{V}_t \} \cap \mathcal{E}^5_t \cap \mathcal{E}^7_t ) \leq 2 \mathcal{C}' \mathbb{E}\left [ \mathcal{R}^{\kappa} \right ] e^{-\rho' \phi(t)}. \]
Combining with the bounds for $\mathbb{P} ( V \notin \mathcal{V}_t )$ and $\mathbb{P} ( \overline{\mathcal{E}^5_t} ) + \mathbb{P} ( \overline{\mathcal{E}^7_t} )$ given by respectively Fact \ref{minimacoincide} and \eqref{bigeventsleq} we get \eqref{minonbvalleesvisit}. 
\end{proof}

Now, let us recall a fact about Brownian local time: 

\begin{fact} \label{tlmb}
There is a positive constant $c$ such that for all $t$ large enough 
\begin{align}
%\mathbb{P} \left ( \sup_{|y| \leq e^{-h_t/3}} \left | \mathcal{L}_{B} ( \tau(B, 1), y ) - \mathcal{L}_{B} ( \tau(B, 1), 0 ) \right | > e^{-h_t/9} \mathcal{L}_{B} ( \tau(B, 1), 0 ) \right ) & \leq e^{-c h_t}, \label{rayknight3} \\
\mathbb{P} \left ( \sup_{|y| \leq e^{-h_t/8}} \left | \mathcal{L}_B(\sigma_{B}(1,0), y) - 1 \right | \geq e^{-h_t /20} \right ) & \leq e^{-h_t}, \label{rayknight1} \\
\mathbb{P} \left ( \sup_{y \in \mathbb{R}} \mathcal{L}_B(\sigma_{B}(1,0), y) \geq u \right ) & \leq 2/u, \label{rayknight2}. 
\end{align}
\end{fact}

\begin{proof}

%\eqref{rayknight3} is $(7.11)$ of \cite{advech} applied with $\delta = e^{-h_t/3}$, $\epsilon = e^{-h_t/9}$. 
\eqref{rayknight1} comes from the second Ray-Knight Theorem combined with estimate $(7.15)$ of \cite{advech} applied with $u = e^{-h_t /20}$, $v = e^{-h_t/8}$. \eqref{rayknight2} comes from the second Ray-Knight Theorem combined with estimate $(7.16)$ of \cite{advech}. 

\end{proof}

\begin{proof} of Fact \ref{bigevents}

According to the combination of Lemmas 5.21 and 5.22 of \cite{caslevyvech}, there is a positive constant $c$ such that $\mathbb{P} ( \overline{\mathcal{E}^1_t} ) \leq e^{-c h_t}$ for large $t$. Since $h_t$ is ultimately greater than $\phi(t)$ we get $\mathbb{P} ( \overline{\mathcal{E}^1_t} ) \leq e^{-c \phi(t)}$ for large $t$. 

$\mathbb{P} ( \overline{\mathcal{E}^2_t} ) \leq e^{-c \phi(t)}$ and $\mathbb{P} ( \overline{\mathcal{E}^3_t} ) \leq e^{-c \phi(t)}$, for $c$ possibly decreased and large $t$, come from Fact 4.4 of \cite{caslevyvech} (respectively the second and third point). 

For $\mathbb{P} ( \overline{\mathcal{E}^4_t} )$, note that $\mathcal{L}_X(H(\tilde L_j), .) - \mathcal{L}_X(H(\tilde m_j), .) = \mathcal{L}_{X_{\tilde m_j}}(H_{X_{\tilde m_j}}(\tilde L_j),.)$. Since $\tilde c > 0$ has been chosen "small enough" in the Introduction, we can assume $\tilde c \leq 1/9$ so $\mathbb{P} ( \overline{\mathcal{E}^4_t} ) \leq e^{-c h_t} \leq e^{-c \phi(t)}$, for $c$ possibly decreased and large $t$, follows from Fact 4.12 of \cite{caslevyvech} (and the fact that $h_t$ is ultimately greater than $\phi(t)$).

$\mathbb{P} ( \overline{\mathcal{E}^5_t} ) \leq e^{-c \phi(t)}$, for $c$ possibly decreased and large $t$, comes from Fact 4.3 of \cite{caslevyvech}. 

$\tilde c > 0$ has been fixed "small enough" in the Introduction. We can assume that it was chosen so small such that Proposition 4.5 of \cite{caslevyvech} apply with $\tilde c$ instead the constant $\epsilon/7$ there. $\mathbb{P} ( \overline{\mathcal{E}^6_t} ) \leq e^{-c h_t} \leq e^{-c \phi(t)}$ and $\mathbb{P} ( \overline{\mathcal{E}^7_t} ) \leq e^{-c h_t} \leq e^{-c \phi(t)}$, for $c$ possibly decreased and large $t$, thus come from Proposition 4.5 of \cite{caslevyvech} (and the fact that $h_t$ is ultimately greater than $\phi(t)$). 

The second point of the fact is Lemma 4.11 of \cite{caslevyvech}.

\end{proof}

\begin{proof} of Lemma \ref{measofgoodenv}

According to \eqref{cvrlapl} and Chernoff inequality we have for any $j \geq 1$
\begin{eqnarray}
P \left ( R_j^t \geq e^{h_t /8} \right ) \leq e^{- \lambda_0 e^{h_t /8} /2}. \label{encr1}
\end{eqnarray}
We have
\begin{align}
P \left ( R_j^t \leq e^{- \epsilon h_t /4} \right ) \leq P \left ( R_j^t \leq e^{- h_t^{1/3}} \right ) & \leq 2 P \left ( I(V^{\uparrow}) \leq 2 e^{- h_t^{1/3}} \right ) + 2 e^{- \delta \kappa h_t /3} \nonumber \\ 
& \leq 2 e^{- K_0 e^{h_t^{1/3}}/2} + 2 e^{- \delta \kappa h_t /3}, \label{encr2}
\end{align}
where, for the last two inequalities, we used Lemma \ref{approxder2cas} (applied with $z_t = e^{h_t^{1/3}}$, $a=2$) and \eqref{majouniv}, and where $K_0$ is the constant in \eqref{majouniv}. The term $2 e^{- \delta \kappa h_t /3}$ is only necessary in the case $V = W_{\kappa}$ (because of \eqref{approxder2casmdb}). Combining \eqref{encr1} and \eqref{encr2} we get that \eqref{invtl1.1.0} is satisfied with probability at least $1 - e^{- c h_t}$. 

According to Fact \ref{3integrals} we have that \eqref{invtl1.1.00}, \eqref{invtl1.1.001}, \eqref{invtl1.1.1} and \eqref{invtl1.1.12} are satisfied with probability at least $1 - e^{- c h_t}$. 

According to \eqref{bigeventsleq} we have $\mathbb{P} ( \cup_{i=1}^7 \overline{\mathcal{E}^i_t}) \leq e^{-L \phi(t)}$. Then, 
\[ e^{-L \phi(t)/2} P \left ( P^V ( \cup_{i=1}^7 \overline{\mathcal{E}^i_t}) > e^{-L \phi(t)/2} \right ) \leq E \left [ P^V ( \cup_{i=1}^7 \overline{\mathcal{E}^i_t}) \right ] = \mathbb{P} \left ( \cup_{i=1}^7 \overline{\mathcal{E}^i_t} \right ) \leq e^{-L \phi(t)}, \]
%\begin{align*}
%e^{-c \phi(t)/2} \mathbb{P} \left ( V \in \overline{\mathcal{G}_t} \right ) & \leq e^{-c \phi(t)/2} \mathbb{P} \left ( V \notin \mathcal{V}_t \right ) +  e^{-c \phi(t)/2} \mathbb{P} \left ( P^V ( \cup_{i=1}^7 \overline{\mathcal{E}^i_t}) > e^{-c \phi(t)/2} \right ) \\
%& \leq e^{-c \phi(t)/2} \mathbb{P} \left ( V \notin \mathcal{V}_t \right ) + \mathbb{E} \left [ P^V ( \cup_{i=1}^7 \overline{\mathcal{E}^i_t}) \right ] \\
%& = e^{-c \phi(t)/2} \mathbb{P} \left ( V \notin \mathcal{V}_t \right ) + \mathbb{P} \left ( \cup_{i=1}^7 \overline{\mathcal{E}^i_t} \right ) \leq e^{-c \phi(t)}, 
%\end{align*}
%where we used Fact \ref{minimacoincide} and the upper bound for $\mathbb{P} ( \cup_{i=1}^7 \overline{\mathcal{E}^i_t})$. 
We thus deduce that \eqref{bonenv1} is satisfied with probability at least $1 - e^{-L \phi(t) /2}$. 
%\begin{eqnarray}
%\mathbb{P} ( V \in \overline{\mathcal{G}_t} ) \leq e^{-c \phi(t) /2}. \label{bonenv2}
%\end{eqnarray}
Combing all this with Fact \ref{minimacoincide}, we get Lemma \ref{measofgoodenv} for a suitably chosen constant $c$. 
\end{proof}

%\begin{fact} \label{invtlinf}
%Let $y_t$ be any positive function of $t$. Then
%\[ \left \{ \sigma(t y_t,\tilde m_k) > H(\tilde L_k) \right \} \subset \left \{ t y_t R_k > (1-...) H(\tilde L_k) \right \} \cup \overline{\mathcal{E}^7_t} \]
%\end{fact}
%
%\begin{proof}
%
%We now study the case where $\sigma(t y_t,\tilde m_k) > H(\tilde L_k)$, following the ideas of the proof of Lemma 5.3 of \cite{advech}. Recall that $B$ is the Brownian motion driving the diffusion starting at $\tilde m_k$. We have
%\[ \sigma(t y_t,\tilde m_k) > H(\tilde L_k) \Leftrightarrow \sigma_B(t y_t, 0) > \tau_B(A^k(\tilde L_k)) \Leftrightarrow t y_t > \mathcal{L}_B \left ( \tau_B(A^k(\tilde L_k)), 0 \right ) \]
%Now, let $\tilde B := B((A^k(\tilde L_k))^2 .)/A^k(\tilde L_k)$. From the scaling property $\tilde B$ is still a Brownian motion. $\mathcal{L}_B \left ( \tau^B(A^k(\tilde L_k)), 0 \right ) = A^k(\tilde L_k) \mathcal{L}_{\tilde B} \left ( \tau^{\tilde B}(k), 0 \right )$. Note that from the definition of $e_k$ given in ... we have $\mathcal{L}_{\tilde B} \left ( \tau{\tilde B}(1), 0 \right ) = e_k$. As a consequence, 
% \[ \sigma(t y_t,\tilde m_k) > H(\tilde L_k) \Leftrightarrow t y_t > A^k(\tilde L_k) e_k \Leftrightarrow t R_k y_t > A^k(\tilde L_k) e_k R_k. \]
%According to (MINO $A^k(\tilde L_k)$ PAR $S_j^t$) and the definition of $\mathcal{E}^7_t$ we deduce that 
%\[ \left \{ \sigma(t y_t,\tilde m_k) > H(\tilde L_k) \right \} \subset \left \{ t R_k y_t > (1-...) H(\tilde L_k) \right \} \cup \overline{\mathcal{E}^7_t} \]
%
%\end{proof}

\begin{proof} of Fact \ref{fameuxe8}

The ideas are very similar to the ones used for the proof of Lemma 5.3 of \cite{advech}. Since we are in the more general context of a L\'evy potential and since we do not prove exactly the same thing (here, we work conditionally to the environment), we give the details. 

Recall that $\epsilon > 0$ has been fixed in the definition of $\mathcal{G}_t$. $\tilde c > 0$ has been fixed "small enough" in the Introduction (and we already fixed some constraints about how small it must be in the proof of Fact \ref{bigevents}). We can assume further that it was chosen so small such that $\tilde c < \min (1/20, \epsilon /4)$. Let us fix an environment $v \in \mathcal{G}_t$, $z \in [0, 1-4/\log(h_t)]$ and $k \leq n_t$. We put $\tilde x_t := (1+e^{- \tilde c h_t})^{-1} x_t$, $Z := (1-z) / (1-e^{- \tilde c h_t}) R_k^t$. To study $\mathcal{E}^8_t(v, k, z)$, we look at its intersections with the events $\{ \sigma_{X_{\tilde m_k}} ( t \tilde x_t, \tilde m_k) < H_{X_{\tilde m_k}} (\tilde L_k) \}$ and $\{ \sigma_{X_{\tilde m_k}} ( t \tilde x_t, \tilde m_k) > H_{X_{\tilde m_k}}(\tilde L_k) \}$. We have
\begin{align}
& \mathcal{E}^8_t(v, k, z) \cap \{ \sigma_{X_{\tilde m_k}} ( t \tilde x_t, \tilde m_k) < H_{X_{\tilde m_k}} (\tilde L_k) \} \nonumber \\
= & \left \{ t Z \leq t \tilde x_t, \ \sup_{\mathcal{D}_k} \mathcal{L}_{X_{\tilde m_k}}(t(1-z), .) \geq t x_t, \ \sigma_{X_{\tilde m_k}} ( t \tilde x_t, \tilde m_k) < H_{X_{\tilde m_k}} (\tilde L_k) < H_{X_{\tilde m_k}}(\tilde L_{k-1}) \right \}, \label{invtl}
\end{align}
where $\mathcal{D}_k$ is defined in \eqref{defdj}. 
On the above event, $\sigma_{X_{\tilde m_k}} ( t Z, \tilde m_k)$ is finite and the diffusion stays in $[\tilde L_{k-1}, \tilde L_k]$ until this time. $\sigma_{X_{\tilde m_k}} ( t Z, \tilde m_k)$ is thus equal to
\begin{eqnarray}
I := \int_{\tilde L_{k-1}}^{\tilde L_k}e^{-\tilde v^{(k)}(x)} \mathcal{L}_B(\sigma_{B}(t Z,0), A^k(x))dx, \label{invtl1.1}
\end{eqnarray}
where $B$ is the Brownian motion driving the diffusion $X_{\tilde m_k}$ and $A^k$ is defined in Subsection \ref{trapsdiff}. Let $\tilde B := B((t Z)^2 .)/t Z$ and $\tilde A^k(.) := A^k(.) / t Z$. From the scaling property $\tilde B$ is still a Brownian motion. We can write
\begin{eqnarray}
I = t Z \int_{\tilde L_{k-1}}^{\tilde L_k}e^{-\tilde v^{(k)}(x)} \mathcal{L}_{\tilde B}(\sigma_{\tilde B}(1,0), \tilde A^k(x))dx. \label{formulei}
\end{eqnarray}
Since $v \in \mathcal{G}_t$ it satisfies \eqref{invtl1.1.0}. 
%\begin{align}
%e^{- \epsilon h_t /4} \leq R_k^t = \int_{\tilde \tau_j^-(h_t / 2)}^{\tilde \tau_j^+(h_t / 2)} e^{-(v(u) - v(\tilde m_j))} du \leq e^{h_t /8}, \label{invtl1.1.0} \\
%\left | A^k(\tilde \tau_j^-(h_t / 2)) \right | \vee A^k(\tilde \tau_j^+(h_t / 2)) \leq e^{5 h_t /8}. \label{invtl1.1.00}
%\end{align}
Recall also that $t \geq e^{h_t}$. For $t$ large enough so that $4 e^{-h_t /8}/ (1-e^{- \tilde c h_t}) \log(h_t) \geq e^{-h_t /4}$ we thus have $t Z \geq e^{3 h_t /4}$ for any choice of $z \in [0, 1-4/\log(h_t)]$. Combining this with the fact that $v$ satisfies \eqref{invtl1.1.00} and \eqref{invtl1.1.001} we get
%\[ \mathbb{P} \left ( \sup_{[\tilde \tau_j^-(h_t / 2), \tilde \tau_j^+(h_t / 2)]} \left | \tilde A^k(.) \right | \geq e^{-h_t /8} \right ) \leq \mathbb{P} \left ( \sup_{[\tilde \tau_j^-(h_t / 2), \tilde \tau_j^+(h_t / 2)]} \left | A^k(.) \right | \geq e^{5 h_t /8} \right ) \leq e^{-c h_t}, \]
%where we used \eqref{tildeamilieu} for the last inequality. 
\[ \sup_{[\tilde \tau_k^-(h_t / 2), \tilde \tau_k^+(h_t / 2)]} \left | \tilde A^k(.) \right | \leq \left ( \left | A^k(\tilde \tau_k^-(h_t / 2)) \right | \vee A^k(\tilde \tau_k^+(h_t / 2)) \right ) / tZ \leq e^{- h_t /8}. \]
Combining with \eqref{rayknight1} we get
%\begin{eqnarray}
%\mathbb{P} \left ( \sup_{x \in [\tilde \tau_k^-(h_t /2), \tilde \tau_k^+(h_t /2)]} \left | \mathcal{L}_{\tilde B}(\sigma_{\tilde B}(1,0), \tilde A^k(x)) - 1 \right | \geq ... \right ) \leq e^{-c h_t}, \label{invtl1.1.1.1}
%\end{eqnarray}
\begin{eqnarray}
P^v \left ( \sup_{x \in [\tilde \tau_k^-(h_t /2), \tilde \tau_k^+(h_t /2)]} \left | \mathcal{L}_{\tilde B}(\sigma_{\tilde B}(1,0), \tilde A^k(x)) - 1 \right | \geq e^{-h_t /20} \right ) \leq e^{-h_t}, \label{invtl1.1.1.1pre}
\end{eqnarray}
and we deduce that
\begin{eqnarray}
P^v \left ( \left | \int_{\tilde \tau_k^-(h_t /2)}^{\tilde \tau_k^+(h_t /2)} e^{-\tilde v^{(k)}(x)} \mathcal{L}_{\tilde B}(\sigma_{\tilde B}(1,0), \tilde A^k(x))dx - R_k^t \right | \leq e^{-h_t /20} R_k^t \right ) \geq 1 - e^{-h_t}. \label{invtl1.1.1.1}
\end{eqnarray}
Since $v \in \mathcal{G}_t$ it satisfies \eqref{invtl1.1.1} and \eqref{invtl1.1.12}. 
%the events in \eqref{neggauche}, \eqref{negdroite}: 
%\begin{eqnarray}
%\int_{\tilde{L}_{k-1}}^{\tilde \tau_k^-(h_t / 2)} e^{-(v(u) - v(\tilde m_k))} du \leq e^{-\epsilon h_t} \ \ \ \text{and} \ \ \ \int_{\tilde \tau_k^+(h_t / 2)}^{\tilde L_k} e^{-(v(u) - v(\tilde m_k))} du \leq e^{-\epsilon h_t}. \label{invtl1.1.1}
%\end{eqnarray}
Combining with \eqref{rayknight2} (applied with $u = e^{\epsilon h_t /2}$) and the lower bound for $R_k^t$ in \eqref{invtl1.1.0} we get 
\begin{align}
P^v \left ( \int_{\tilde \tau_k^+(h_t /2)}^{\tilde L_k}e^{-\tilde v^{(k)}(x)} \mathcal{L}_{\tilde B}(\sigma_{\tilde B}(1,0), A^k(x))dx \leq e^{-\epsilon h_t /2} \leq e^{-\epsilon h_t /4} R_k^t \right ) & \geq 1 - e^{-c h_t}, \label{invtl1.1.2} \\
P^v \left ( \int_{\tilde L_{k-1}}^{\tilde \tau_k^-(h_t /2)}e^{-\tilde v^{(k)}(x)} \mathcal{L}_{\tilde B}(\sigma_{\tilde B}(1,0), A^k(x))dx \leq e^{- \epsilon h_t /2} \leq e^{-\epsilon h_t /4} R_k^t \right ) & \geq 1 - e^{-c h_t}. \label{invtl1.1.3}
\end{align}
Now putting \eqref{invtl1.1.1.1}, \eqref{invtl1.1.2} and \eqref{invtl1.1.3} into \eqref{formulei} (and combining with the fact that $\tilde c < \min (1/20, \epsilon /4)$) we get 
\[ P^v \left( \mathcal{E}_{neg}^1 := \left \{ \left | I - t Z R_k^t \right | > e^{- \tilde c h_t} t Z R_k^t \right \} \right) \leq e^{- c h_t}. \]
Combining with the definition of $Z$ we see that $t (1-z) \leq I$ on the complementary of $\mathcal{E}_{neg}^1$. Combining this with \eqref{invtl} we get that $\mathcal{E}^8_t(v, k, z) \cap \{ \sigma_{X_{\tilde m_k}} ( t \tilde x_t, \tilde m_k) < H_{X_{\tilde m_k}} (\tilde L_k) \}$ is included into
\begin{align}
\left \{ t Z \leq t \tilde x_t, \ \sup_{\mathcal{D}_k} \mathcal{L}_{X_{\tilde m_k}}(\sigma_{X_{\tilde m_k}} ( t Z, \tilde m_k), .) \geq t x_t, \ \sigma_{X_{\tilde m_k}} ( t \tilde x_t, \tilde m_k) < H_{X_{\tilde m_k}} (\tilde L_k) < H_{X_{\tilde m_k}}(\tilde L_{k-1}) \right \} \cup \mathcal{E}_{neg}^1. \label{invtl1.2}
\end{align}
%\begin{align}
%\left \{ t Z \leq t \tilde x_t, \ \sup_{\mathcal{D}_k} \mathcal{L}_{X_{\tilde m_k}}(\sigma ( t Z, \tilde m_k), .) \geq t x_t, H_{X_{\tilde m_k}} (\tilde L_k) > \sigma_{X_{\tilde m_k}} ( t \tilde x_t, \tilde m_k) \vee t(1-z), H_{X_{\tilde m_k}} (\tilde L_k) < H(\tilde L_{k-1}) \right \} \cup ... \label{invtl1.2}
%\end{align}
On the main event, $\sigma_{X_{\tilde m_k}} ( t Z, \tilde m_k)$ is finite. On this event, $\mathcal{L}_{X_{\tilde m_k}}(\sigma_{X_{\tilde m_k}} ( t Z, \tilde m_k), y)$ is thus equal to
\[ L(y) := e^{-\tilde v^{(k)}(y)} \mathcal{L}_B(\sigma_{B}(t Z,0), A^k(y)). \]
Here again, for $\tilde B := B((t Z)^2 .)/t Z$ and $\tilde A^k(.) := A^k(.) / t Z$. We have
\[ L(y) = t Z e^{-\tilde v^{(k)}(y)} \mathcal{L}_{\tilde B}(\sigma_{\tilde B}(1,0), \tilde A^k(y)). \]
Since $\mathcal{D}_k \subset [\tilde \tau_k^-(h_t /2), \tilde \tau_k^+(h_t /2)]$ we can apply \eqref{invtl1.1.1.1pre} (together with $\tilde c < 1/20$) and get 
\begin{eqnarray}
P^v \left( \mathcal{E}_{neg}^2 := \left \{ \exists y \in \mathcal{D}_k, \ L(y) \geq (1+e^{- \tilde c h_t}) t Z \right \} \right) \leq e^{-c h_t}. \label{invtl1.3}
\end{eqnarray}
On the big event in \eqref{invtl1.2} we thus have both $(1+e^{- \tilde c h_t}) t Z > t x_t$ and $t Z \leq t \tilde x_t$, (expect possibly on the event $\mathcal{E}_{neg}^2$). Since these two inequalities are not compatible we get 
\begin{eqnarray}
P^v \left( \mathcal{E}^8_t(v, k, z) \cap \{ \sigma_{X_{\tilde m_k}} ( t \tilde x_t, \tilde m_k) < H_{X_{\tilde m_k}} (\tilde L_k) \} \right ) \leq P^v \left( \mathcal{E}_{neg}^1 \right ) + P^v \left( \mathcal{E}_{neg}^2 \right ) \leq 2 e^{-c h_t}. \label{invtl1.4}
\end{eqnarray}

We now study the case where $\sigma_{X_{\tilde m_k}} ( t \tilde x_t, \tilde m_k) > H_{X_{\tilde m_k}} (\tilde L_k)$. First, we have
\begin{align*}
& \mathcal{E}^8_t(v, k, z) \cap \{ \sigma_{X_{\tilde m_k}} ( t \tilde x_t, \tilde m_k) > H_{X_{\tilde m_k}} (\tilde L_k) \} \nonumber \\
= & \left \{ t Z \leq t \tilde x_t, \ \sup_{\mathcal{D}_k} \mathcal{L}_{X_{\tilde m_k}}(t(1-z), .) \geq t x_t, \  H_{X_{\tilde m_k}} (\tilde L_k) < H_{X_{\tilde m_k}}(\tilde L_{k-1}) \wedge \sigma_{X_{\tilde m_k}} ( t \tilde x_t, \tilde m_k) \right \} \nonumber \\
\subset & \left \{ \sup_{\mathcal{D}_k} \mathcal{L}_{X_{\tilde m_k}}(t(1-z), .) \geq t x_t, \ \mathcal{L}_{X_{\tilde m_k}}(H_{X_{\tilde m_k}} (\tilde L_k), \tilde m_k) < t \tilde x_t \right \} \nonumber \\
\subset & \left \{ \sup_{\mathcal{D}_k} \mathcal{L}_{X_{\tilde m_k}}(H_{X_{\tilde m_k}} (\tilde L_k), .) \geq t x_t, \ \mathcal{L}_{X_{\tilde m_k}}(H_{X_{\tilde m_k}} (\tilde L_k), \tilde m_k) < t \tilde x_t \right \} \cup \overline{\mathcal{E}^1_t}. 
%\label{invtl1.5}
\end{align*}
Indeed, $\sup_{\mathcal{D}_k} \mathcal{L}_{X_{\tilde m_k}}(t(1-z), .) \leq \sup_{\mathcal{D}_k} \mathcal{L}_{X_{\tilde m_k}}(H_{X_{\tilde m_k}} (\tilde L_k), .)$ on $\{ t(1-z) \leq H_{X_{\tilde m_k}} (\tilde L_k) \}$ and $\sup_{\mathcal{D}_k} \mathcal{L}_{X_{\tilde m_k}}(t(1-z), .) = \sup_{\mathcal{D}_k} \mathcal{L}_{X_{\tilde m_k}}(H_{X_{\tilde m_k}} (\tilde L_k), .)$ on $\{ H_{X_{\tilde m_k}} (\tilde L_k) < t(1-z) \} \cap \mathcal{E}^1_t$. Since we are dealing with $X_{\tilde m_k}$, the diffusion shifted at time $H(\tilde m_k)$, we can see that the main event above is included in $\overline{\mathcal{E}^4_t}$. We thus get
\begin{eqnarray}
\mathcal{E}^8_t(v, k, z) \cap \{ \sigma_{X_{\tilde m_k}} ( t \tilde x_t, \tilde m_k) > H_{X_{\tilde m_k}} (\tilde L_k) \} \subset \overline{\mathcal{E}^1_t} \cup \overline{\mathcal{E}^4_t}. \label{invtl1.4.1}
\end{eqnarray}
\eqref{majoe8bis} follows easily from the combination of \eqref{invtl1.4} and \eqref{invtl1.4.1} applied with $z = 1-4/\log(h_t)$. Then, the right hand sides of \eqref{invtl1.4} and \eqref{invtl1.4.1} do not depend on $z$ (which is arbitrary in $[0, 1-4/\log(h_t)]$) and $H(\tilde m_{k})/t$ is independent from $X_{\tilde m_k}$. We can thus replace $z$ by $H(\tilde m_{k})/t$ in \eqref{invtl1.4} and \eqref{invtl1.4.1} (at least on $\{ N_t \geq k, \ H(\tilde m_{k})/t \leq 1 - 4/\log(h_t) \}$). Using the combination of \eqref{invtl1.4} and \eqref{invtl1.4.1} to study the union of events in \eqref{majoe8} we get
\[ P^v \left ( \cup_{k=1}^{n_t} \{ N_t \geq k, \ H(\tilde m_{k})/t \leq 1 - 4/\log(h_t) \} \cap \mathcal{E}^8_t(v, k, H(\tilde m_{k})/t) \right ) \leq 2 n_t e^{-c h_t} + P^v \left ( \overline{\mathcal{E}^1_t} \cup \overline{\mathcal{E}^4_t} \right ). \]
Since $v \in \mathcal{G}_t$ it satisfies \eqref{bonenv1}. We thus have $P^v ( \overline{\mathcal{E}^1_t} \cup \overline{\mathcal{E}^4_t}) \leq e^{-L \phi(t)/2}$ (where $L$ is the constant defined in Fact \ref{bigevents}) and from the definition of $h_t$ and $\phi(t)$ we have easily $2 n_t e^{-c h_t} \leq e^{-c \phi(t)}$ for large $t$. For $t$ large enough we thus get \eqref{majoe8}. 

We now prove \eqref{majoe9}. Let us fix an environment $v \in \mathcal{G}_t$, $z \in [0, 1]$ and $k \leq n_t$. Inversing the local time in the definition of $\mathcal{E}^9_t(v, k, z)$ we get that $\mathcal{E}^9_t(v, k, z)$ coincides with the event
%\begin{align}
%& \left \{ \mathcal{L}_{X_{\tilde m_k}}(t(1-z), \tilde m_k) \leq t / x_t, H_{X_{\tilde m_k}} (\tilde L_k) \geq t(1-z), H_{X_{\tilde m_k}} (\tilde L_k) < H(\tilde L_{k-1}) \right \} \nonumber \\
%= & \left \{ \sigma ( t / x_t, \tilde m_k) \wedge H_{X_{\tilde m_k}} (\tilde L_k) \geq t(1-z), H_{X_{\tilde m_k}} (\tilde L_k) < H(\tilde L_{k-1}) \right \}. \label{invtl2}
%\end{align}
\begin{eqnarray}
& \left \{ R_k^t / x_t < (1-e^{- \tilde c h_t})(1-z), \ \sigma_{X_{\tilde m_k}} ( t / x_t, \tilde m_k) \wedge H_{X_{\tilde m_k}} (\tilde L_k) \geq t(1-z), \right. \nonumber \\
& \left. H_{X_{\tilde m_k}} (\tilde L_k) < H_{X_{\tilde m_k}}(\tilde L_{k-1}) \right \} . \label{invtl2}
\end{eqnarray}
We have to distinguish the cases $H_{X_{\tilde m_k}} (\tilde L_k) > \sigma_{X_{\tilde m_k}} (t / x_t,\tilde m_k)$ and $\sigma_{X_{\tilde m_k}} (t / x_t,\tilde m_k) > H_{X_{\tilde m_k}} (\tilde L_k)$. On $\mathcal{E}^9_t(v, k, z) \cap \{ \sigma_{X_{\tilde m_k}} ( t \tilde x_t, \tilde m_k) < H_{X_{\tilde m_k}} (\tilde L_k) \}$, 
%$\{H_{X_{\tilde m_k}} (\tilde L_k) > \sigma ( t / x_t, \tilde m_k) \geq t(1-z), H_{X_{\tilde m_k}} (\tilde L_k) < H(\tilde L_{k-1})\}$, 
$\sigma_{X_{\tilde m_k}} ( t / x_t, \tilde m_k)$ is finite and the diffusion stays in $[\tilde L_{k-1}, \tilde L_k]$ until this time. On this event, $\sigma_{X_{\tilde m_k}} ( t / x_t, \tilde m_k)$ is thus equal to 
\[ I' := \int_{\tilde L_{k-1}}^{\tilde L_k}e^{-\tilde v^{(k)}(x)} \mathcal{L}_B(\sigma_{B}(t / x_t,0), A^k(x))dx, \]
where, as in \eqref{invtl1.1}, $B$ is the Brownian motion driving the diffusion $X_{\tilde m_k}$. Now, let $\tilde B := B((t / x_t)^2 .)/(t / x_t)$ and $\tilde A^k(.) := A^k(.) / (t / x_t)$. From the scaling property $\tilde B$ is still a Brownian motion. We can write
\[ I' = (t / x_t) \int_{\tilde L_{k-1}}^{\tilde L_k}e^{-\tilde v^{(k)}(x)} \mathcal{L}_{\tilde B}(\sigma_{\tilde B}(1,0), \tilde A^k(x))dx. \]
From $t \geq e^{h_t}$, the definition \eqref{defxtlim} of $x_t$ and the definition \eqref{paramtaillevallees} of $h_t$ we see that $t/x_t \geq e^{3 h_t /4}$ for $t$ large enough. Recall also that $v \in \mathcal{G}_t$ implies that $v$ satisfies \eqref{invtl1.1.00} and \eqref{invtl1.1.001}. We thus get
\[ \sup_{[\tilde \tau_k^-(h_t / 2), \tilde \tau_k^+(h_t / 2)]} \left | \tilde A^k(.) \right | \leq \left ( \left | A^k(\tilde \tau_k^-(h_t / 2)) \right | \vee A^k(\tilde \tau_k^+(h_t / 2)) \right ) / (t / x_t) \leq e^{- h_t /8}, \]
We can now proceed as in the proof of \eqref{majoe8} to get analogues of \eqref{invtl1.1.1.1}, \eqref{invtl1.1.2} and \eqref{invtl1.1.3}. We deduce that
\[ P^v \left( \mathcal{E}_{neg}^3 := \left \{ \left | I' - (t / x_t) R_k^t \right | > e^{- \tilde c h_t} (t / x_t) R_k^t \right \} \right) \leq e^{- c h_t}. \]
Now, note from \eqref{invtl2} that $\mathcal{E}^9_t(v, k, z) \subset \{ (1-e^{- \tilde c h_t}) \sigma_{X_{\tilde m_k}} ( t / x_t, \tilde m_k) \geq t R_k^t / x_t \}$. Since \\ $\sigma_{X_{\tilde m_k}} ( t / x_t, \tilde m_k) = I'$ on $\mathcal{E}^9_t(v, k, z) \cap \{ \sigma_{X_{\tilde m_k}} ( t / x_t, \tilde m_k) < H_{X_{\tilde m_k}} (\tilde L_k) \}$ and $(1+e^{- \tilde c h_t}) (1-e^{- \tilde c h_t}) < 1$ we deduce that $\mathcal{E}^9_t(v, k, z) \cap \{ \sigma_{X_{\tilde m_k}} ( t / x_t, \tilde m_k) < H_{X_{\tilde m_k}} (\tilde L_k) \} \subset \mathcal{E}_{neg}^3$. Then, 
%\begin{eqnarray}
%P^v \left ( R_k^t / x_t < (1+e^{- \tilde c h_t})^{-1}(1-z), H_{X_{\tilde m_k}} (\tilde L_k) > \sigma ( t / x_t, \tilde m_k) \geq t(1-z), H_{X_{\tilde m_k}} (\tilde L_k) < H(\tilde L_{k-1}) \right ) \leq ... \label{invtl3}
%\end{eqnarray}
\begin{eqnarray}
P^v \left( \mathcal{E}^9_t(v, k, z) \cap \{ \sigma_{X_{\tilde m_k}} ( t / x_t, \tilde m_k) < H_{X_{\tilde m_k}} (\tilde L_k) \} \right ) \leq e^{-c h_t}. \label{invtl3}
\end{eqnarray}

%We now study the case where $\sigma_{X_{\tilde m_k}} (t / x_t,\tilde m_k) > H_{X_{\tilde m_k}} (\tilde L_k)$. According to fact \ref{invtlinf} applied with $y_t = 1/x_t$ we deduce that 
%\[ \left \{ \sigma_{X_{\tilde m_k}} (t / x_t,\tilde m_k) > H_{X_{\tilde m_k}} (\tilde L_k) \right \} \subset \left \{ t R_k / x_t > (1-...) H_{X_{\tilde m_k}} (\tilde L_k) \right \} \cup \overline{\mathcal{E}^7_t} \]

We now study the case where $\sigma_{X_{\tilde m_k}} (t / x_t,\tilde m_k) > H_{X_{\tilde m_k}} (\tilde L_k)$, following the ideas of the proof of Lemma 5.3 of \cite{advech}. Recall that $B$ is the Brownian motion driving the diffusion $X_{\tilde m_k}$. We have
\[ \sigma_{X_{\tilde m_k}} (t / x_t,\tilde m_k) > H_{X_{\tilde m_k}} (\tilde L_k) \Leftrightarrow \sigma_B(t / x_t, 0) > \tau(B, A^k(\tilde L_k)) \Leftrightarrow t / x_t > \mathcal{L}_B \left ( \tau(B, A^k(\tilde L_k)), 0 \right ) \]
Now, let $\tilde B := B((A^k(\tilde L_k))^2 .)/A^k(\tilde L_k)$. From the scaling property $\tilde B$ is still a Brownian motion. $\mathcal{L}_B \left ( \tau(B, A^k(\tilde L_k)), 0 \right ) = A^k(\tilde L_k) \mathcal{L}_{\tilde B} \left ( \tau(\tilde B, 1), 0 \right )$. Note that from the definition of $e_k$ given in Subsection \ref{trapsdiff} we have $\mathcal{L}_{\tilde B} \left ( \tau{\tilde B}(1), 0 \right ) = e_k$. As a consequence, 
 \[ \sigma_{X_{\tilde m_k}} (t / x_t,\tilde m_k) > H_{X_{\tilde m_k}} (\tilde L_k) \Leftrightarrow t / x_t > A^k(\tilde L_k) e_k \Leftrightarrow t R_k^t / x_t > A^k(\tilde L_k) e_k R_k^t. \]
Using the trivial inequality 
\[ A^k(\tilde L_j) = \int_{\tilde m_j}^{\tilde L_j} e^{v(u) - v(\tilde m_j)} du \geq \int_{\tilde \tau_j^+(h_t / 2)}^{\tilde L_j} e^{v(u) - v(\tilde m_j)} du = S_j^t \] 
and the definition of $\mathcal{E}^7_t$ we deduce that 
\[ \left \{ \sigma_{X_{\tilde m_k}} (t / x_t,\tilde m_k) > H_{X_{\tilde m_k}} (\tilde L_k) \right \} \subset \left \{ t R_k^t / x_t > (1+e^{- \tilde c h_t})^{-1} H_{X_{\tilde m_k}} (\tilde L_k) \right \} \cup \overline{\mathcal{E}^7_t} \]
Now, note from \eqref{invtl2} that $\mathcal{E}^9_t(v, k, z) \subset \{ (1-e^{- \tilde c h_t}) H_{X_{\tilde m_k}} (\tilde L_k) \geq t R_k^t / x_t \}$. Since $(1+e^{- \tilde c h_t}) (1-e^{- \tilde c h_t}) < 1$, the inequality $t R_k^t / x_t > (1+e^{- \tilde c h_t})^{-1} H_{X_{\tilde m_k}} (\tilde L_k)$ is in contradiction with the event $\mathcal{E}^9_t(v, k, z)$. We thus get
\begin{eqnarray}
\mathcal{E}^9_t(v, k, z) \cap \{ \sigma_{X_{\tilde m_k}} ( t / x_t, \tilde m_k) > H_{X_{\tilde m_k}} (\tilde L_k) \} \subset \overline{\mathcal{E}^7_t} \label{invtl4}
\end{eqnarray}
%We thus get
%\begin{eqnarray}
%P^v \left ( R_k^t / x_t < (1-...)(1-z), \sigma ( t / x_t, \tilde m_k) > H_{X_{\tilde m_k}} (\tilde L_k) \geq t(1-z), H_{X_{\tilde m_k}} (\tilde L_k) < H(\tilde L_{k-1}) \right ) \leq ... \label{invtl4}
%\end{eqnarray}
Since the right hand sides of \eqref{invtl3} and \eqref{invtl4} do not depend on $z$ (which is arbitrary in $[0, 1]$) and $H(\tilde m_{k})/t$ is independent from $X_{\tilde m_k}$, we can replace $z$ by $H(\tilde m_{k})/t$ in \eqref{invtl3} and \eqref{invtl4} (at least on $\{ N_t \geq k \}$). We can thus use the combination of \eqref{invtl3} and \eqref{invtl4} to study the union of events in \eqref{majoe9}. We get
\[ P^v \left ( \cup_{k=1}^{n_t} \{ N_t \geq k \} \cap \mathcal{E}^9_t(v, k, H(\tilde m_{k})/t) \right ) \leq n_t e^{-c h_t} + P^v \left ( \overline{\mathcal{E}^7_t} \right ) \]
Since $v \in \mathcal{G}_t$ it satisfies \eqref{bonenv1}. We thus have $P^v ( \overline{\mathcal{E}^7_t}) \leq e^{-L \phi(t)/2}$ (where $L$ is the constant defined in Fact \ref{bigevents}) and from the definition of $h_t$ and $\phi(t)$ we have easily $n_t e^{-c h_t} \leq e^{-c \phi(t)}$ for large $t$. For $t$ large enough we thus get \eqref{majoe9}. 
%Finally, the combination of \eqref{invtl2}, \eqref{invtl3} and \eqref{invtl4} gives
%\[ P^v \left ( R_k^t / x_t < (1-...)(1-z), \mathcal{L}_{X_{\tilde m_k}}(t(1-z), \tilde m_k) \leq t / x_t, H_{X_{\tilde m_k}} (\tilde L_k) \geq t(1-z), H(\tilde L_k) < H(\tilde L_{k-1}) \right ) \leq ... \]
%and this is precisely (...). 

\end{proof}

\begin{proof} of Fact \ref{queueiid} 

\eqref{cvmeasure7.1} and \eqref{cvmeasure9.1} are included into Lemma 4.16 of \cite{caslevyvech} while \eqref{cvrlapl} comes from Proposition 3.8 of \cite{caslevyvech} applied with $h = h_t$. 
\end{proof}

\begin{proof} of Lemma \ref{infsupdiff}

First, note that from \eqref{paramtaillevallees} and the definition of $n_t$ just after, we have $2 t^{\kappa} e^{\kappa \delta (\log (\log (t)))^{\omega}} /q = 2 n_t e^{\kappa h_t} /q$. From the definition of $N_t$, we know that on $\{ V \in \mathcal{V}_t \}$, at time $t$ , $\tilde m_{N_t +1}$ has never been reached by the diffusion and neither $\tilde L_{N_t +1}$ (because $\tilde L_{N_t +1} > \tilde m_{N_t +1}$), we thus have
\begin{align*}
\mathbb{P} \left ( \sup_{[0, t]} X \geq 2 n_t e^{\kappa h_t} /q \right ) & \leq \mathbb{P} \left ( \tilde L_{N_t +1} \geq 2 n_t e^{\kappa h_t} /q \right ) + \mathbb{P} \left (  V \notin \mathcal{V}_t \right ) \\
& \leq \mathbb{P} \left ( \tilde L_{n_t} \geq 2 n_t e^{\kappa h_t} /q \right ) + \mathbb{P} \left (  V \notin \mathcal{V}_t \right ) + \mathbb{P} \left ( N_t \geq n_t \right ) \\
& \leq \mathbb{P} \left ( (1 + e^{-c_1 h_t}) \sum_{j=1}^{n_t} D_j^t \geq 2 n_t e^{\kappa h_t} /q \right ) + e^{- c h_t}, 
\end{align*}
where $c$ is a positive constant, $c_1$ has the same meaning as in Proposition \ref{contribdist} and $t$ is large enough. For the last inequality, we used Proposition \ref{contribdist}, Fact \ref{minimacoincide} and \eqref{majonbvalleesvisit}. 
\[ \leq e^{- (1 + e^{-c_1 h_t})^{-1} n_t} \left ( \mathbb{E} \left [ e^{q e^{-\kappa h_t} D_j^t/2} \right ] \right )^{n_t} + e^{- c h_t} = e^{(\log(2)- (1 + e^{-c_1 h_t})^{-1}) n_t} + e^{- c h_t}, \]
where we used Markov's inequality, the fact that the sequence $(q e^{-\kappa h_t} D_j^t)$ is an \textit{iid} sequence of exponential random variable with parameter $1$ and the expression of the Laplace transform for the exponential distribution. Since $\log(2) < 1$ and $n_t = e^{\kappa (1+\delta) (\log (\log(t)))^{\omega}} >> h_t$, we get the first point for $t$ large enough. 

We now prove the second point. We first note that from \eqref{paramtaillevallees} and the definition of $\tilde n_t$ just after, we have $t^{\kappa} e^{(\rho - \kappa) (\log (\log (t)))^{\omega}} /2q \underset{t \rightarrow +\infty}{\sim} \tilde n_t e^{\kappa h_t} /2q$. On $\{ V \in \mathcal{V}_t \} \cap \{ N_t < n_t \} \cap \mathcal{E}^1_t$ we have $X(t) \geq \tilde L_{N_t -1}$ so
\begin{align*}
\mathbb{P} \left ( X(t) \leq \tilde n_t e^{\kappa h_t} /2q \right ) & \leq \mathbb{P} \left ( \tilde L_{N_t -1} \leq \tilde n_t e^{\kappa h_t} /2q \right ) + \mathbb{P} \left (  V \notin \mathcal{V}_t \right ) + \mathbb{P} \left ( N_t \geq n_t \right ) + \mathbb{P} \left ( \overline{\mathcal{E}^1_t} \right ) \\
& \leq \mathbb{P} \left ( \tilde L_{\tilde n_t} \leq \tilde n_t e^{\kappa h_t} /2q \right ) + \mathbb{P} \left (  V \notin \mathcal{V}_t \right ) + \mathbb{P} \left ( N_t \leq \tilde n_t \right ) + \mathbb{P} \left ( N_t \geq n_t \right ) + \mathbb{P} \left ( \overline{\mathcal{E}^1_t} \right ) \\
& \leq \mathbb{P} \left ( (1 - e^{-c_1 h_t}) \sum_{j=1}^{\tilde n_t} D_j^t \leq \tilde n_t e^{\kappa h_t} /2q \right ) + e^{- c \phi(t)}
\end{align*}
where $c$ is a positive constant, $c_1$ has the same meaning as in Proposition \ref{contribdist} and $t$ is large enough. For the last inequality we used Proposition \ref{contribdist}, Fact \ref{minimacoincide}, Lemma \ref{minmajnbvalleesvisit}, \eqref{bigeventsleq} and the fact that $e^{- c h_t} \leq e^{- c \phi(t)}$ for large $t$. 
\[ \leq e^{(1 - e^{-c_1 h_t})^{-1} \tilde n_t} \left ( \mathbb{E} \left [ e^{-2 q e^{-\kappa h_t} D_j^t} \right ] \right )^{\tilde n_t} + e^{- c \phi(t)} = e^{(-\log(3) + (1 - e^{-c_1 h_t})^{-1}) \tilde n_t} + e^{- c \phi(t)}, \]
where we used Markov's inequality, the fact that the sequence $(q e^{-\kappa h_t} D_j^t)_{j \geq 1}$ is an \textit{iid} sequence of exponential random variable with parameter $1$ and the expression for the Laplace transform of the exponential distribution. Since $\log(3) > 1$ and $\tilde n_t = \lfloor e^{\rho (\log (\log(t)))^{\omega}} \rfloor >> h_t$, we get the second point for $t$ large enough. 

We now prove \eqref{minoinfndiff}. Note that for a fixed environment $v \in \mathcal{V}$ we have
%\begin{align*}
%P^v \left ( \inf_{[0, +\infty[} X \leq truc \right ) & = P^v \left ( H(truc) < H(+\infty) \right ) = \int_{0}^{+\infty} e^{v(y)} dy / \int_{truc}^{+\infty} e^{v(y)} dy \\
%& \leq \int_{0}^{+\infty} e^{v(y)} dy / \int_0^{truc} e^{v(y)} dy
%We have
%\begin{align*}
%\mathbb{P} \left ( \int_{truc}^{+\infty} e^{V(y)} dy \leq ... \right ) & \leq \mathbb{P} \left ( V(truc /2) \leq ... \right ) + \mathbb{P} \left ( \inf_{[truc, truc/2]} V(.) - V(truc /2) \leq ... \right ) \\
%& = \mathbb{P} \left ( V(truc /2) \geq ... \right ) + \mathbb{P} \left ( \sup_{[0, truc/2]} V \geq ... \right ), 
%\end{align*}
%where we used the fact that $(V(-x), x \geq 0)$ is equal in law to $-V$ and the stationarity of increments. 
%\begin{align*}
%& \leq \mathbb{P} \left ( e^{\kappa V(truc /2) /2} \geq e^{\kappa ... /2} \right ) + \mathbb{P} \left ( \sup_{[0, +\infty]} V \geq ... \right ) \\
%& \leq e^{truc \psi_V(\kappa/2) /2 - \kappa .../2} + e^{-\kappa ...}, 
%\end{align*}
%where we used Markov inequality and the definition of $\psi_V$ for the first term, the fact that the supremum of $V$ follows an exponential distribution with parameter $\kappa$ for the second term. 
\begin{eqnarray}
P^v \left ( \inf_{[0, +\infty[} X \leq -r \right ) = P^v \left ( H(-r) < H(+\infty) \right ) = \frac{\int_{0}^{+\infty} e^{v(x)} dx}{\int_{-r}^{+\infty} e^{v(x)} dx} \leq \frac{\int_{0}^{+\infty} e^{v(x)} dx}{\int_{-r}^{-r/2} e^{v(x)} dx} \label{lemmeinfmajocompdn1}
\end{eqnarray}
and note that
\[ \int_{-r}^{-r/2} e^{V(x)} dx = \int_{r/2}^{r} e^{V(-x)} dx \overset{\mathcal{L}}{=} \int_{r/2}^{r} e^{-V(x)} dx, \]
where the equality in law comes from the time-reversal property. Applying Lemma \ref{supdev} with $t = r /2$ and $a = \sqrt{r}$ we get
\begin{eqnarray}
P \left ( \int_{r/2}^{r} e^{-V(x)} dx \leq \frac{r}{2} e^{ \sqrt{r}} \right ) \leq e^{c_1 \sqrt{r} - c_2 r /2} + e^{-\kappa \sqrt{r}}, \label{lemmeinfmajocompdn2}
\end{eqnarray}
where $c_1$ and $c_2$ are the constants in the lemma. Then, applying Lemma \ref{foncexpov} we get for $r$ large enough, 
\begin{eqnarray}
P \left ( \int_{0}^{+\infty} e^{V(x)} dx \geq e^{ \sqrt{r}} \right ) \leq e^{-\kappa \sqrt{r} / 2}. \label{lemmeinfmajocompdn3}
\end{eqnarray}

Putting \eqref{lemmeinfmajocompdn2} and \eqref{lemmeinfmajocompdn3} into \eqref{lemmeinfmajocompdn1} we get that with $P$-probability greater than $1 - (e^{c_1 \sqrt{r} - c_2 r /2} + e^{-\kappa \sqrt{r}} + e^{-\kappa \sqrt{r}/2})$: $P^V \left ( \inf_{[0, +\infty[} X \leq -r \right ) \leq 2 r^{-1}$ (and it is bounded by $1$ when this estimates fails) so integrating on $\mathcal{V}$ with respect to $P$ we get
\[ \mathbb{P} \left ( \inf_{[0, +\infty[} X \leq -r \right ) \leq 2 r^{-1} + e^{c_1 \sqrt{r} - c_2 r /2} + e^{-\kappa \sqrt{r}} + e^{-\kappa \sqrt{r}/2}. \]
\eqref{minoinfndiff} follows for $r$ large enough. Finally, \eqref{queuetlneg} is included in Lemma 5.19 of \cite{caslevyvech}. 

\end{proof}

\subsection{Almost sure constantness of $\limsup$ and $\liminf$} \label{cteness}

We now use a classical argument involving Kolmogorov $0-1$ law to justify the almost sure constantness of the $\limsup$ stated in Remark \ref{limsupetinfcte}. We here treat the case of the $\limsup$ with the renormalization $t \log (\log(t))$. The same argument can be used with the $\liminf$ instead of the $\limsup$, or with any of the other renormalizations used in the paper. 

We first fix $v \in \mathcal{V}$, a realization of the environment. For any $n \in \mathbb{N}$, the process $X^n := X(\tau(X, n) + .)$ is, according to the Markov property, a diffusion in the environment $v(n + .)$ and it is independent from $(X(s), \ 0 \leq s \leq \tau(X, n))$. We have
% \overset{\mathcal{L}}{=} V$, so it has the same law as $X$. In particular, 
\begin{eqnarray}
\limsup_{t \rightarrow + \infty} \frac{\mathcal{L}_{X}^*(t)}{t \log (\log(t))} = \limsup_{t \rightarrow + \infty} \frac{\mathcal{L}_{X^n}^*(t)}{(\tau(X, n) + t) \log (\log(\tau(X, n) + t))} = \limsup_{t \rightarrow + \infty} \frac{\mathcal{L}_{X^n}^*(t)}{t \log (\log(t))}, \label{limsupcteas}
\end{eqnarray}
where the first equality comes from the fact that $\mathbb{P}$-almost surely the favorite site $F^*(t)$ goes to $+\infty$. Indeed, the diffusion $\mathbb{P}$-almost surely converges to $+\infty$ and as we can see from the results of Subsection \ref{results}, $\mathcal{L}_{X}^*(t)$ converges $\mathbb{P}$-almost surely to infinity, these two facts imply the convergence of $F^*(t)$ to $+\infty$. As a consequence $F^*(t)$ will become greater than $n$ for $t$ large enough which implies the first equality in \eqref{limsupcteas}. The second equality comes from the equivalence when $t$ goes to infinity between $(\tau(X, n) + t) \log (\log(\tau(X, n) + t))$ and $t \log (\log(t))$. 

The $\limsup$ in the right hand side of \eqref{limsupcteas} belongs to the $\sigma$-field $\sigma (X^n(t), \ t \geq 0$, it is thus independent from $\sigma (X(s), \ 0 \leq s \leq \tau(X, n))$. Since this is true for any $n \in \mathbb{N}$ we get, according to Kolmogorov $0-1$ law, that the $\limsup$ is constant $P^v$-almost surely. In other words, the $\limsup$ is only a deterministic function of the environment $v$, let us denote it by $L(v)$. 

Let us fix $v \in \mathcal{V}$ and $n \in \mathbb{N}$. Note that $P^v (\inf X^n \geq n-1) > 0$ so the $\limsup$ still equals $L(v)$ on $\{ \inf X^n \geq n-1 \}$. As a consequence, $L(v)$ is only a function of $(v(x + n-1) - v(n-1), \ x \geq 0)$. 
%Since almost surely, for any $n \in \mathbb{N}$, the interval $]-\infty, n]$ is left in finite time, the value of the $\limsup$, $L(v)$, is only a function of $v(n + x) - v(n)$. 
If we consider the space $\mathcal{V}$ equipped with probability $P$, this implies that $L(V)$ is independent from the $\sigma$-field $\sigma (V(x), \ x \leq n-1)$. Since this is true for all $n \in \mathbb{N}$ we deduce from Kolmogorov $0-1$ law that $L(V)$ is constant $P$-almost surely. This proves that the $\limsup$ is in fact constant $\mathbb{P}$-almost surely. In other words, 
\[ \exists \lambda \in [0, +\infty] \ \text{such that} \ \mathbb{P}-a.s. \ \limsup_{t \rightarrow + \infty} \frac{\mathcal{L}_{X}^*(t)}{t \log (\log(t))} = \lambda. \]

\bibliographystyle{plain}
\bibliography{thbiblio}

\begin{thebibliography}{10}

\bibitem{Pierre2}
P.~Andreoletti.
\newblock On the concentration of {S}inai's walk.
\newblock {\em Stoch. Proc. Appl.}, 116:1377--1408, 2006.

\bibitem{Pierre3}
P.~Andreoletti.
\newblock Almost sure estimates for the concentration neighborhood of {S}inai's
  walk.
\newblock {\em Stoch. Proc. Appl.}, 117, 2007.

\bibitem{Pierre4}
P.~Andreoletti.
\newblock On the estimation of the potential of {S}inai's rwre.
\newblock {\em Braz. J. Probab. Stat.}, 25:121--235, 2011.

\bibitem{AndDev}
P.~Andreoletti and A.~Devulder.
\newblock Localization and number of visited valleys for a transient diffusion
  in random environment.
\newblock {\em Electronic Journal of Probability}, pages \ 1--59, 2015.

\bibitem{advech}
P.~Andreoletti, A.~Devulder, and G.~V\'echambre.
\newblock Renewal structure and local time for diffusions in random
  environment.
\newblock {\em ALEA, Lat. Am. J. Probab. Math. Stat.}, \textbf{13}:863--923,
  2016.

\bibitem{AndDiel}
P.~Andreoletti and R.~Diel.
\newblock Limit law of the local time for brox's diffusion.
\newblock {\em J. Theoretical Probab.}, \textbf{24}:\ 634--656, 2011.

\bibitem{Bertoin}
J.~Bertoin.
\newblock {\em L\'evy Processes}.
\newblock Tracts in Mathematics, Cambridge, 1996.

\bibitem{regvar}
N.~H. Bingham, C.~M. Goldie, and J.~L. Teugels.
\newblock {\em Regular Variation}.
\newblock Cambridge University Press, 1987.
\newblock Cambridge Books Online.

\bibitem{Brox}
T.~Brox.
\newblock A one-dimensional diffusion process in a {W}iener medium.
\newblock {\em Ann. Probab.}, \textbf{14}(4):\ 1206--1218, 1986.

\bibitem{Carmona}
P.~Carmona.
\newblock The mean velocity of a {B}rownian motion in a random {L}\'evy
  potential.
\newblock {\em Ann. Probab.}, \textbf{25}(4):\ 1774--1788, 1997.

\bibitem{Cheliotis2006715}
D.~Cheliotis.
\newblock One-dimensional diffusion in an asymmetric random environment.
\newblock {\em Annales de l'Institut Henri Poincare (B) Probability and
  Statistics}, 42(6):715 -- 726, 2006.

\bibitem{Comets_etal}
F.~Comets, M.~Falconnet, O.~Loukianov, D.~Loukianova, and C.~Matias.
\newblock Maximum likelihood estimator consistency for ballistic random walk in
  a parametric random environment.
\newblock {\em Stochastic Processes and Applications}, 124(1):268--288, 2014.

\bibitem{DemGanPerShi}
A.~Dembo, N.~Gantert, Y.~Peres, and Z.~Shi.
\newblock Valleys and the maximal local time for random walk in random
  environment.
\newblock {\em Probability Theory and Related Fields}, 137:\ 443--473, 2007.

\bibitem{Devmaxloc}
A.~Devulder.
\newblock The maximum of the local time of a diffusion process in a drifted
  {B}rownian potential.
\newblock {\em Preprint, arxiv 0604078}, 2015.

\bibitem{devulder2016}
A.~Devulder.
\newblock Persistence of some additive functionals of {S}inai's walk.
\newblock {\em Ann. Inst. H. Poincar\'e Probab. Statis.}, 52(3):\ 1076--1105,
  2016.

\bibitem{Diel}
R.~Diel.
\newblock Almost sure asymptotics for the local time of a diffusion in
  {B}rownian environment.
\newblock {\em Stoch. Proc. Appl.}, \textbf{121}:\ 2303--2330, 2011.

\bibitem{Dielvois}
R.~Diel and G.~Voisin.
\newblock Local time of a diffusion in a stable {L}\'evy environment.
\newblock {\em Stochastics An International Journal of Probability and
  Stochastic Processes}, 83(2):127--152, 2011.

\bibitem{Doney}
R.~Doney.
\newblock {\em Fluctuation Theory for L\'evy Processes (Summer school, St-Floor
  France)}.
\newblock Springer, 2005.

\bibitem{DoMoFi}
P.~Le Doussal, C.~Monthus, and D.~Fisher.
\newblock Random walkers in one-dimensional random environments : Exact
  renormalization group analysis.
\newblock {\em Phys, Rev. E}, pages \ 4795--4840, 1999.

\bibitem{enriquez2013}
N.~Enriquez, C.~Sabot, L.~Tournier, and O.~Zindy.
\newblock Quenched limits for the fluctuations of transient random walks in
  random environment on $\mathbb{Z}$.
\newblock {\em Ann. Appl. Probab.}, 23(3):1148--1187, 06 2013.

\bibitem{EnSaZi}
N.~Enriquez, C.~Sabot, and O.~Zindy.
\newblock Aging and quenched localization for one dimensional random walks in
  random environment in the sub-ballistic regime.
\newblock {\em Bulletin de la Soci\'et\'e Math\'ematique de France}, 137:\
  423--452, 2009.

\bibitem{Enriquez200900}
N.~Enriquez, C.~Sabot, and O.~Zindy.
\newblock Limit laws for transient random walks in random environment on
  $\mathbb{Z}$.
\newblock {\em Annales de l'institut Fourier}, 59(6):2469--2508, 2009.

\bibitem{Faraud2011}
G.~Faraud.
\newblock Estimates on the speedup and slowdown for a diffusion in a drifted
  {B}rownian potential.
\newblock {\em Journal of Theoretical Probability}, 2011.

\bibitem{GanPerShi}
N.~Gantert, Y.~Peres, and Z.~Shi.
\newblock The infinite valley for a recurrent random walk in random
  environment.
\newblock {\em Ann. Inst. H. Poincar\'e Probab. Statist.}, 46(2):525--536, 05
  2010.

\bibitem{GanShi}
N.~Gantert and Z.~Shi.
\newblock Many visits to a single site by a transient random walk in random
  environment.
\newblock {\em Stoch. Proc. Appl.}, 99:\ 159--176, 2002.

\bibitem{HuShi1}
Y.~Hu and Z.~Shi.
\newblock The local time of simple random walk in random environment.
\newblock {\em J. of Theoret. Probab.}, \textbf{11}(3), 1998.

\bibitem{hu2004}
Y.~Hu and Z.~Shi.
\newblock Moderate deviations for diffusions with {B}rownian potentials.
\newblock {\em Ann. Probab.}, 32(4):3191--3220, 10 2004.

\bibitem{HuShYo}
Y.~Hu, Z.~Shi, and M.~Yor.
\newblock Rates of convergence for one-dimensional diffusion in a drifted
  {B}rownian potential.
\newblock {\em Trans. Amer, Math. Soc}, 2000.

\bibitem{KawazuTanaka}
K.~Kawazu and H.~Tanaka.
\newblock A diffusion process in a {B}rownian environment with drift.
\newblock {\em J. Math. Soc. Japan}, \textbf{49}, pages={\ 189--211},
  year={1997}, %OPTvolume = {}, %OPTnumber = {}, %OPTpages = {}, %OPTmonth =
  {}, %OPTnote = {}, %OPTannote = {}.

\bibitem{kawazu1998}
K.~Kawazu and H.~Tanaka.
\newblock Invariance principle for a {B}rownian motion with large drift in a
  white noise environment.
\newblock {\em Hiroshima Math. J.}, 28(1):129--137, 1998.

\bibitem{Kochler}
M.~Kochler.
\newblock {\em Random Walks in Random Environment, Random Orientations and
  Branching}.
\newblock PhD thesis, Technische Universitat Munchen, 2012.

\bibitem{Mathieu}
P.~Mathieu.
\newblock Zero white noise limit through dirichlet forms, with application to
  diffusions in a random media.
\newblock {\em Proba. Theory Relat. Fields}, 99:\ 549--480, 1994.

\bibitem{Mathieu2}
P.~Mathieu.
\newblock Limit theorems for diffusions with a random potential.
\newblock {\em Stochastic Process. Appl.}, \textbf{60}:\ 103--111, 1995.

\bibitem{Revesz}
P.~R\'ev\'esz.
\newblock {\em Random walk in random and non-random environments}.
\newblock World Scientific, 1989.

\bibitem{Schumacher}
S.~Schumacher.
\newblock Diffusions with random coefficients.
\newblock {\em Contemp. Math.}, \textbf{41}:\ 351--356, 1985.

\bibitem{Shi}
Z.~Shi.
\newblock A local time curiosity in random environment.
\newblock {\em Stoch. Proc. Appl.}, \textbf{76}(2):\ 231--250, 1998.

\bibitem{Shi1}
Z.~Shi.
\newblock {S}inai's walk via stochastic calculus.
\newblock {\em Panoramas et Synth\`eses}, \textbf{12}:\ 53--74, 2001.

\bibitem{Singh2007101}
A.~Singh.
\newblock Limiting behavior of a diffusion in an asymptotically stable
  environment.
\newblock {\em Annales de l'Institut Henri Poincare (B) Probability and
  Statistics}, 43(1):101 -- 138, 2007.

\bibitem{Singh2007}
A.~Singh.
\newblock A slow transient diffusion in a drifted stable potential.
\newblock {\em Journal of Theoretical Probability}, 20(2):153--166, 2007.

\bibitem{Singh}
A.~Singh.
\newblock Rates of convergence of a transient diffusion in a spectrally
  negative l\'evy potential.
\newblock {\em Ann. Probab.}, \textbf{36}:\ 279--318, 2008.

\bibitem{Taleb0}
M.~Taleb.
\newblock Large deviations for a {B}rownian motion in a drifted {B}rownian
  potential.
\newblock {\em Ann. Probab.}, \textbf{29}(3):\ 1173--1204, 2001.

\bibitem{talet2007}
M.~Talet.
\newblock Annealed tail estimates for a {B}rownian motion in a drifted
  {B}rownian potential.
\newblock {\em Ann. Probab.}, 35(1):32--67, 01 2007.

\bibitem{tanaka1993}
H.~Tanaka.
\newblock Recurrence of a diffusion process in a multidimensional {B}rownian
  environment.
\newblock {\em Proc. Japan Acad. Ser. A Math. Sci.}, 69(9):377--381, 1993.

\bibitem{Tanaka19971807}
H.~Tanaka.
\newblock Limit theorems for a {B}rownian motion with drift in a white noise
  environment.
\newblock {\em Chaos, Solitons and Fractals}, \textbf{8}(11):\ 1807--1816,
  1997.

\bibitem{foncexpovech}
G.~V\'echambre.
\newblock Exponential functionals of spectrally one-sided {L}\'evy processes
  conditioned to stay positive.
\newblock {\em to appear in Ann. Inst. H. Poincar\'e Probab. Statist.,
  Preprint, arxiv 1507.02949, version of February 2018}, 2015.

\bibitem{caslevyvech}
G.~V\'echambre.
\newblock Path decompostion of spectrally negative {L}\'evy processes, and
  application to the local time of a diffusion in those environments.
\newblock {\em Markov Processes Relat. Fields}, \textbf{24}:\ 563--668, 2018.

\bibitem{Zeitouni}
O.~Zeitouni.
\newblock Lectures notes on random walks in random environment.
\newblock {\em St Flour Summer School}, 2001.

\bibitem{Zindy2006UpperLO}
O~Zindy.
\newblock Upper limits of sinai's walk in random scenery.
\newblock {\em Stoch. Proc. Appl.}, \textbf{118}:\ 981--1003, 2008.

\end{thebibliography}

\end{document}